\documentclass[reqno,11pt]{amsart}

\usepackage{amssymb,amsmath,amsthm,amsfonts,xcolor} 

\usepackage{comment,graphicx}

\usepackage{enumitem}

\usepackage[colorlinks=true, pdfstartview=FitV, linkcolor=blue, citecolor=blue, urlcolor=blue]{hyperref}

\usepackage{color,setspace,wrapfig,multicol}

\usepackage[labelformat=empty]{subcaption}

\usepackage[mathscr]{eucal}

\usepackage{lipsum}

\usepackage[left=3cm,right=3cm,top=2.5cm,bottom=2.5cm]{geometry}

\allowdisplaybreaks
\numberwithin{equation}{section}

\newtheorem{theorem}{Theorem}[section]
\newtheorem{proposition}[theorem]{Proposition}
\newtheorem{lemma}[theorem]{Lemma}
\newtheorem{corollary}[theorem]{Corollary}

\theoremstyle{definition}
\newtheorem{definition}[theorem]{Definition}

\newtheorem{remark}[theorem]{Remark}
\newtheorem{example}[theorem]{Example}

\newcommand{\bA}{\mathbf{A}}
\newcommand{\be}{\mathbf{e}}

\newcommand{\cahnhofman}{\mathrm{CH}}
\newcommand{\cB}{\mathcal{B}}
\newcommand{\cH}{\mathcal{H}}
\newcommand{\cK}{\mathcal{K}}
\newcommand{\curve}{\gamma}
\newcommand{\degree}{m}
\newcommand{\dist}{\mathrm{dist}}
\renewcommand{\div}{\mathrm{div}}
\newcommand{\dualnorm}{\phi^o}
\newcommand{\exponent}{\sigma}
\newcommand{\F}{\omega}
\newcommand{\facet}{F}
\newcommand{\Frank}{B^{\dualnorm}}
\newcommand{\initialnetwork}{\sS^0}
\newcommand{\kappafi}{\kappa^\phi}
\newcommand{\Lip}{\mathrm{Lip}}
\newcommand{\maximaltime}{T^\dag}

\newcommand{\Nmin}{N^0}
\newcommand{\nufi}{\nu^{\phi^o}}
\newcommand{\p}{\partial}

\newcommand{\R}{\mathbb{R}}
\newcommand{\Rn}{\mathbb{R}^n}
\renewcommand{\S}{\mathbb{S}}
\newcommand{\sA}{\mathscr{A}}
\newcommand{\sB}{\mathscr{B}}
\newcommand{\sC}{\mathscr{C}}
\newcommand{\sG}{\mathscr{G}}
\newcommand{\sS}{\mathscr{S}}
\newcommand{\sT}{\mathscr{T}}
\newcommand{\tandiv}{\mathrm{div}_\tau}
\newcommand{\triplejunction}{X}
\newcommand{\Wulff}{B^\phi}
\newcommand{\Wball}{\mathring{B}^\phi}
\newcommand{\Z}{\mathbb{Z}}

\renewcommand{\bar}{\overline}

\newcommand{\Int}[1]{\mathrm{Int}(#1)}

\setlist[itemize]{leftmargin=6mm} 

\title[Crystalline hexagonal curvature flow of networks]{Crystalline hexagonal curvature flow of networks: short-time, long-time and self-similar evolutions}

\author[G. Bellettini] {Giovanni Bellettini} 
\address[G. Bellettini]{University of Siena, via Roma 56,  53100 Siena, Italy  \&   International Centre for Theoretical Physics (ICTP), Strada Costiera 11, 34151 Trieste (Italy)}
\email{bellettini@diism.unisi.it}

\author[Sh. Kholmatov] {Shokhrukh Yu. Kholmatov} 
\address[Sh. Kholmatov]{University of Vienna,  Oskar-Morgenstern Platz 1, 1090 Vienna 
(Austria)}
\email{shokhrukh.kholmatov@univie.ac.at}

\author[F. Almuratov] {Firdavsjon M. Almuratov} 
\address[F. Almuratov]{Jizzakh branch of the National University of Uzbekistan, Sh. Rashidov avenue 259, 130100 Jizzakh (Uzbekistan)}
\email{almurotov@gmail.com}



\date{\today}

\begin{document}
 

\begin{abstract}
We study the crystalline curvature flow of planar networks with a single hexagonal anisotropy. After proving the local existence of a classical solution for a rather large class of initial conditions, we classify the homothetically shrinking solutions having one bounded component. We also provide an example of network shrinking to a segment with multiplicity two.
\end{abstract}

\maketitle

\section{Introduction}

Crystalline evolution, more generally, \emph{geometric interface motions} in which surface tension acts  as a main driving force, model many processes in material sciences such as phase transformation, grain growth, crystal growth, ion beam and chemical etching {\it etc.}, and therefore, became the topic of many papers (see e.g. \cite{BBP:2001,BBBP:1997,Cahn:1991,CHT:1992,CK:1994,GGM:1998,Herring:1999,KL:2001,Mullins:1956,Taylor:1978,TCH:1992} and references therein). In the planar case, the interface is usually represented by a family of curves bounding different regions (phases, or grains) and moving in a nonequilibrium state \cite{CDR:1989,DKSch:1993,Herring:1951,Schl:2000}. 
In simplified models, the motion of these curves is described by a geometric equation relating, for instance, the normal velocity of the interface to its curvature. When the interface is represented by a single closed curve, i.e., the two phase case, such an evolution is usually called \emph{anisotropic curve shortening flow} (see e.g., \cite{Andrews:1998,AB:2011,Andrews:2002,GP:2022}). In presence of more than two phases in the plane, the interface is called a \emph{network}, and consists of a set of curves with multiple (typically triple)
junctions. 

The anisotropic curvature evolution of a network $\sS \subset \R^2$ is the formal gradient flow of the energy functional (the anisotropic length, or weighted $\phi$-length)
$$
\ell_\phi(\sS):=\int_{\sS} \phi^o (\nu_{\sS})~d\cH^1,
$$
where $\nu_{\sS}$ is a unit normal vector field to $\sS$ and the energy density $\phi^o$, sometimes called surface tension and initially defined on $\mathbb S^1$, is extended on $\R^2$ in a one-homogeneous way to a norm $\phi^o : \R^2\to[0,+\infty)$. This gradient flow is well-posed when $\phi^o$ is smooth and elliptic (for instance Euclidean) and $\sS$ is a finite union of sufficiently smooth curves with boundary, satisfying a suitable balance condition at triple junctions. In this case the network evolves, at least for short-times, by its anisotropic curvature in normal direction; furthermore, several qualitative properties and long time behaviour are known, see for instance \cite{BKh:2023,KNP:2021,MNP:2017,MNPS:2016,MNT:2004}. 

A challenging case is when $\phi^o$ is crystalline, i.e., its unit ball $B^{\phi^o}$ is a (centrally symmetric) polygon, hence with facets and corners. Here the phases are expected to be mostly polygonal, and to evolve under a sort of nonlocal (i.e., crystalline) curvature. A further mathematical obstruction to the study of long-time behaviour of the flow is the possible appearence of nonpolygonal curves arising from triple junctions during the evolution \cite{BCN:2006}. Even more difficult is the case when the curve $\Sigma_{ij}$ separating  phase $i$ and phase $j$ has its own anisotropy $\phi_{ij}^o,$ and the corresponding total length is the sum of all corresponding weighted $\phi_{ij}$-lengths $\ell_{\phi_{ij}}(\Sigma_{ij})$. When each $ \phi^o _{ij}$ is crystalline, this is a model for polycrystalline materials in the plane \cite{ELM:2021,GN:2000}; see also \cite{BKh:2023} for more.

In this paper we study short-time and long-time crystalline curvature flow of networks (Definition \ref{def:phi_curvature_flow}) with a single anisotropy whose Wulff shape $\Wulff$ (the dual body of the unit ball $\Frank$ of $\dualnorm$) is a regular hexagon with two horizontal facets (see Figure \ref{fig:wulffs}). Such an assumption on $\phi$ brings a lot of simplifications, and makes it possible an almost complete analysis, which would be too complicated and probably not available for a generic regular polygonal anisotropy.  It is worth to mention that most of the techniques developed here can be adapted to a rather general anisotropic setting.

Our main interests in the present paper are in short-time existence of the $\phi$-curvature flow, in singularity formation at the maximal time, and in homothetically shrinking solutions and, as a matter of fact, in an analysis of the conical critical points and conical local minimizers of $\ell_\phi$. These problems are nonlocal, and  the starting point is a rigorous definition of crystalline curvature of the network (i.e., the velocity of the flow), based on the notion of Lipschitz Cahn-Hoffman (CH) field (Section \ref{subsec:phi_regular_curves}) satisfying a balance condition at the multiple junctions (see \eqref{eq:balance_condition}); here we  mainly follow \cite{BCN:2006,BKh:2023}, where crystalline curvature is defined in a multi-anisotropic setting, and triple junctions might interact in the definition of crystalline curvature. 

Since the anisotropy is polygonal, we mostly restrict the flow to initial simple networks $\initialnetwork$-- polygonal networks whose segments/half-lines are parallel to some facet of the Wulff shape, and admit only triple junctions with a $120^o$-balance condition. Segments of $\initialnetwork$ are expected to evolve by parallel translation in normal direction, whereas the half-lines stay still. However, the notion of $\phi$-regular flow does not restrict to initial networks with the $120^o$-condition. For instance, any  critical network (Definition \ref{def:critical_network}) is a stationary solution, and there are many (even minimal) conical critical networks with three half-lines not satisfying the $120^o$-condition at the junctions (Theorem \ref{teo:phi_minimal_triods}). 

As observed in \cite{BCN:2006}, not all polygonal networks preserve their topology during the flow, and new segments or curves can arise at some time (in a continuous manner) from multiple junctions. To prevent such phenomena, which make difficult the description of the subsequent flow, since one looses the description via a system of ordinary differential equations, as in \cite{BCN:2006,BKh:2023} we need some topological assumptions on the initial network. In contrast to \cite{BCN:2006,BKh:2023} where polycrystalline networks consisting of three polygonal curves made by one segment and one half-line meeting at a single triple junction were considered, our simple networks admit an arbitrary finite number of triple junctions. 

Our main existence result reads as follows (see Theorem \ref{teo:short_time}).

\begin{theorem}\label{teo:existence_intro}
For any simple network $\initialnetwork,$ there exists the unique $\phi$-curvature flow $\{\sS(t)\}_{t\in [0,T^\dag)}$ starting from $\initialnetwork$ on a maximal time interval $[0, T^\dag)$. Moreover, if $T^\dag< +\infty$, then some segment of $\sS(t)$ vanishes as $t\nearrow T^\dag.$
\end{theorem}

As mentioned earlier, critical networks are examples of initial networks for which $T^\dag=+\infty.$ In Example \ref{exa:formation_of_two_triple_junctions} we provide a 
noncritical $\mathcal N^0$ for which $T^\dag=+\infty$. 

Simple networks admit the following remarkable property. If we partition a simple network into connected graphs by removing all simple (not triple) vertices, then for each graph $G$ containing at least one triple junction, \emph{either} a minimal CH field is constant along each segment/half-line of $G$ and coincides with some vertex of the Wulff shape, \emph{or} its values never coincide on $G$ with vertices of the Wulff shape, except at the removed simple vertices (see Lemma \ref{lem:graph_triple}). In particular, those graphs whose segments have a constant minimal CH field, do not evolve by translation, i.e., stays still. Thus, to prove Theorem \ref{teo:existence_intro} we just need to study the evolution of the heights from the remaining graphs. 

This observation can be generalized to networks with junctions of higher degree provided that the segments/half-lines forming those junctions have zero $\phi$-curvature (see Theorem \ref{teo:chalpak_flow}). Such higher degree junctions could appear as a singularity after a collapse of two (or more) triple junctions, and Theorem \ref{teo:chalpak_flow} sometimes allows us to restart the flow after singularities in a regular manner (the topology of the network now may change, see Corollary \ref{cor:restart_of_flow}). However, it is worth to mention that unlike the networks containing only triple junctions with the $120^o$-condition, the networks admitting higher degree junctions or triple junctions not satisfying the $120^o$-condition are not simple, and may not be reached generically, for instance, by weak solutions. 

The next question is, of course, the asymptotic behaviour of the flow at a singular time. In the Euclidean case the blow-up behaviour of the rescaled networks can be established by means of Huisken's monotonicity formula \cite{INS:2014,MNPS:2016}; then, using parabolic rescaling, one approaches some limiting network as $t\nearrow T$, which may admit various singularities such as the loose of the $120^\circ$ condition (collapse of triple junctions), collapse of some curve (higher multiplicity), or even collapse of a phase to a point or to a segment. To our best knowledge, at the moment there are no examples of networks reducing to a network with higher multiplicity (see the multiplicity-one conjecture in \cite{MNP:2022}).
In our crystalline setting, we can provide explicit examples leading to those phenomena (see Section \ref{sec:examples}), including higher multiplicity segments (Example \ref{exa:high_multiple_chuvaks}), and it is worth to mention that if after vanishing of some segments, the resulting network remains simple (not necessarily parallel to the initial one and  possibly with multiple junctions), then the flow restarts until a subsequent singularity is reached. 

\begin{figure}[htp!]
\includegraphics[width=\textwidth]{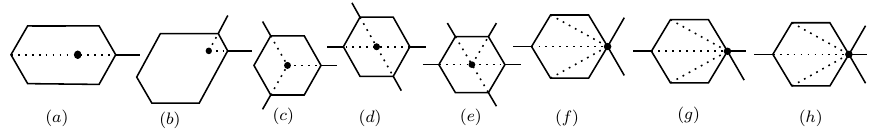}
\caption{All possible self-shrinkers with one bounded phase.}
\label{fig:self_shrinko}
\end{figure}

Our next main result  is a classification of self-shrinkers with a single bounded phase (see Section \ref{sec:homo_shrinkos} for more precise statements and the assumption on the topology of the initial network).

\begin{theorem}
Up to a rotation and mirror reflection, there are only eight different simple self-shrinkers possibly with multiple junctions (see Figure \ref{fig:self_shrinko} (a)-(h)).
\end{theorem}

Notice that we do not need a priori any symmetry assumption  (for instance, Figure \ref{fig:self_shrinko} (b) has no symmetry lines). Recall that such a classification was done in the Euclidean case in \cite{ChG:2007} where the authors,  under some symmetry assumptions, characterize six different self-shrinking networks having only one bounded phase.

As in \cite{BKh:2023}, we do not treat here weak (i.e., generalized) flows: for this broad argument we refer the reader to \cite{BChKh:2020,BKh:2018,LO:2016,Tonegawa:2019}. 

\subsection*{Acknowledgements.} 
The authors are grateful to Matteo Novaga for some useful remarks on minimality of triple and quadruple junctions. G. Bellettini acknowledges support from PRIN 2022PJ9EFL ``Geometric Measure Theory: Structure of Singular Measures, Regularity Theory and Applications in the Calculus of Variations'', and from GNAMPA (INdAM). Sh. Kholmatov acknowledges support from the Austrian Science Fund (FWF) Stand-Alone project P 33716. 

\newpage 

\section{Preliminaries} 

In this section we introduce the notation and the definitions used throughout the paper.
 
\begin{wrapfigure}[11]{l}{0.3\textwidth}
\vspace*{-7mm}
\includegraphics[width=0.28\textwidth]{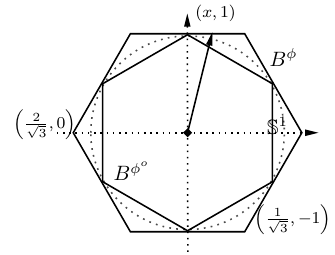} 
\caption{\small The Wulff shape $\Wulff$ of sidelength $\frac{2}{\sqrt3}$, and its dual $B^{\phi^o}.$}
\label{fig:wulffs}
\end{wrapfigure}
\subsection{Notation} 

Unless otherwise stated, all sets we consider are subsets of $\R^2.$ We choose the standard oriented basis $\{\be_1, \be_2\}$ of $\R^2$ and denote by $(x_1,x_2)$ the coordinates of $x\in\R^2$  with respect to this basis. $\Int{A}$ is the interior of $A\subset \R^2.$ $D_r(x)$ stands for the open (Euclidean) disc in $\R^2$ centered at $x\in\R^2$ of radius $r>0,$ and for shortness, set $D_r:=D_r(0)$. By $\cH^1$ we denote the one-dimensional Hausdorff measure in $\R^2.$

\subsection{Anisotropies} 

An (even) \emph{anisotropy} is a nonnegative, positively one-homogeneous even convex function $\phi$ in $\R^2$ satisfying $\{\phi=0\} = \{0\}$ (i.e., $\phi$ is a norm). In what follows, we fix the anisotropy $\phi$ in $\R^2$ whose closed unit ball (also called Wulff shape) $\Wulff:=\{\phi\le 1\}$ is the regular hexagon circumscribed to the unit circle centered at the origin, with two horizontal facets (see Figure \ref{fig:wulffs}). The closed unit ball $\Frank$ (also called Frank diagram) of the dual anisotropy 
$$
\phi^o(\xi):=\sup\limits_{\eta\in\R^2,\,\phi(\eta)=1}\,\,\xi\cdot\eta,\quad \xi\in\R^2,
$$
is also a regular hexagon inscribed to the unit circle with two vertical facets, as in Figure \ref{fig:wulffs}. The (six) facets of $\p \Wulff$ and of $\p \Frank$ are closed. 

By the definition of $\phi^o$ the following Young  inequality holds:
\begin{equation}\label{young_inequality}
\xi\cdot \eta \le \phi^o(\xi)\phi(\eta),\quad \xi,\eta\in \R^2. 
\end{equation}
We  write 
$$
\Wulff_R:=\{x\in\R^2:\,\,\phi(x)\le R\} 
\quad\text{and}\quad 
\Wball_R:=\{x\in\R^2:\,\,\phi(x)< R\},\quad R>0,
$$ 
with $\Wulff = \Wulff_1$.

\subsection{Curves} 

We call a closed set $\Gamma$ in $\R^2$ a \emph{curve}\footnote{We include unbounded curves without endpoints (case $I=(0,1)$) such as straight lines, parabolas, a union of two half-lines etc. meeting at one point, and unbounded curves with just one boundary point (case $I=[0,1)$) such as half-lines, half-parabolas etc., and finally, compact curves with two endpoints (possibly coinciding) such as segments, circles, arcs of circles etc.} if there exists an interval $I$ of the form $[0,1],$ $[0,1)$ or $(0,1),$ and an absolutely continuous function $\gamma:I\to\R^2$ such that $\gamma(I) = \Gamma.$ The function $\gamma$ is called a \emph{parametrization} of $\Gamma.$ In this paper we consider only \emph{embedded} curves, i.e., the map $\gamma:(0,1)\to\R^2$ is injective (and sometimes we identify the map $\gamma$ with the set $\Gamma$). When $I=[0,1]$ and $\gamma(0) = \gamma(1),$ we say $\Gamma$ is \emph{closed}. When $\gamma$ is $C^1$ (resp. Lipschitz) and $|\gamma'| > 0$ in $I$ (resp. a.e. in $I$), the map $\gamma$ is called a \emph{regular parametrization} of $\Gamma.$  A curve $\Gamma$ is $C^{k+\alpha}$ for some $k\ge0$ and $\alpha\in[0,1]$, $k+\alpha \geq 1$, if it admits a regular $C^{k+\alpha}$-parametrization. The tangent line to $\Gamma$ at a point $p\in \Gamma$ is denoted by $T_p\Gamma$ (provided it exists).
The (Euclidean) unit tangent vector to $\Gamma$ at $p$ is denoted by $\tau_\Gamma(p)$ and the unit normal vector is $\nu_\Gamma(p) = \tau_\Gamma(p)^\perp,$ where ${}^\perp$ is the counterclockwise $90^o$ rotation.  When there is no risk of confusion, we simply write $\tau$ and $\nu$ in place of $\tau_\Gamma$ and $\nu_\Gamma$. If $p= \gamma(x)$ and $\gamma$ is differentiable at $x$, then 
$$
\tau(p) = \frac{\gamma'(x)}{|\gamma'(x)|}
\qquad\text{and}\qquad 
\nu(p) = \frac{\gamma'(x)^\perp}{|\gamma'(x)|}.
$$
Unless otherwise stated, we choose tangent vectors in the direction of the parametrization. In particular, two oriented segments/half-lines are called \emph{parallel} provided they lie on parallel straight lines and their unit normals coincide.

A curve $\Gamma=\gamma(I)$ is \emph{polygonal} if for any $[a,b]\Subset I$ the curve $\gamma([a,b])$ is a finite union of segments. 
Any polygonal curve is a union of segments and at most two half-lines. A curve $\Gamma$ is (locally) \emph{rectifiable} if for any $[a,b]\Subset I$ the supremum
$$
\sup\limits_{a=t_0<t_1<\ldots < t_n = b}\quad \sum_{i=1}^n\,|\gamma(t_i) - \gamma(t_{i+1})| 
$$
is finite; equivalently, if and only if for any $[a,b]\Subset I$ and $\epsilon>0$ there exists a polygonal curve $\sigma : I \to \R^2$ such that 
$$
\sup_{x\in[a,b]}\,|\gamma(x) - \sigma(x)| <\epsilon.
$$
By definition any polygonal curve is rectifiable. Using \cite[Lemmas 3.2, 3.5]{Falconer:1985} one checks that a curve $\Gamma$ is rectifiable if for any $[a,b]\Subset I$ one has $\cH^1(\gamma([a,b])) <+ \infty$ and any rectifiable curve $\Gamma$ admits a unit tangent vector $\tau$ (and a corresponding unit normal $\nu$) $\cH^1$-a.e. along $\Gamma.$

The $\phi$-length of $\Gamma$ in an open set $\Omega\subset\R^2$ is defined\footnote{The $\phi$-length coincides with the Minkowski content of $\Gamma$ in $\Omega$, defined by means of the distance function induced by $\phi$.} as 
$$
\ell_\phi(\Gamma,\Omega):=\int_{\Omega\cap \Gamma} \phi^o(\nu)\,d\cH^1.
$$
When $\Omega=\R^2,$ we simply write 
$$
\ell_\phi(\Gamma):=\ell_\phi(\Gamma,\R^2).
$$

\subsection{Tangential divergence of a vector field} 

The \emph{tangential divergence} of a vector field $g\in C^1(\R^2; \R^2)$ over an embedded Lipschitz curve $\Gamma$ is defined as
$$
\tandiv g(p) = \nabla g(p)\tau(p)\cdot \tau(p)\quad \text{for $\cH^1$-a.e. $p\in\Gamma$}.
$$

The tangential divergence can also be introduced using parametrizations. More precisely, if $\gamma\in \Lip(I; \R^2)$ is a regular parametrization of $\Gamma$ and $g:\Gamma\to  \R^2$ is a Lipschitz vector field along $\Gamma$, i.e., $g\circ \gamma\in \Lip(I; \R^2)$, then 
\begin{equation*}
\tandiv\, g(p) = \frac{[g\circ \gamma]'(x)\cdot \gamma'(x)}{|\gamma'(x)|^2},\qquad p =\gamma(x)
\end{equation*}
at points of differentiability. One can readily check that the tangential divergence is independent of the parametrization.

\subsection{$\phi$-regular curves}\label{subsec:phi_regular_curves}

Let $\Gamma$ be a rectifiable curve, with $\cH^1$-almost everywhere defined unit normal $\nu(p).$ A vector field $N:\Gamma\to \p \Wulff$ is called a \emph{Cahn-Hoffman field} (CH field) if 
\begin{equation}\label{atirgulim}
N\cdot \nu = \phi^o(\nu)\,\, \quad \text{$\cH^1$-a.e. on $\Gamma,$}
\end{equation}
\begin{wrapfigure}[8]{l}{0.3\textwidth}
\vspace*{-2mm}
\includegraphics[width=0.28\textwidth]{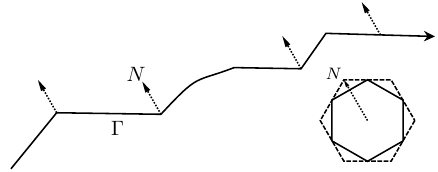}
\caption{\small A curve $\Gamma$ admitting a constant CH field $N$.} \label{fig:constant_cahn_hofman}
\end{wrapfigure}
namely $N\in \p\phi^o(\nu),$ where $\p$ stands for  the subdifferential.  Notice that reversing the orientation of the curve translates into a change of sign of $\nu$ and of the corresponding CH field, which is always ``co-directed'' as the unit normal. In what follows we shorthand $\nufi:= \frac{\nu}{\phi^o(\nu)}$.

\begin{definition}[\textbf{Lipschitz $\phi$-regular curve}]
We say the curve $\Gamma$ is \emph{Lipschitz $\phi$-regular} ($\phi$-\emph{regular}, for short) if it admits a \emph{Lipschitz} CH field.
\end{definition}

A typical example of $\phi$-regular curve is a polygonal curve with $120^o$-angle between adjacent segments/half-lines. However, $\phi$-regular curves need not be polygonal; for instance, the six arcs of unit circle in Figure \ref{fig:wulffs} having the same endpoints as the facets of $\p \Frank$ are $\phi$-regular.

\begin{proposition}\label{prop:characterization}
A rectifiable curve $\Gamma$ admits a constant CH field if and only if there is a facet $\facet \subset \p \Frank$ such that $\nufi(x)\in \facet$ for $\cH^1$-a.e. $x \in \Gamma$. 
\end{proposition}

\begin{proof}
Since $\Frank$ is a regular hexagon, the CH field on each facet  $\facet \subset\p \Frank$ is the unique closest vertex $N$ of $\Wulff$ (see Figures \ref{fig:wulffs} and \ref{fig:constant_cahn_hofman}).
Thus, if $\nu_\Gamma\in \facet$ $\cH^1$-a.e. on $\Gamma,$ then $N$ provides a constant CH field on $\Gamma.$ Conversely, suppose $\Gamma$ admits a constant CH field $N,$ i.e., $\phi(N) = 1$ and $N\cdot \nufi  = 1$ $\cH^1$-a.e. on $\Gamma.$ Thus, 
\begin{equation}\label{t6zre}
N\cdot [\nufi(x) -\nufi(y)] = 0\quad\text{for $\cH^1$-a.e. $x,y\in\Gamma$}.
\end{equation}
We have two cases: 
\smallskip 

{\it Case 1:} $\nufi(x)$ is constant. In this case $\Gamma$ is a straight line, and as $S$ we take any facet of $\p \Frank$ which contains $\nufi(x).$
\smallskip

{\it Case 2:} $\nufi(x)$ is not constant. By \eqref{t6zre} the difference $\nufi(x) - \nufi(y)$ lies on a straight line $L$ orthogonal to $N.$ If $L$ contains some facet $\facet$ of $\p \Frank,$ then by \eqref{t6zre} $\nufi(x)\in L\cap \p \Wulff = \facet$ for $\cH^1$-a.e. $x\in\Gamma.$ On the other hand, if $L$ intersects two facets $\facet_1$ and $\facet_2$ of $\p \Frank,$ then up to a $\cH^1$-negligible set, we can write $\Gamma=X_1\cup X_2$, where $\nu(x)\in \facet_i$ for any $x\in X_i.$ 
Clearly, $\nu(x) = L\cap \facet_i$ for any $x\in X_i,$ and by nonconstancy, $\nu(x)$ cannot belong to the intersection $\facet_1\cap \facet_2.$ Let $N_1$ and $N_2$ be the vertices of $\Wulff,$ closest to $\facet_1$ and $\facet_2.$ Then easy geometric arguments show that $N_i\cdot \nufi(x) = 1$ for any $x\in X_i$ and $Q\cdot \nufi(x)<1$ for $x\in X_i$ and $Q\in \p \Wulff\setminus\{N_i\}.$ This implies there is no $N\in\p \Wulff$ such that $\nufi(x)\cdot N = 1$ for $\cH^1$-a.e. $x\in \Gamma$, a contradiction.
\end{proof}

\noindent
From Proposition \ref{prop:characterization} it follows, in particular, that there exists $k\in\Z$ such that a $60^ok$-rotation of $\Gamma$ is the generalized\footnote{I.e., possibly with vertical parts.} graph of a monotone function. We mention also that Proposition \ref{prop:characterization} holds for any crystalline anisotropy whose Wulff-shape is a regular polygon with an even number of facets, provided one changes appropriately the $60^o$-rotation condition.

\begin{remark}\label{rem:ajoyib_tenglik}
We shall frequently use the following: 

\begin{itemize}
\item[(a)] if sides $[AB]$ and $[BC]$ of the triangle $ABC$ are parallel to adjacent facets of $\Wulff,$ then  
$$
\ell_\phi([AC]) = \ell_\phi([AB]) + \ell_\phi([BC]);
$$

\item[(b)] if $ABCD$ is a trapezoid\footnote{A convex quadrangle whose two opposite sides are parallel.} with sides parallel to facets of $\Wulff$ such that $[AB]$ and $[CD]$ are parallel and $\cH^1([AB])<\cH^1([CD]),$ then  
$$
\ell_\phi([CD]) = \ell_\phi([AB]) + \ell_\phi([BC]);
$$
\item[(c)] if $ABC$ is a regular triangle with sides parallel to three (nonadjacent) facets of $\Wulff,$ then for any $X\in[AB]$
$$
\ell_\phi([AB]) = \ell_\phi([BC]) = \ell_\phi([CA]) = \ell_\phi([CX]).
$$
\end{itemize} 
\end{remark}

\begin{lemma}[\textbf{Curves with constant CH field}]\label{lem:curves_contant_cahnhoffman}
Let $\Gamma$ be a Lipschitz curve admitting a constant CH field $N$, and let $X,Y\in\Gamma.$ Let $\Sigma\subset\Gamma$ and $S:=[XY]$ be respectively the arc of $\Gamma$ and the segment connecting $X$ and $Y$. Then 
$$
\ell_\phi(\Sigma) = \ell_\phi(S).
$$
Moreover, $N$ is a (constant) CH field also for $S.$
\end{lemma}

\begin{proof}
By the minimality of segments\footnote{A consequence of Jensen's inequality, see e.g. \cite{FFLM:2011}.} $\ell_\phi(\Sigma) \ge \ell_\phi(S)$. Let us prove the converse inequality. By assumption $\Sigma=\curve([0,1])$ for some $\curve\in \Lip([0,1];\R^2)$ with $|\curve'|>0$ a.e. on $[0,1].$ Let 
$$
J:=\{x\in[0,1]:\,\, \curve(x)\in S\}.
$$
Clearly, $J$ is a nonempty closed set so that $[0,1]\setminus J=\cup_j(a_jb_j)$ is an open set, $(a_j,b_j)$ its connected components. By the continuity of $\curve,$ $X_j:=\curve(a_j) $ and $Y_j:=\curve(b_j)$ belong to $S$ and $\Sigma$ does not intersect the (relative) interior of the segment $S_j:=[X_jY_j].$ Consider the bounded open set $C$ whose boundary consisting of the two rectifiable curves $\Sigma_j:=\curve([a_j,b_j])$ and $S_j.$ We may assume that the parametrization of $S_j$ is oriented from $X_j$ to $Y_j$ so that the outer unit normal $\nu_C$ of $C$ on $\Sigma_j$ coincides with $\nu_{\Sigma_j}$ and on $S_j$ with $-\nu_S.$ Applying the divergence theorem with the constant vector field $\xi:\R^2\to\R^2$, $\xi:=N$,  we get
\begin{align*}
0 = \int_C\div \xi \,d x = \int_{\Sigma_j} \nu_{\Sigma}
\cdot \xi d\cH^1 - \int_{S_j} \nu_S\cdot N d\cH^1.
\end{align*}
Hence by \eqref{atirgulim} and the Young inequality \eqref{young_inequality} we get 
\begin{equation}\label{ketshila}
\int_{\Sigma_j}\phi^o(\nu_\Sigma)d\cH^1 =  \int_{\Sigma} \nu_\Sigma\cdot \xi d\cH^1 =  \int_{S_j} \nu_S\cdot N d\cH^1 \le \int_{S_j}\phi^o(\nu_S)d\cH^1. 
\end{equation}
Therefore, 
$$
0\le \ell_\phi(\Sigma) - \ell_\phi(S) = \sum_j \Big(\int_{\Sigma_j}\phi^o(\nu_\Sigma)d\cH^1-\int_{S_j}\phi^o(\nu_S)d\cH^1\Big) \le 0.
$$
Hence, the inequality in \eqref{ketshila} is in fact an equality. Since both $N$ and $\nu_S$ are constant vector fields satisfying $\nu_S\cdot N = \phi^o(\nu_S),$ by definition $N$ is a CH field also for $S.$
\end{proof}

\subsection{Networks and polygonal networks}

An  \emph{ oriented network} (a network, for short) is a closed set $\sS \subset \R^2$ consisting of finitely many curves $\{\Gamma_i\}_{i=1}^M$  whose relative interiors are pairwise disjoint and the endpoints of each curve $\Gamma_i$ is also an endpoint of at least two other curves in case $\Gamma_i$ is not closed, or of another curve in case $\Gamma_i$ is closed. We call such an endpoint an $\degree$-multiple junction ($\degree$-junction, for short); $\degree$ is called the degree of the junction. The orientation of the network is given by the unit normals to each curve (defined via parametrization). 
Clearly, a network is a connected set. When all curves are polygonal with finitely many segments, $\sS$ is called  \emph{polygonal}, and the endpoints of half-lines and of segments of $\sS$ are called \emph{vertices} of $\sS$. A vertex is \emph{simple} if it is not a multiple junction. 

In the special case a polygonal network is a finite union of half-lines starting at the same point, it is called \emph{conical}. 
We write a polygonal network $\sS=\cup_i\Gamma_i$ frequently as a union $\sS=\cup_j S_j$ of its relatively closed segments/half-lines, where $S_j$ is a segment/half-line of a unique $\Gamma_i$ with  $\nu_{S_j} = \nu_{\Gamma_i}.$

The $\phi$-length of a network $\sS=\bigcup_{i=1}^M \Gamma_i$ in an open set $\Omega\subseteq\R^2$ is defined as 
$$
\ell_\phi(\sS,\Omega) = \sum\limits_{i=1}^M \ell_\phi(\Gamma_i,\Omega)
$$
(hence, possibly, $+\infty$).

\begin{definition}[\textbf{Admissible network}]\label{def:admissible_network}
A polygonal network $\sS=\cup_i S_i$ is \emph{admissible} if each segment/half-line is parallel to some facet of $\Wulff$, and the angle at any simple vertex  of $\sS$ is $120^o.$ 
\end{definition}

\begin{wrapfigure}[7]{l}{0.4\textwidth}
\vspace*{-2mm}
\includegraphics[width=0.38\textwidth]{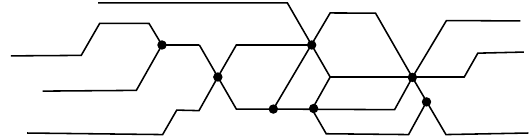}
\caption{\small An admissible network containing $\degree$-junctions for $\degree=3,4,5,6.$}
\label{fig:phi_reg_netw}
\end{wrapfigure}
Since $\Wulff$ is a regular hexagon, the degree of a junction of an admissible network is at most $6$ (see Figure \ref{fig:phi_reg_netw}). Definition \ref{def:admissible_network} is similar to the one in the two-phase case, where segments not parallel to facets of the Wulff shape are not considered, or to the multiphase case in \cite[Definition 4.10]{BCN:2006}, where (nonpolygonal) curves are excluded. For technical reasons, in \cite[Definition 4.10]{BCN:2006} admissible networks may contain only triple junctions and each curve of an admissible network should contain at least one segment. However, unlike \cite{BKh:2023}, in the present paper the half-lines of the network may end at a triple junction and admissible networks may contain multiple junctions. In particular, this includes Brakke-type spoons (a network consisting of the union of a closed curve and a half-line) and non-Lipschitz sets such as the union of two Wulff-shapes touching at a single point. 

\begin{definition}[\textbf{$\phi$-regular networks and CH fields}]
\label{def:phi_regular_networks_and_CH_fields}
An oriented network $\sS=\cup_i\Gamma_i$ is called \emph{Lipschitz $\phi$-regular} ($\phi$-\emph{regular}, for short) if every $\Gamma_i$ is $\phi$-regular, i.e., it admits a Lipschitz CH field $N_i,$ and if $X$ is a junction which is an endpoint of $\degree \ge3$-curves $\Gamma_{i_1},\ldots,\Gamma_{i_\degree}$ 
then the \emph{balance condition}
\begin{equation}\label{eq:balance_condition}
\sum\limits_{j=1}^\degree (-1)^{\exponent_j} N_{i_j}(X) = 0,
\end{equation}
holds, where $\exponent_j=0$ if $\Gamma_{i_j}$ is oriented from $X,$ and $\exponent_j=1$ otherwise (see Figure \ref{fig:admissible_netw}). 
\end{definition}
We call the map $N,$ defined as $N^{}_{\vert \Gamma_i}=N_i,$ a \emph{Cahn-Hoffman field} (shortly, a CH field) on $\sS.$ 

\begin{figure}[htp!]
\includegraphics[width=0.8\textwidth]{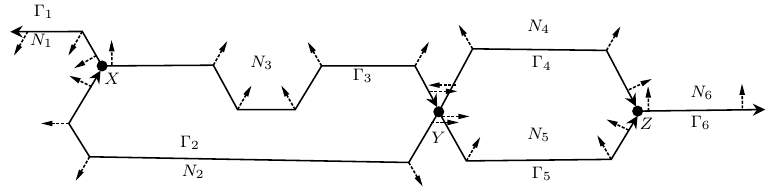}
\caption{\small 
A $\phi$-regular admissible network (with $m$-junctions, $m=3,4$) consisting of the union of six polygonal curves $\Gamma_i$ with a CH field. At the triple junction $\triplejunction$ we have $N_1 - N_2 + N_3=0,$ since $\Gamma_1$ and $\Gamma_3$ exit from $X$, while $\Gamma_2$ enters to $X.$ Thus, the balance condition \eqref{eq:balance_condition} holds with $\exponent_1=\exponent_2=0$ and $\exponent_3=1$. Similarly, $N_2-N_3 + N_4+N_5 = 0$ at the quadruple junction $Y$ and $-N_4-N_5+N_6=0$ at the triple junction $Z.$}
\label{fig:admissible_netw}
\end{figure}

\begin{wrapfigure}[8]{l}{0.35\textwidth}
\vspace*{-2mm}
\includegraphics[width=0.33\textwidth]{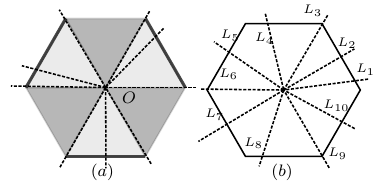}
\caption{}
\label{fig:conic_critic_network} 
\end{wrapfigure}

The collection of all CH fields over $\sS$ will be denoted by $\cahnhofman(\sS).$ In the polygonal $\phi$-regular case, when $\sS=\cup_iS_i$ and $N\in \cahnhofman(\sS),$ we abuse the notation $N_i:=N\big|_{S_i}.$ 

\begin{definition}[\textbf{Critical network}]\label{def:critical_network}
A $\phi$-regular network $\sS=\cup_i\Gamma_i$ is called a \emph{critical point} of the $\phi$-length, or shortly a \emph{critical network}, if it admits a CH field constant over each curve $\Gamma_i$ (the constant typically depends on $i$). 
\end{definition}
By definition, any network consisting of just one curve without boundary which admits a constant CH field, is critical. Next, consider a conical network $\sS$ consisting of $n\ge3$ half-lines starting from the same point, say the origin $O.$ When $n=3,$ $\sS$ is called a conical \emph{triod}. Unlike networks in the Euclidean setting, the non-strict convexity of $\Wulff$ allows several conical networks.

\begin{lemma}[\textbf{Conical critical networks}]\label{lem:conical_critical_networks}
A conical triod is critical if and only if its three half-lines intersect three non-adjacent facets of $\p \Wulff.$ More generally, a conical network with $n\ge3$ half-lines is critical if and only if there exist two integers $l_1,l_2\ge0$ such that its half-lines can be divided into $l_1$ pairwise disjoint groups of triplets and $l_2$ pairwise disjoint groups of doublets in a way that three half-lines in each triplet intersect three non-adjacent facets of $\p \Wulff$ and two half-lines in each doublet intersect two opposite facets of $\p \Wulff.$
\end{lemma}
\begin{wrapfigure}{r}{0.45\textwidth}
\vspace*{-2mm}
\includegraphics[width=0.43\textwidth]{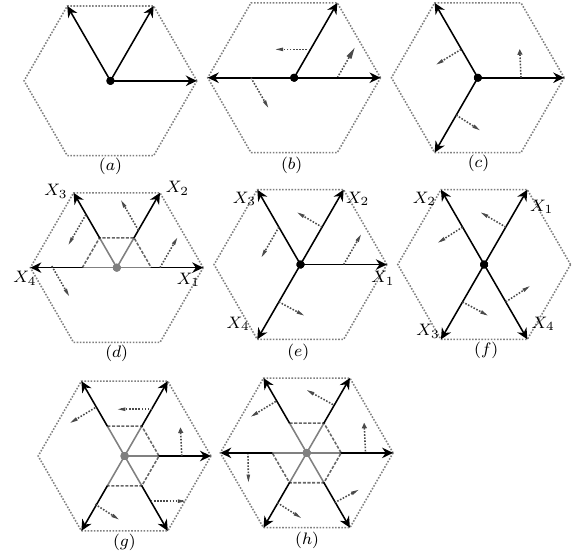}
\caption{\small Up to a rotation and a mirror reflection, there are exactly eight possible $\phi$-regular conical networks, the last seven being critical. Notice that in networks (d), (g) and (e) if we replace segments starting from the multiple junction with facets of a Wulff shape of sufficiently small radius, then the length of the network inside the larger Wulff shape strictly decreases.}
\label{fig:quad_junc}
\end{wrapfigure}
We omit the proof of this elementary lemma. In Figure \ref{fig:conic_critic_network} (a) three non-adjacent facets of $\p \Wulff$ are highlighted. Clearly, the remaining three facets are also non-adjacent. Notice that the conical network in Figure \ref{fig:conic_critic_network} (b), consisting of $10$ half-lines, is critical. Indeed, if we group the half-lines as $(L_1,L_5,L_9),$ $(L_2,L_7),$ $(L_3,L_8)$ and $(L_4,L_6,L_{10}),$ then these triplets and doublets satisfy the assertion of Lemma \ref{lem:conical_critical_networks} with $l_1=l_2=2.$

From Lemma \ref{lem:conical_critical_networks} and the symmetry of $\Wulff,$ up to a rotation by an integer multiple of $\pm60^o$ and a mirror reflection, there are eight possible admissible conical networks, seven of which are critical (see Figure \ref{fig:quad_junc}). Notice that the network in case (a) is not $\phi$-regular, because at the triple junction we cannot define any triple satisfying the balance condition \eqref{eq:balance_condition}.

\begin{example}\label{exa:four_networks}
Consider the conical $\phi$-regular networks $\sS$ with four half-lines $L_1,L_2,L_3,L_4$ crossing $\p \Wulff$ at vertices $X_1,X_2,X_3,X_4,$ respectively. As mentioned above, up to a rotation and a mirror reflection, $\sS$ can be only one of the networks drawn in Figure \ref{fig:quad_junc} (d)-(f). Since the half-lines are oriented from the quadruple junction, we can immediately check that a balance condition at quadruple junction for any admissible CH field $N$ over $\sS$ implies $N_1 + N_3 = 0$ and $N_2+N_4=0.$ 
One can readily check that: in case of ``$W$'' in Figure \ref{fig:quad_junc} (d) $N$ is uniquely defined by $N_1=-N_3=(\frac{1}{\sqrt3},1)$ and $N_2=-N_4=(-\frac{1}{\sqrt3},1)$; in case of ``$\Psi$'' in Figure \ref{fig:quad_junc} (e),  $N_1=-N_3=(\frac{1}{\sqrt3},1)$ but $N_2$ ($=-N_4$) can be arbitrarily chosen (still satisfying the constraints); in case of ``$X$'' in Figure \ref{fig:quad_junc} (f), both $N_1$ ($=-N_3$) and $N_2$ ($=-N_4$) can be arbitrarily chosen satisfying the constraints.
\end{example}

\subsection{$\phi$-minimal networks}
Let $\sS$ be a network and $\Omega \subseteq \R^2$ an open set; a network $\sA$ is called a compact perturbation of $\sS$ in $\Omega$ provided $\sS \Delta \sA \Subset \Omega$. 

\begin{definition}[\textbf{Local minimizers and minimal networks}]
A polygonal admissible network $\sS$ is called \emph{a local minimizer} of the $\phi$-length (shortly \emph{a local minimizer}) in an open set $U\subset\R^2$ if
$$
\ell_\phi(\sS,U) \le \ell_\phi(\sA,U)
$$
for any \emph{compact perturbation} $\sA$ of $\sS$ in $U$, i.e., for any network $\sA$ such that $\sS\Delta\sA\Subset U.$ If $\sS$ is a local minimizer in every bounded open subset of $\R^2,$ we call it \emph{$\phi$-minimal (shortly, minimal).} 
\end{definition}

Notice that to check the $\phi$-minimality of a network, it is enough to show its local minimality in every disc or every $\phi$-ball.

\begin{remark}
Compact perturbations of a network are still networks, in particular they are connected. However, they do not need to be polygonal; unlike minimal partitions in \cite{BChKh:2020,BKh:2018}, they \emph{need not preserve} the number of phases (or regions). 
\end{remark}
\newpage 
\begin{wrapfigure}[22]{l}{0.4\textwidth}
\includegraphics[width=0.38\textwidth]{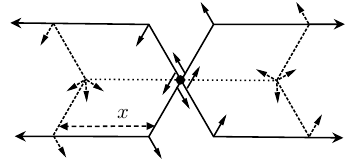}
\caption{\small A critical nonminimal network $\sS^0$ 
and its length-decreasing compact perturbation. This network is a local minimizer in every disc centered at the quadruple junction and not intersecting the half-lines. 
However, the dotted network $\mathcal N^0$ (obtained from $\sS^0$ with a ``large'' 
perturbation) has $\phi$-length strictly smaller than the one of $\sS^0$ (in addition, it satisfies the 
interior to the constraint condition in 
the sense of Definition \ref{def:interior_to_and_at_the_boundary_of_the_constraint}). 
Unlike $\sS^0$, the segments of 
$\mathcal N^0$ have nonzero $\Phi$-curvatures, and  during the flow they slide away from the quadruple junction. 
The evolution of $\sS^0$ and $\mathcal N^0$ 
will be considered in Example \ref{exa:formation_of_two_triple_junctions}.
}
\label{fig:non_minimal}
\end{wrapfigure}
Clearly, not every critical network is minimal (see Figure \ref{fig:quad_junc} (d), (g) and (h) and Figure \ref{fig:non_minimal}), however every minimal network is critical. 
Indeed, if $\Gamma_i$ does contain a small arc $\Gamma'$ with endpoints $X,Y\in \Gamma_i$ such that any $N\in \cahnhofman(\sS)$ is not constant over $\Gamma',$ then using a calibration argument as in Proposition \ref{prop:characterization}, one checks that the segment $[XY]$ has strictly less $\phi$-length than $\Gamma'.$ Thus, replacing $\Gamma'$ with $[XY]$ (since $\Gamma'$ is small, such a replacement still produces a network) we get a compact perturbation of $\sS$ which has strictly less $\phi$-length in any disc containing $\Gamma'.$ 

\begin{example}
Let $\sS$ consist a unique curve $\Gamma$  without boundary, admitting a constant CH vector field $N$. As we mentioned above, $\sS$ is critical. Let us show that $\sS$ is also minimal. Let $R>0$ and $\sA$ be any compact perturbation of $\sS$ in $D_R:= D_R(O)$, and let $X,Y\in \p D_R\cap \Gamma$ be such that the curve $\Gamma'\subset\Gamma$ connecting $X$ and $Y$ is the maximal, i.e., any other curve $\Gamma''$ with endpoint at $\p D_R$ satisfies $\Gamma''\subseteq \Gamma'.$ By Lemma \ref{lem:curves_contant_cahnhoffman} $\ell_\phi(\Gamma') =  \ell_\phi([XY]).$ As $X,Y\in \sA$ and $\sA$ is (arcwise) connected, there exists a curve $\Sigma'\subset\sA$ of $\sA$ connecting $X$ to $Y.$ By the anisotropic minimality of segments 
$$
\ell_\phi(\Sigma') \ge \ell_\phi([XY])=\ell_\phi(\Gamma').
$$
Since $\sA\setminus D_R = \sS\setminus D_R,$ for any $\bar R > R$ such that $\Gamma'\cup (\sA\Delta \sS)\Subset D_{\bar R}$ we have 
$$
\ell_\phi(\sA,D_{\bar R}) = \ell_\phi(\sA\setminus \Sigma',D_{\bar R})+\ell_\phi(\Sigma')\ge \ell_\phi(\sS\setminus \Gamma',D_{\bar R})+\ell_\phi(\Gamma') = \ell_\phi(\sS,D_{\bar R}).
$$
Thus, $\sS$ is minimal.
\end{example}

Next we study the minimality of some critical conical networks. 

\begin{theorem}[\textbf{Minimal conical triods}]
\label{teo:phi_minimal_triods}
Any conical critical triod is minimal. 
\end{theorem}

\begin{proof}
Let $\sS$ be a conical critical triod -- a network consisting of a union of three (different) half-lines $L_1,L_2,L_3$ starting at the origin $O$ and crossing three non-adjacent facets of $\Wulff.$

Let $R>0$ and $\sA$ be any compact perturbation of $\sS$ in $\Wball_R$ and let $X_i:=L_i\cap \p \Wulff_R.$ Since $\Wulff_R\cap \sA$ is connected there exists a point $T\in\sA$ and three curves $\Gamma_1,\Gamma_2,\Gamma_3 \subset \Wulff_R\cap \sA$ with disjoint relative interiors such that $\Gamma_i$ connects $X_i$ and $T$ so that $\Gamma_1\cup\Gamma_2\cup\Gamma_3$ form a partition of $\mathring{B}_R^\phi$ with the same boundary conditions as $\sS.$ 
For, take any curve $\Gamma'\subset \Wulff_R\cap \sA$ connecting $X_1$ to $X_3,$ and a curve $\Gamma''\subset \Wulff_R\cap \sA$ connecting $X_2$ to $\Gamma'$. Let $T$ be the first intersection of $\Gamma''$ with $\Gamma'$ so that its subcurve $\Gamma_2$ connecting $X_2$ to $\Gamma'$ (at $T$) is minimal. Notice that $T$ divides $\Gamma'$ into two subcurves $\Gamma_1$ and $\Gamma_3$ connecting $T$ to $X_1$ and $X_3.$
Let $\tilde \sA:=\Gamma_1\cup \Gamma_2\cup\Gamma_3\cup (\sS \setminus \mathring{B}_R^\phi).$ Let $N_1,N_2,N_3\in\p \Wulff$ be constant vectors such that $N_i\cdot \nu_i = \phi^o(\nu_i)$ on $L_i$ and $N_1+N_2+N_3 = 0$ at $O$. Define
$$
\xi_1 := N_3,\quad \xi_2 := 0,\quad \xi_3 := -N_1,
$$
so that $\xi_i-\xi_j = N_k$ for $ijk\in\{123,231,312\}$. These $\xi_i$ play the role of a (constant) paired calibration \cite[Theorem 3.2]{LM:1994}, and thus
$
\ell_\phi(\sS,\Wball_R) \le \ell_\phi(\tilde \sA,
\Wball_R) \le \ell_\phi(\sA, \Wball_R).
$
Hence, $\sS$ is minimal.
\end{proof}

\begin{wrapfigure}[8]{l}{0.35\textwidth}
\includegraphics[width=0.33\textwidth]{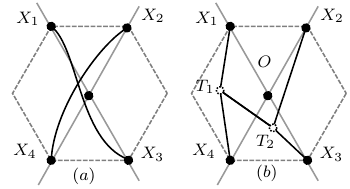}
\caption{\small }
\label{fig:persey}
\end{wrapfigure}

In the case of four half-lines we use a different method (see Remark \ref{rem:not_calibrated_eshshak}).

\begin{theorem}[\textbf{Minimality of $4$-junctions}]
Let $\sS$ be a conical network consisting of four half-lines starting from the origin $O$, parallel to facets of $\Wulff$ and not lying in a half-plane. Then $\sS$ is minimal.
\end{theorem}

\begin{proof}
Up to a rotation we have only two possibilities, see Figure \ref{fig:quad_junc} (e) and (f). Since the ideas are the same, we only prove the minimality of $\sS$ in Figure \ref{fig:quad_junc} (f). Fix $R>0$, take a compact perturbation $\sA$ of $\sS$ in $\Wball_R$ and let $X_1,X_2,X_3,X_4\in\p \Wulff_R$ be the intersection points of the half-lines of $\sS$ with  $\p \Wulff_R.$ Since $\Wulff_R \cap \sA$ is connected, there exist curves $\Gamma_1,\Gamma_2\subset \Wulff_R \cap \sA$ connecting $X_1$ with $X_3$ and $X_2 $ with $X_4.$ Notice that if $\cH^1(\Gamma_1\cap \Gamma_2)=0$ (see Figure \ref{fig:persey} (a)), then using the additivity of the $\phi$-length and the minimality of segments we get
\begin{align*}
\ell_\phi(\sA,\Wball_R) \ge \ell_\phi(\Gamma_1\cup \Gamma_2)  = \ell_\phi(\Gamma_1) + \ell_\phi(\Gamma_2) 
\ge   \ell_\phi([X_1X_3]) + \ell_\phi([X_2X_4]) = \ell_\phi(\sS,\Wball_R).
\end{align*}
Hence, we may assume $\cH^1(\Gamma_1\cap \Gamma_2)>0$ (see Figure \ref{fig:persey} (b)). 
Let $T_1,T_2\in \Gamma_1\cap \Gamma_2$ be two points such that the subcurves $\Gamma_1'\subset\Gamma_1$ and $\Gamma_2'\subset\Gamma_2$ connecting $T_1$ and $T_2$ are maximal. Notice that $T_1$ and $T_2$ may coincide. These two points divide $\Gamma_1$ and $\Gamma_2$ into three parts: $\Gamma_1$ is divided into  $(X_1,T_1),$ $(T_1,T_2)$ and $(T_2,X_3),$ and $\Gamma_2$ is divided into $(X_2,T_1),$ $(T_1,T_2)$ and $(T_2,X_4).$ Notice that by maximality the subcurves ending at $X_i$ have disjoint relative interiors.

By the minimality of segments and assumption $\cH^1(\Gamma_1\cap \Gamma_2)>0,$ we may replace those subcurves with segments with the same endpoints so that $T_1,T_2$ are two triple junctions of segments so that $T_1\ne T_2$ (see Figures \ref{fig:persey} (b) and \ref{fig:four_quado}); notice that unlike the subcurves, interiors of those segments may intersect.

Assume first $T_1$ is a triple junction of segments $[X_1T_1],$ $[X_4T_1]$ and $[T_1T_2]$, and $T_2$ is a triple junction of segments $[X_2T_2],$ $[X_3T_2]$ and $[T_1T_2]$ as in Figure \ref{fig:persey} (b). In this case by the minimality of segments 
\begin{equation}\label{ujzaert}
\ell_\phi(\sA,\Wball_R) \ge \ell_\phi([X_1T_1X_4]) + \ell_\phi([X_2T_2X_3]) \ge \ell_\phi([X_1X_4]) + \ell_\phi([X_2X_3]).  
\end{equation}
Since both sides of the triangles $X_1OX_4$ and $X_2OX_3$ ending at $O$ are parallel to adjacent facets of $\Wulff,$ by Remark \ref{rem:ajoyib_tenglik} (a) we have 
$$
\ell_\phi([X_1X_4]) = \ell_\phi([OX_1])+\ell_\phi([OX_4])
\quad \text{and}\quad 
\ell_\phi([X_2X_3]) = \ell_\phi([OX_2])+\ell_\phi([OX_3]).
$$
Thus, placing these equalities into \eqref{ujzaert} we get
$$
\ell_\phi(\sA,\Wball_R) 
\ge \sum\limits_{i=1}^4 \ell_\phi([OX_i]) = \ell_\phi(\sS,\Wball_R).
$$

\begin{figure}[htp!]
\includegraphics[width=\textwidth]{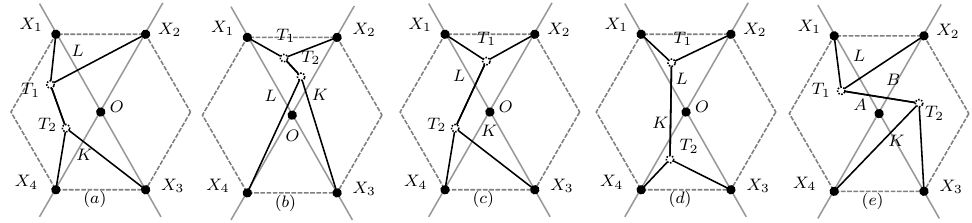}
\caption{\small Some possible locations of $T_1$ and $T_2.$}
\label{fig:four_quado}
\end{figure}

Now we turn to a more involved case, where $T_1$ is a triple junction of segments $[X_1T_1],$ $[X_2T_1]$ and $[T_1T_2]$ and $T_2$ is a triple junction of $[X_4T_2],$ $[X_3T_2]$ and $[T_1T_2].$ 
According to the location $T_1$ and $T_2$, as well as to the symmetry of $\sS,$ we have the following five possibilities.
\smallskip

{\it Case 1:} $T_1$ and $T_2$ belong to the left parallelogram with three vertices at $X_1,$ $O$ and $X_4,$ see Figure \ref{fig:four_quado} (a). Let $\{L\}:=[OX_1]\cap [T_1X_2]$ and $\{K\}:=[OX_4]\cap [T_2X_3].$ By the minimality of segments and Remark \ref{rem:ajoyib_tenglik} (a)
$$
\ell_\phi([X_1T_1T_2X_4]) \ge \ell_\phi([X_1X_4]) = \ell_\phi([X_1OX_4]) = \ell_\phi([X_1O]) + \ell_\phi([OX_4]),
$$
where $[ABC\ldots]$ is the polygonal curve made of segments $[AB],[BC],\ldots$. On the other hand, by Remark \ref{rem:ajoyib_tenglik} (c) and the monotonicity of the $\phi$-length,  
$$
\ell_\phi([T_1X_2]) \ge \ell_\phi([LX_2]) = \ell_\phi([OX_2]),
\quad 
\ell_\phi([T_2X_3]) \ge \ell_\phi([KX_3]) = \ell_\phi([OX_3]).
$$ 
Therefore,
$$
\ell_\phi(\sA,\Wball_R) \ge \ell_\phi([X_1T_1T_2X_4]) +\ell_\phi([T_1X_2]) + \ell_\phi([T_2X_3]) \ge \sum\limits_{i=1}^4 \ell_\phi([OX_i]) = \ell_\phi(\sS,\Wball_R).
$$

{\it Case 2:} $T_1$ and $T_2$ lie in the upper triangle $X_1OX_2,$ see Figure \ref{fig:four_quado} (b). Let $\{L\}: = [T_2X_4]\cap [X_1X_3]$ and $\{K\}: = [T_2X_3]\cap [X_4X_2].$ By Remark \ref{rem:ajoyib_tenglik} (a) applied with the triangles $T_2KO$ and $KOX_4$,
$$
\ell_\phi([T_2X_4]) = \ell_\phi([T_2L]) + \ell_\phi([LOX_4]) = \ell_\phi([T_2LO]) + \ell_\phi([OX_4]).
$$
Similarly, 
$$
\ell_\phi([T_2X_3]) = \ell_\phi([T_2K]) + \ell_\phi([KOX_3]) = \ell_\phi(T_2KO) + \ell_\phi(OX_3).
$$
Next, by the minimality of segments, 
$
\ell_\phi([T_1T_2LO]) \ge \ell_\phi([T_1O]).
$
Thus,
\begin{align*}
\ell_\phi([T_2X_4]) + \ell_\phi([T_2X_3]) + \ell_\phi([T_1T_2]) \ge & \ell_\phi(T_1T_2LO) + \ell_\phi(OX_4)+\ell_\phi(OX_3)\\
\ge & \ell_\phi(T_1O) + \ell_\phi(OX_4)+\ell_\phi(OX_3).
\end{align*}
Consider the network $\sB$ consisting of segments $[X_1T_1],$ $[X_2T_1],$ $[T_1O]$ and $[OX_4].$ 
Clearly, it is a compact perturbation of the conical triod $\sC$ consisting of the three half-lines starting from the origin and passing through $X_1,$ $X_2$ and $X_4,$ respectively. By Proposition \ref{teo:phi_minimal_triods} $\sC$ is minimal, and thus, 
$
\ell_\phi(\sB,\Wball_R) \ge \ell_\phi(\sC,\Wball_R).
$
Equivalently,
$$
\ell_\phi([X_1T_1]) + \ell_\phi([X_2T_1]) + \ell_\phi([T_1O]) \ge \ell_\phi([OX_1]) + \ell_\phi([OX_2]).
$$
Then 
\begin{align*}
\ell_\phi(\sA,\Wball_R) \ge & \ell_\phi([T_2X_4]) + \ell_\phi([T_2X_3]) + 
\ell_\phi([X_1T_1]) + \ell_\phi([X_2T_1]) + \ell_\phi([T_1T_2]) \\
\ge  & \ell_\phi([X_1T_1])
+ \ell_\phi([X_2T_1]) + \ell_\phi(T_1O)
+ \ell_\phi(OX_4) + \ell_\phi(OX_3) \\
\ge & \sum\limits_{i=1}^4\ell_\phi([OX_i])  = \ell_\phi(\sS,\Wball_R).
\end{align*}

{\it Case 3:} $T_1$ lies in the upper triangle $X_1OX_2$ and $T_2$ lies in the left parallelogram with three vertices $X_1,O,X_4$, Figure \ref{fig:four_quado} (c). Let $\{L\}:=[OX_1]\cap [T_1T_2]$ and $\{K\}:=[OX_4]\cap [T_2X_3].$ Then 
$$
\ell_\phi([T_1T_2X_4]) \ge \ell_\phi([T_1L]) + \ell_\phi([LX_4]) = \ell_\phi([T_1LO]) + \ell_\phi([OX_4]) \ge \ell_\phi([T_1O])+\ell_\phi([OX_4]),
$$
and
$$
\ell_\phi([T_2X_3]) \ge \ell_\phi([KX_3])=\ell_\phi([OX_3]).
$$ 
Moreover, as in case 2 
$$
\ell_\phi([X_1T_1]) + \ell_\phi([X_2T_1]) + \ell_\phi([T_1O]) \ge \ell_\phi([OX_1]) + \ell_\phi([OX_2]).
$$
Summing these  inequalities we get 
$$
\ell_\phi(\sA,\Wball_R) \ge \sum\limits_{i=1}^4\ell_\phi([OX_i]) = \ell_\phi(\sS,\Wball_R).
$$

{\it Case 4:} $T_1$ lies in the upper triangle  $X_1OX_2$ and $T_2$ lies in the lower triangle $X_4OX_3$, Figure \ref{fig:four_quado} (d). Let $L\in [X_1OX_2]\cap [T_1T_2]$ and $K\in [X_4OX_3]\cap [T_1T_2].$ By Remark \ref{rem:ajoyib_tenglik} (a) 
$$
\ell_\phi([T_1T_2]) = \ell_\phi([T_1LOKT_2]) \ge \ell_\phi([T_1O]) + \ell_\phi([T_2O]).
$$
Now applying the minimality of conical critical triods as in case 2 we get 
$$
\ell_\phi([X_1T_1]) + \ell_\phi([X_2T_1]) + \ell_\phi([OT_1]) \ge \ell_\phi([OX_1]) + \ell_\phi([OX_2]) 
$$
and 
$$
\ell_\phi([X_4T_2]) + \ell_\phi([X_3T_2]) + \ell_\phi([OT_2]) \ge \ell_\phi([OX_3]) + \ell_\phi([OX_4]). 
$$
Summing these inequalities we get $\ell_\phi(\sA,\Wball_R) \ge \ell_\phi(\sS,\Wball_R).$
\smallskip

{\it Case 5:} $T_1$ lies in the left parallelogram and $T_2$ lies in the right parallelogram, Figure \ref{fig:four_quado} (e). Let $A\in [T_1T_2]\cap [X_1X_3],$ $B\in [T_1T_2]\cap [X_2X_4],$ $L\in [OX_1]\cap [T_1X_2]$ and $K\in [OX_3]\cap [T_2X_4].$ Then 
$$
\ell_\phi([T_1X_2]) \ge \ell_\phi([LX_2]) = \ell_\phi([OX_2]),\quad \ell_\phi([T_2X_4]) \ge \ell_\phi([KX_4]) = \ell_\phi([OX_4])
$$
and
$$
\ell_\phi([X_1T_1T_2X_3]) \ge\ell_\phi([X_1X_3]).
$$
Summing these inequalities we get $\ell_\phi(\sA,\Wball_R) \ge \ell_\phi(\sS,\Wball_R).$
\end{proof}

\begin{remark}\label{rem:not_calibrated_eshshak}
In the case of a quadruple junction we cannot produce a paired calibration made of constant fields. Indeed, let half-lines $L_1,\ldots,L_4$ start from the origin and parallel to the facets of $\p \Wulff$ as in Figure \ref{fig:quad_junc} (e)-(f), and let $N\in\cahnhofman(\sS).$
As observed in Example \ref{exa:four_networks}, at quadruple junction we have $N_1+N_3=0=N_2+N_4$ and additionally, $N_1=-N_3=(-\frac{2}{\sqrt3},0)$ for Figure \ref{fig:quad_junc} (e). 
Suppose we can choose vectors $\xi_i\in\R^2$, $i=1,\dots,4$, such that $\phi(\xi_i-\xi_j)\le 1$ for all $i,j$ and
$
N_i=\xi_i-\xi_{i+1},
$
with $\xi_{5} = \xi_1.$ In case of ``X'' (Figure \ref{fig:quad_junc} (f)), $N_1$ and $N_2$ lie in the facets of $\Wulff$ with common vertex $(-\frac{2}{\sqrt3},0),$ and thus, $N_1+N_2\in \Wulff$ if and only if $N_1=(-\frac{1}{\sqrt3},-1)$ and $N_2=(-\frac{1}{\sqrt3},1).$ However, in this case, $N_1+N_4 = (0,2)\notin \Wulff,$ i.e., $\phi(\xi_4-\xi_2) = \phi(N_1+N_4)>1.$ A similar reasoning applies in  case of ``$\Psi$'' (Figure \ref{fig:quad_junc} (e)). 
\end{remark}

Applying Remark \ref{rem:ajoyib_tenglik} (a) and (c) we immediately find that conical minimal networks $\sS$ with a quadruple junction are not ``isolated'': in fact, in an arbitrarily small neighborhood\footnote{Differently from the network in Figure \ref{fig:non_minimal}, where the perturbation is not ``small''.} of the quadruple junction, we can find a compact perturbation $\sA$ of $\sS$ having the same $\phi$-length as $\sS$ (see Figure \ref{fig:ostanocish}).
\begin{figure}[htp!]
\includegraphics[width=0.48\textwidth]{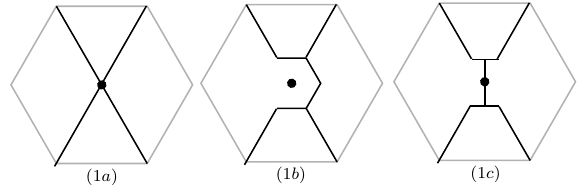}
\includegraphics[width=0.48\textwidth]{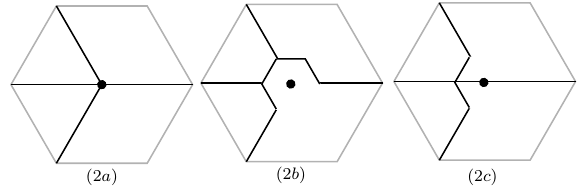}
\caption{\small }
\label{fig:ostanocish}
\end{figure}
Similar perturbations can be done for a $\phi$-regular conical minimal triods $\sS$ not satisfying  the $120^o$-condition. However, for a $\phi$-regular conical triod with the $120^o$-condition the following condition holds: if $\sA$ is any compact perturbation of $\sS$ in $D_R$ with a triple junction different from $\sS,$ then $\ell_\phi(\sS,D_R) < \ell_\phi(\sA,D_R).$

\subsection{$\phi$-curvature of an admissible  network}
In what follows we work only with polygonal admissible networks. As we have seen after Definition \ref{def:phi_regular_networks_and_CH_fields}, conical triods as in Figure \ref{fig:quad_junc} (a) do not admit a CH field (no balance condition at the triple junction), and hence, not every admissible network is $\phi$-regular. The following lemma shows that without such ``noncritical'' triple junctions any admissible network is $\phi$-regular. 

\begin{theorem}[\textbf{Existence of a minimal CH field}]
\label{teo:existence_of_a_minimal_CH_field}
Let $\sS = \cup_{i=1}^n S_i$ be a polygonal admissible network whose triple junctions are locally contained in conical critical triods. Then $\sS$ is $\phi$-regular. Moreover, there exists a minimizer $\Nmin$ of the problem
\begin{equation}\label{min_problem91}
\min\limits_{N\in\cahnhofman(\sS)} \quad \sum\limits_{i=1}^n
\int_{S_i} 
\Big[
{\tandiv N}^{}_{\vert_{S_i}}
\Big]^2 \phi^o(\nu_{S_i}) 
\,d\cH^1,
\end{equation}
and for any $S_i$ the number ${\tandiv \Nmin}^{}_{\vert_{S_i}}$ does not depend on $\Nmin$.
\end{theorem}

\begin{proof}
As in \cite[Remark 4.7]{BKh:2023}, it is enough to determine a CH field $N$ at the endpoints of each $S_i$ and then linearly interpolate it inside $S_i.$ Therefore, given a segment/half-line $S_i,$ we first define a vector field $N_i:=N^{}_{\vert S_i}$ with $N_i\in \p \dualnorm(\nu_{S_i})$ at the endpoints of $S_i.$

Take any vertex $X$ of $\sS.$ Assume that $X$ is a common endpoint of only two segments/half-lines, say  $S_1$ and $S_2.$ The definition is standard: since the angle between $S_1$ and $S_2$ is $120^o,$ we define $N_1(X) = N_2(X)= V$ resp. $N_1(X) = N_2(X) =-V,$ where $V$ is the vertex of $\Wulff$ whose adjacent facets have outer resp. inner unit normals $\nu_{S_1}$ and $\nu_{S_2}.$ 

Next, suppose $\triplejunction$ is a triple junction, say a common endpoint of three segments/half-lines $S_1,S_2$ and $S_3.$ Recall that, up to a rotation and a mirror reflection, we have only two critical triple junctions (see Figure \ref{fig:quad_junc} (b) and (c)). In either case we can choose three vectors $N_1(\triplejunction),$ 
$N_2(\triplejunction),$ $N_3(\triplejunction)$ satisfying $N_i(\triplejunction) \in \p\dualnorm(\nu_{S_i})$ such that 
$$
\sum\limits_{i=1}^3 (-1)^{\exponent_i} N_i(\triplejunction) = 0,
$$
where $\exponent_i = 0$ if $S_i$ oriented from $\triplejunction,$ and $\exponent_i=1$ otherwise. 

Similarly, if $X$ is an $\degree$-junction of, say, $S_1,\ldots, S_\degree$ for some $\degree=4,5,6$, then the vectors $N_i(X)\in \p\dualnorm(\nu_{S_i})$ and numbers $\exponent_i\in\{0,1\}$ for $i=1,\ldots,m$ with 
$$
\sum\limits_{i=1}^\degree (-1)^{\exponent_i} N_i(X) = 0
$$
can be defined, for instance, as in Figure \ref{fig:quad_junc} (d)-(h). We omit the details. 

Now we extend $N_i$ to the relative interior of $S_i.$ If $S_i$ is a half-line starting from a vertex $X$ of $\sS,$ we define 
$$
N_i(Z):=N_i(X)\quad \text{for all $Z\in S_i.$}
$$
If $S_i=[XY]$ is a segment, then we extend $N_i$ linearly inside $S_i$ as 
$$
N_i(Z):= N_i(X) + \lambda [N_i(Y) - N_i(X)],\quad Z:=X + \cH^1(S_i) 
\lambda \tau_{S_i},\quad \lambda\in[0,1],
$$
where $\tau_{S_i}$ is the tangent of $S_i$ -- the clockwise $90^o$-rotation of $\nu_{S_i}.$ Now defining $N^{}_{\vert_{S_i}} = N_i,$ we get $N\in \cahnhofman(\sS)$ and 
\begin{equation}\label{u_nego_bomba}
{\tandiv N}^{}_{\vert_{S_i}}
= 
\begin{cases}
0 & \text{if $S_i$ is a half-line,}\\[2mm]
\frac{[N_i(Y)- N_i(X)]\cdot \tau_{S_i}}{\cH^1(S_i)} & \text{if $S_i=[XY]$ is a segment.}
\end{cases}
\end{equation}
Thus, $\cahnhofman(\sS)$ is nonempty and so $\sS$ is $\phi$-regular.

Finally, we prove that the minimum problem \eqref{min_problem91} admits a solution. Let $(N^k)\subset\cahnhofman(\sS)$ be a minimizing sequence and consider the sequence $(N_i^k(X))_k$ at each endpoint $X$ of a segment/half-line $S_i$ in $\sS.$  Since $\p \Wulff$ is compact and the number of vertices of $\sS$ is finite, up to a (not relabelled) subsequence, $N_i^k(X)\to N_i(X)$ for some $N_i(X)\in \p \Wulff.$ Clearly, $N_i^k(X) = N_i(X)$ if $X$ is not a multiple junction, and letting $k\to+\infty$ in the equalities
$$
\sum\limits_{i} (-1)^{\exponent_i} N_i^k(X) = 0\quad\text{and}\quad N_i^k(X)\cdot \nu_{S_i} = \phi^o(\nu_{S_i}) 
$$
we obtain
$$
\sum\limits_{i} (-1)^{\exponent_i} N_i(X) = 0\quad\text{and}\quad N_i(X)\cdot \nu_{S_i} = \phi^o(\nu_{S_i})
$$
at multiple junctions. As above,  let us extend $N_i$ in the (relative) interior of $S_i,$ linearly interpolating its values at the endpoints, and denote by $\Nmin$ such an extension. Then denoting by $M$ the left-hand-side of \eqref{min_problem91} we get
\begin{align*}
M = & \lim\limits_{k\to+\infty}\,  \sum\limits_{i=1}^n \int_{S_i} 
\big[\tandiv N_i^k\big]^2 \phi^o(\nu_{S_i}) 
\,d\cH^1 
\\
\ge & \lim\limits_{k\to+\infty}\,  
\sum\limits_{\text{$S_i=[X_iY_i]$ segment}}  \phi^o(\nu_{S_i}) 
\cH^1(S_i)  \Big[\tfrac{1}{\cH^1(S_i)}\int_{S_i} \tandiv N_i^k\,d\cH^1 \Big]^2 
\\
= & 
\lim\limits_{k\to+\infty}\,\sum\limits_{\text{$S_i=[X_iY_i]$ segment}}  \phi^o(\nu_{S_i}) \tfrac{|N_i^k(Y_i) - N_i^k(X_i)|^2}{\cH^1(S_i)}\\ 
= & \sum\limits_{\text{$S_i=[X_iY_i]$ segment}} \phi^o(\nu_{S_i}) \tfrac{|N_i^0(X_i) - N_i^0(Y_i)|^2}{\cH^1(S_i)} \\
= & \sum\limits_{i=1}^n\int_{S_i} 
\big[\tandiv \Nmin\big]^2 \phi^o(\nu_{S_i}) 
\,d\cH^1,
\end{align*}
where in the first inequality we used the Jensen inequality, in the second equality the definition of tangential divergence and the fundamental theorem of calculus and in the last equality we used \eqref{u_nego_bomba}. This implies $\Nmin\in\cahnhofman(\sS)$ is a minimizer. 

The independence of $\tandiv \Nmin$ on $\Nmin$ follows from the strict convexity of the functional in \eqref{min_problem91} in the tangential divergence.
\end{proof}

\noindent
We call any minimizer $N$ of \eqref{min_problem91} a  \emph{minimal CH field} (of $\sS$).

\begin{remark}\label{rem:prop_minimal_chn}
Let $\sS=\cup_iS_i$ be as in Theorem \ref{teo:existence_of_a_minimal_CH_field} and $\Nmin$ be any minimal CH field. Then: 

\begin{itemize}
\item $\Nmin$ is uniquely defined at the simple vertices of $\sS$: if the simple vertex $X$ is a common endpoint of $S_i$ and $S_j,$ then $\Nmin(X)$ is defined as $V$ resp. $-V,$ where $V$ is the vertex of $\Wulff$ whose adjacent facets are parallel to $S_i$ and $S_j,$ and $\nu_{S_i}$ and $\nu_{S_j}$ are the outer resp. inner unit normals to $\Wulff$;

\item $\Nmin_i:=N_{\vert S_i}^0$ is constant on half-lines of $\sS$;

\item If $S_i$ is a segment, then $N_{\vert S_i}^0$ is linear on $S_i,$ and $\Nmin$ solves the minimum problem 
\begin{equation}\label{shapat_minimi}
\inf_{N\in\cahnhofman(\sS)}\,\sum\limits_{\text{$S_i=[X_iY_i]$ segment}} \phi^o(\nu_{S_i}) \tfrac{| N_i(Y_i) - N_i(X_i)|^2}{\cH^1(S_i)};
\end{equation}

\item If, as in \cite{BKh:2023},  we call the number 
$$
\kappafi_{S_i}:=\tandiv \Nmin_i
$$
the \emph{$\phi$-curvature} (or crystalline curvature) of $S_i$, then $\kappafi_{S_i}\nu_{S_i}$ is independent of the orientation of $S_i$ and the $\phi$-curvature of any half-line of $\sS$ is $0$. Furthermore, as in the two-phase case \cite{GP:2022}, if the network forms locally a convex set around the segment $S_i:=[X_iY_i],$ not ending at a multiple junction, and $\nu_{S_i}$ is directed ``inward'' resp. ``outward'' to that set, then 
$$
\kappafi_{S_i} = -\tfrac{|\Nmin(Y_i) - \Nmin(X_i)|}{\cH^1(S_i)}\quad\text{resp.}\quad 
\kappafi_{S_i} = \tfrac{|\Nmin(Y_i) - \Nmin(X_i)|}{\cH^1(S_i)}.
$$
When $\sS$ is not locally convex around $S_i,$ then $\kappafi_{S_i} = 0.$

\end{itemize}
\end{remark}

{}From now on, when we speak about the crystalline curvature of a network, we always assume that the triple junctions  are locally contained in conical critical triods.

\begin{example}
In general, \eqref{min_problem91} may have more than one minimizer. For instance, consider an admissible triod of half-lines $S_1,S_2,S_3$ meeting at $X$ at $120^o$-angles. Then $N_1(X)$ can be any vector in the facet of $\Wulff$ parallel to $S_1$ having the same unit normal. Clearly, $N_2(X)$ and $N_3(X)$ are uniquely chosen (satisfying the balance condition). Then the locally constant vector field $N_i:=N_i(X)$ is a minimal CH field. 
In particular, in this case there are infinitely many minimal CH fields. 
\end{example}

\begin{lemma}[\textbf{$\phi$-curvature-balance condition}]\label{lem:sum_curvature_balance}
Let $\sS$ be a polygonal admissible network containing at least one segment, $\triplejunction$ be a triple junction of three segments/half-lines $S_1,S_2,S_3$ of $\sS$ meeting at $120^o$-angles and suppose that there exists a minimal CH field whose values at $\triplejunction$ do not coincide with vertices of $\Wulff.$ Then 
\begin{equation}\label{sum_kurvacha0}
(-1)^{\exponent_1}\kappa_{S_1}^\phi + (-1)^{\exponent_2}\kappa_{S_2}^\phi 
+ (-1)^{\exponent_3}\kappa_{S_3}^\phi = 0,
\end{equation}
where $\exponent_i=0$ if $S_i$ is oriented from $X$ and $\exponent_i=1$ otherwise.
\end{lemma}

\begin{proof}
Let  $\Nmin$ be a minimal CH field as in the statement. Without loss of generality we assume that $S_1,S_2,S_3$ are oriented from $\triplejunction$ so that 
\begin{equation*}
\Nmin_1(\triplejunction) + \Nmin_2(\triplejunction) + \Nmin_3(\triplejunction)=0. 
\end{equation*}
Since $\sS$ contains at least one segment and $\triplejunction$ is a triple junction, at least one $S_i$, say  $S_1,$ is a segment $[\triplejunction Y_1]$. First assume that both $S_2$ and $S_3$ are half-lines, and in this case we define a CH field $N$ on $\sS$ as follows: we set $N=\Nmin$ on all segments/half-lines $S_i$ of $\sS$ with $i\ge4$ (if any), and $N_i=N_{\vert S_i}$ is constant on $S_1,S_2,S_3$ with $N_1= \Nmin_1(Y_1)$ and $N_2$ and $N_3$ are unique constant vectors satisfying $\Nmin_1(Y_1) + N_2+N_3=0.$ Obviously, $N$ is also a CH field minimizing  the functional in \eqref{min_problem91}, and by the uniqueness of the tangential divergence of the minimizers, $\kappafi_{S_i}=0$ for $i=1,2,3.$ Hence, \eqref{sum_kurvacha0} holds.

Now, assume $S_2=[\triplejunction Y_2]$ is a segment and $S_3$ is a half-line. Clearly, $\kappa_{S_3}^\phi=0.$ Let $\bar N\in\cahnhofman(\sS)$ be any CH field such that $\bar N=\Nmin$ on all $S_i$ with $i\ge4$. Let $V_1$ be a vertex of $\Wulff$ closest to $N_1(\triplejunction),$ and $V_2,V_3$ be other two vertices directed as $N_2(\triplejunction)$ and $N_3(\triplejunction),$ basically obtained rotating $V_1$ by $\pm120^o.$ Let us define 
$$
x:=|\bar N_1(\triplejunction)-V_1|=|\bar N_2(\triplejunction)-V_2|\in[0,\tfrac{2}{\sqrt3}],
\quad a_1:=|\bar N_1(Y_1)-V_1|,\quad a_2:=|\bar N_2(Y_2)-V_2|.
$$
Then by the minimality of $\Nmin$ (see also the proof of Lemma \ref{lem:graph_triple} below) $x=x^0:=|\Nmin_1(\triplejunction)-V_1|$ is a minimizer of the function 
$$
f(x):=\frac{(x-a_1)^2}{\cH^1(S_1)}+\frac{(x-a_2)^2}{\cH^1(S_2)}.
$$
By assumption $\Nmin_1(\triplejunction)$ does not coincide with vertices of $\Wulff,$ and thus, $x^0\in(0,\tfrac{2}{\sqrt3}).$ Therefore, it is an interior critical point of $f,$ i.e., 
$$
f'(x^0) = \frac{2(x^0-a_1)}{\cH^1(S_1)}+\frac{2(x^0-a_2)}{\cH^1(S_2)} = 0.
$$
Now observing $\frac{x^0-a_1}{\cH^1(S_1)} = -\frac{(\Nmin_1(Y_1)- \Nmin_1(\triplejunction)]\cdot \tau_{S_1}}{\cH^1(S_1)} = -\kappa_{S_1}^\phi$ and $\frac{x^0-a_2}{\cH^2(S_2)} = -\frac{(\Nmin_2(Y_2)- \Nmin_2(\triplejunction)]\cdot \tau_{S_2}}{\cH^2(S_2)} = -\kappa_{S_2}^\phi$ (for the first equality use that $\Nmin_1(\triplejunction),V_1,\Nmin_1(Y_1)$  lie on the same facet of $\Wulff$ and for the second one, see \eqref{u_nego_bomba}), we deduce 
$$
\kappa_{S_1}^\phi+\kappa_{S_2}^\phi+\kappa_{S_3}^\phi=\kappa_{S_1}^\phi+\kappa_{S_2}^\phi=0.
$$

The case when all $S_1, S_2, S_3$ are segments is treated similarly.
\end{proof}

\subsection{Computation of $\phi$-curvature}\label{subsec:comput_curvatures}

\begin{wrapfigure}[13]{l}{0.55\textwidth}
\vspace*{-3mm}
\includegraphics[width=0.54\textwidth]{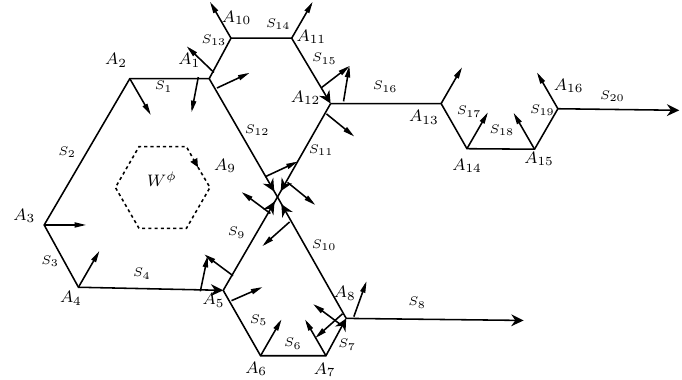} 
\caption{\small}
\label{fig:curva_compa}
\end{wrapfigure}
As an example, let us compute the $\phi$-curvature of the segments of the network $\sS$ in Figure \ref{fig:curva_compa}, where arrows at the endpoints of the polygonal curves show the orientation. Given $N\in\cahnhofman(\sS),$ write as usual $N_i:=N^{}_{\vert S_i}.$ As mentioned in Remark \ref{rem:prop_minimal_chn}, the $\phi$-curvature of the half-lines $S_{20}$ and $S_8$ are $0.$

First, consider the segment $S_2:=[A_2A_3].$ This segment does not end at any multiple junction, $N_2(A_2):=(\frac{1}{\sqrt3},-1)$ and $N_2(A_3):=(\frac{2}{\sqrt3},0).$ The curve $[A_1A_2A_3A_4]$ (and hence $\sS$) is locally convex around $S_2.$ As shown in Figure \ref{fig:curva_compa}, the unit normal of $S_2$ is directed ``inward'', and hence, by Remark \ref{rem:prop_minimal_chn} the $\phi$-curvature of $S_2$ is negative and is equal to $-\frac{|N_2(A_3) - N_2(A_2)|}{\cH^1(S_2)} = -\frac{2}{\sqrt3\cH^1(S_2)},$ where $\frac{2}{\sqrt3}$ is the sidelength of $\Wulff.$
Similarly, $\kappafi_{S} = -\frac{2}{\sqrt3\cH^1(S)}$ for $S\in \{S_3,S_6,S_{18}\}.$ On the other hand, the curve $[A_1A_{10}A_{11}A_{12}]$ (and hence $\sS$) is locally convex around $S_{14}=[A_{10}A_{11}]$ and the unit normal of $S_{14}$ is directed ``outward''. Thus, the $\phi$-curvature of this segment is positive and equals to $\frac{2}{\sqrt3\cH^1(S_{14})}.$ Notice that $\sS$ is not locally convex around $S_{17}$ and $S_{19},$ and therefore, the $\phi$-curvature of these segments is zero.

Now consider the segments ending at multiple junctions, for instance, at the triple junction $A_1.$ Let
$$
x_1:=|N_1(A_2) - N_1(A_1)|\in [0,\tfrac{2}{\sqrt3}].
$$
In view of \cite[Lemma 2.16]{BCN:2006}, given $N_1(A_1)$ we can uniquely define $N_{13}(A_1)$ and  $N_{12}(A_1)$ to fulfill the balance condition
$$
N_1 + N_{13}(A_1) + N_{12}(A_1) = 0,
$$
the signs ``+'' are chosen because of the orientations of $S_1,S_{13}$ and $S_{12}$ with respect to $A_1.$ Then using the symmetry of $\Wulff$ one can readily check that 
$$
\big|N_{13}(A_1)- (-1,0)\big| = 
\big|N_{12}(A_1)- (1,0)\big| = x_1.
$$
In particular, 
$$
|N_{13}(A_{10}) - N_{13}(A_1)| = \tfrac{2}{\sqrt3} - x_1.
$$
Next consider the triple junction $A_{12}.$ Setting 
$$
x_2:=|N_{15}(A_{12}) - (1,0)|=|N_{16}(A_{12})-(\tfrac{1}{\sqrt3},1)| = |N_{11}(A_{12}) - (1,0)| \in[0,\tfrac{2}{\sqrt3}],
$$
we have 
$$
|N_{15}(A_{12}) - N_{15}(A_{11})| = |N_{16}(A_{13}) - N_{16}(A_{12})| = \frac{2}{\sqrt3} -x_2.
$$
Similarly, for the triple junction $A_8$ we have 
$$
x_3:=|N_{10}(A_8) - (-1,0)| = |N_8(A_8) - (-\tfrac{1}{\sqrt3},1)| = |N_7(A_8) - (-1,0)|\in[0,\tfrac{2}{\sqrt3}]
$$
and 
$$
|N_7(A_8) - N_7(A_7)| = x_3,
$$
whereas for the triple junction $A_5$ we have 
$$
x_4:=|N_9(A_5) - (-1,0)| = |N_4(A_5) - (\tfrac{1}{\sqrt3},1)| = |N_5(A_5) - (1,0)|\in[0,\tfrac{2}{\sqrt3}]
$$
and 
$$
|N_4(A_5) - N_4(A_4)| = |N_5(A_6) - N_5(A_5)| = x_4.
$$
Finally, we turn to the quadruple junction $A_9.$ According to Figure \ref{fig:curva_compa} let 
$$
\big|N_{12}(A_9) - (1,0)\big| = x_5 \in[0,\tfrac{2}{\sqrt3}],\quad 
\big|N_{11}(A_9) - (1,0)\big| = x_6 \in[0,\tfrac{2}{\sqrt3}].
$$ 
Since all segments $S_{19}, S_{20}, S_{21}, S_{22}$ enter (are directed to) the quadruple junction, one can readily check (see also Example \ref{exa:four_networks}) that the balance condition
$$
[N_{12} + N_{11} + N_9 + N_{10}]\big|_{A_9} = 0
$$
holds if and only if $N_{12}(A_9) = -N_{10}(A_9)$ and $N_{11}(A_9) = -N_9(A_9).$ Thus, 
$$
|N_{12}(A_9) - N_{12}(A_1)| = |x_5 - x_1|,
\quad 
|N_{11}(A_9) - N_{11}(A_{12})| = |x_6 - x_2|,
$$
and
$$
|N_{10}(A_9) - N_{10}(A_8)| = |x_5 - x_3|,
\quad 
|N_9(A_9) - N_9(A_5)| = |x_6 - x_4|.
$$
In view of these observations $N_0$ is a solution of the minimum problem \eqref{shapat_minimi} if and only if $x^0:=(x_1^0,\ldots, x_6^0),$ defined as above with $N=\Nmin,$ minimizes the quadratic function 
\begin{multline*}
g(x):= 
\tfrac{x_1^2}{\cH^1(S_1)} 
+ 
\tfrac{(\frac{2}{\sqrt3} - x_1)^2}{\cH^1(S_{13})} 
+ 
\tfrac{(x_5-x_1)^2}{\cH^1(S_{12})} 
+
\tfrac{(\frac{2}{\sqrt3}-x_2)^2}{\cH^1(S_{15})} 
+ 
\tfrac{(\frac{2}{\sqrt3}-x_2)^2}{\cH^1(S_{16})} \\
+
\tfrac{(x_6-x_2)^2}{\cH^1(S_{11})} 
+
\tfrac{x_3^2}{\cH^1(S_7)}
+
\tfrac{(x_5-x_3)^2}{\cH^1(S_{10})}
+
\tfrac{x_4^2}{\cH^1(S_5)}
+
\tfrac{x_4^2}{\cH^1(S_4)}
+
\tfrac{(x_6-x_4)^2}{\cH^1(S_9)}
\end{multline*}
among all $x:=(x_1,\ldots,x_6)\in[0,\tfrac{2}{\sqrt3}]^6.$

\begin{example}\label{ex:broken_turlik}
Suppose furthermore that in Figure \ref{fig:curva_compa}
$$
\cH^1(S_{12}) = \cH^1(S_{11}) = \cH^1(S_9) = \cH^1(S_{10}) = 1
$$
and 
$$
\cH^1(S_1) = \cH^1(S_{13}) = \cH^1(S_{15}) = \cH^1(S_{16}) = \cH^1(S_4) = \cH^1(S_5) = \cH^1(S_7) = 1 - \epsilon
$$
for some $\epsilon\in(0,1).$ Then 
\begin{multline*}
g(x) = \tfrac{x_1^2 + (\tfrac{2}{\sqrt3}-x_1)^2 +2(\tfrac{2}{\sqrt3}-x_2)^2 +x_3^2 + 2x_4^2}{1-\epsilon}\\
+ (x_5-x_1)^2 + (x_6-x_2)^2 + (x_5-x_3)^2 + (x_6-x_4)^2.
\end{multline*}
Since $g$ is nonnegative and quadratic, solving the linear  system $\nabla g(x)=0$ we find that the minimizer $x^0$ is uniquely defined as 
$$
x_1^0:=\tfrac{2(3-\epsilon)}{\sqrt{3}(7-3\epsilon)},\quad 
x_2^0:=\tfrac{5-\epsilon}{\sqrt{3}(3-\epsilon)},\quad 
x_3^0:=\tfrac{2(1-\epsilon)}{\sqrt{3}(7-3\epsilon)},\\
$$
$$
x_4^0:=\tfrac{1-\epsilon}{\sqrt{3}(3-\epsilon)},\quad 
x_5^0:=\tfrac{2(2-\epsilon)}{\sqrt{3}(7-3\epsilon)},\quad 
x_6^0:=\tfrac{1}{\sqrt3}.
$$
Thus the values of the $\phi$-curvature of the segments ending at the quadruple junction are 
$$
\kappafi_{S_{12}} 
= \kappafi_{S_{10}} = \tfrac{2}{\sqrt{3}(7-3\epsilon)},\quad 
\kappafi_{S_{11}} 
= -\kappafi_{S_9} = \tfrac{2}{\sqrt{3}(3-\epsilon)} \neq \kappafi_{S_{12}},
$$
therefore, if these segments were translating in the direction of their unit normals with velocity equal to their $\phi$-curvature, then the quadruple junction should break into two triple junctions, and the network instantaneously changes its topology.
\end{example}

We can construct other networks containing multiple junctions which exhibit such an unstable behaviour, see Figure \ref{fig:nonstable}.

\begin{figure}[htp!]
\includegraphics[width=0.9\textwidth]{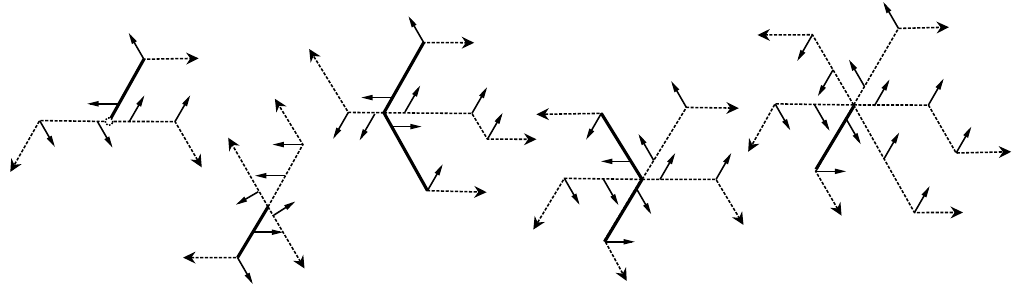}
\caption{\small Networks with multiple junctions at which at least one segment (in bold) has  nonzero $\phi$-curvature (for a suitable choice of the the lengths of the segments).}
\label{fig:nonstable}
\end{figure}

\subsection{Parallel networks}

Following the two-phase case we assume that segments in polygonal networks during the flow translate parallel. As in \cite{BCN:2006,BKh:2023} this encourages the following definition.

\begin{definition}[\textbf{Parallel network}] 
Let $\sS:=\cup_i S_i$ be a polygonal network consisting of a union of $N$ segments and $M$ half-lines. A polygonal network $\bar \sS$ is called \emph{parallel} to $\sS$ provided that:

\begin{itemize}
\item $\bar \sS:=\cup_i \bar S_i$ is a union of $N$ segments and $M$ half-lines and each $S_i$ is parallel to $\bar S_i$ (so that $\nu_{S_i} = \nu_{\bar S_i}$); 

\item if $S_i$ is a segment, then $\bar S_i$ is also a segment;

\item if $S_i$ is a half-line, then $\bar S_i$ is also a half-line and $S_i\Delta \bar S_i$ is bounded (hence $S_i$ and $\bar S_i$ lie on the same straight line);

\item if $\degree \ge2$ segments/half-lines $S_{i_1},\ldots,S_{i_\degree}$ have a common endpoint (for instance they form a simple vertex for $\degree=2$ or an $\degree$-junction for $\degree\ge3$), then so do $\bar S_{i_1},\ldots,\bar S_{i_\degree}$.
\end{itemize}
\end{definition}

If $ \sS=\cup_iS_i$ and $\bar \sS=\cup_i\bar S_i$ are parallel and  $S_i\cap S_j\ne\emptyset$ for some $i\ne j,$ then $\bar S_i \cap \bar S_j \ne \emptyset,$ and the angle between $S_i$ and $S_j$ equals  the angle  between $\bar S_i$ and $\bar S_j$  at their common point. In particular, any network parallel to an admissible network is itself admissible.

\begin{definition}[\textbf{Distance vectors}] 
$\,$

\begin{itemize}
\item Let $S,T$ be two parallel straight lines. A vector $H(S,T)\in\R^2$ satisfying $T = S + H(S,T)$ is called a \emph{distance vector} from $S$ to $T.$ For any interval $S_1\subseteq S$ and interval $T_1\subseteq T$ we write $H(S_1,T_1): = H(S,T).$ The distance from $S_1$ to $T_1$ is defined as
$$
\dist(S_1,T_1):=|H(S_1,T_1)|.
$$
In what follows we frequently refer to the number 
$$
h:=H(S_1,T_1)\cdot \nu_{S_1}
$$
as the \emph{(signed) height} from $S_1$ to $T_1.$ Note that $H(S_1,T_1) = h\nu_{S_1}$.

\item The \emph{distance} between two parallel networks $\sS:=\cup_{i=1}^n S_i$ and $\widehat \sS:=\cup_{i=1}^n\widehat S_i$ is given by  
$$
\dist(\sS,\widehat \sS) := \max\limits_{1\le i\le n} \,\,\dist(S_i,\widehat S_i).
$$
\end{itemize}
\end{definition}

\section{$\phi$-curvature flow of admissible $\phi$-regular networks}

Recalling the definition of admissible network (Definition \ref{def:admissible_network}) and of $\phi$-regular network (Definition \ref{def:phi_regular_networks_and_CH_fields}), we can now introduce the $\phi$-curvature flow.

\begin{definition}[\textbf{$\phi$-curvature flow}]\label{def:phi_curvature_flow}
Let $\initialnetwork$ be a $\phi$-regular (polygonal) admissible network, and $T \in (0,+\infty]$. A family $\{\sS(t)\}_{t\in[0,T)}$ is called a regular \emph{$\phi$-curvature flow in $[0,T)$ (a $\phi$-regular flow, for short) starting  from $\initialnetwork$} provided that:
\begin{itemize}
\item[(a)] $\sS(t)$ is parallel to $\initialnetwork$ for all $t\in[0,T)$;

\item[(b)] if $\initialnetwork=\cup_iS_i^0$ and $\sS(t)=\cup_iS_i(t),$ then the heights 
$$
h_i(\cdot):=H(S_i(\cdot),S_i^0)\cdot \nu_{S_i^0}
$$
belong to $C^1((0,T))\cap C^0([0,T))$ and satisfy 
\begin{equation}\label{mcf_heights}
\frac{d}{dt}\,h_i(t) = -\phi^o(\nu_{S_i(t)})\kappafi_{S_i(t)},\quad t\in(0,T),
\end{equation}
for any $i=1,\dots, n$.
\end{itemize}
\end{definition}

By the admissibility of $\sS(t)$ we have $\phi^o(\nu_{S_i(t)}) = 1$ for all $t\in [0,T)$, so \eqref{mcf_heights} reads as
\begin{equation}\label{eq:mcf_heights_simplified}
\frac{d}{dt}\,h_i(t) = - \kappafi_{S_i(t)},\quad t\in(0,T). 
\end{equation}

\begin{remark} 
$\,$
\begin{itemize}
\item[(a)] By our sign conventions, the $\phi$-curvature of the segments of any convex $\phi$-regular hexagon $\initialnetwork$ is nonnegative and hence the $\phi$-curvature flow $\sS(\cdot)$ starting from $\initialnetwork$ shrinks the hexagon.

\item[(b)] If $\initialnetwork$ is a critical network (Definition \ref{def:critical_network}), then the $\phi$-curvature of its segments/half-lines is $0.$ Therefore, the stationary flow $\sS(t):=\initialnetwork$ is the unique $\phi$-curvature flow starting from $\initialnetwork$ in $[0,+\infty)$.

\item[(c)] Being a gradient flow of the $\phi$-length, the $\phi$-curvature flow is expected to decrease the $\phi$-length. 
However, this is not the case for nonminimal critical networks, see for instance the network in Figure \ref{fig:non_minimal}. 

\item[(d)] Example \ref{ex:broken_turlik} (see also Figure \ref{fig:nonstable}) shows that not every admissible network admits a regular $\phi$-curvature flow: the $\phi$-curvature flow instantaneously should change the topological structure and create new segments or curves (see also \cite{BCN:2006}). 

\end{itemize}
\end{remark}

\begin{definition}[\textbf{Simple network}]\label{def:simple_network}
An admissible network is called \emph{simple} if it only contains triple junctions meeting at $120^o$-angles.
\end{definition}

\begin{remark}
In \cite{BCN:2006,BKh:2023} a network with a single triple junction $\triplejunction$ formed by three polygonal curves, each of which is the union of a segment and a half-line, and possibly with three anisotropies, is called \emph{stable}  provided the values of the unique minimal CH field at $X$ do not coincide with vertices of the corresponding Wulff shapes. Definition \ref{def:simple_network}  and the next theorem generalize \cite{BCN:2006,BKh:2023} in our single anisotropic case to a wider class of networks which may contain several triple junctions, some of which may not be stable.
\end{remark}

The main result of this section is the following

\begin{theorem}[\textbf{Existence and uniqueness of the $\phi$-curvature flow}]\label{teo:short_time}
Let $\initialnetwork$ be a simple network. Then there exist $\maximaltime \in (0,+\infty]$ and a unique family $\{\sS(t)\}_{t\in[0,\maximaltime)}$ of simple  networks such that $\sS(\cdot)$ is the $\phi$-curvature flow starting from $\initialnetwork.$ Moreover, if $\maximaltime<+\infty$ then some segment of $\sS(t)$ vanishes as $t\nearrow \maximaltime.$
\end{theorem}

We need some auxiliary results, which actually provide the steps of the proof, concluded in Section \ref{sec:proof_of_theorem}. 

\begin{lemma}[\textbf{Quadratic minimization}]\label{lem:quadratic_minimization}
For $n\ge1$ let $\{a_i\}_{i=1}^n,$ $\{b_i\}_{i=1}^n$ and $\{c_{ij}\}_{i,j=1}^n$ be finite sets of nonnegative numbers such that
\begin{itemize}
\item $c_{ij}=c_{ji},$ $c_{ii}=0$ for all $i,j=1,\ldots,n;$

\item if $n\ge2,$ the square matrix $C:=(c_{ij})$ is irreducible, i.e., for every $i\ne j,$ there exists $m_{ij}\ge1$ such that $[C^{m_{ij}}]_{ij}>0.$ Equivalently, the unoriented graph $\sG$ with $n$ nodes and adjacency matrix\footnote{The adjacency matrix of an (oriented or unoriented) graph is the matrix $(\alpha_{ij})$, where $\alpha_{ij}=1$ if there is an edge from the node $i$ to node $j$, and $\alpha_{ij}=0$ otherwise.} $(\delta_{c_{ij},0})$ is connected, where $\delta_{\alpha,\beta}=1$ if $\alpha=\beta$ and $\delta_{\alpha,\beta}=0$ if $\alpha\ne\beta$ \cite[Chapter 1]{S:2015}.
\end{itemize}
Consider the quadratic function 
$$
\psi(x) := \sum\limits_{i=1}^n a_i x_i^2 + 
\sum\limits_{i=1}^n b_i (d-x_i)^2 + \sum\limits_{1\le i<j\le n}  c_{ij} 
(x_i - x_j)^2,\quad x = (x_1,\ldots,x_n)\in \Rn.
$$
Then, for $d>0$,
$$
\min_{x \in [0,d]^n}\psi(x) = 0\quad \Longleftrightarrow\quad \text{either all $a_i$ are zero or all $b_i$ are zero.}
$$
Furthermore, if $\min_{x \in [0,d]^n} \psi(x) > 0$ then $\psi$ has a unique minimizer and this minimizer lies in $(0,d)^n.$
\end{lemma}

\begin{proof}
If $\sum_{i=1}^n b_i=0,$ i.e., all $b_i$ are $0,$ then $(0,\ldots,0)\in [0,d]^n$ is the minimizer of $\psi$  and the minimum is $0.$ On the other hand, if $\sum_{i=1}^n a_i=0,$ i.e., all $a_i$ are $0,$ then $(d,\ldots,d)\in [0,d]^n$ is the  minimizer of $\psi$ and again the minimum is $0.$ As we shall see, these are the only cases when the minimizer belongs to the boundary of $[0,d]^n$.

Assume that $\sum_{i=1}^n b_i>0$  and $\sum_{i=1}^n a_i>0.$ If $n=1$ so that $a_1b_1>0,$ then $\min_{x_1\in[0,d]}\psi(x_1)>0$ and its unique minimizer $x_1^0:=\frac{b_1d}{a_1+b_1}$ belongs to $(0,d).$ Therefore, further, we may suppose that $n\ge2$.

Consider the equation $\nabla \psi(x)=0$ for $x \in \Rn$, i.e.,
\begin{equation}\label{dominoDomino}
\Big(a_i+b_i+\sum_{j=1}^nc_{ij}\Big)x_i - \sum_{j=1}^n c_{ij}x_j = b_id,\qquad i=1,\ldots,n.
\end{equation}
Let $A:=(A_{ij})$ be the square matrix whose entries are 
$$
A_{ii}=a_i+b_i+\sum_{l=1}^nc_{il},\quad 
A_{ij}=-c_{ij},\quad i,j =1,\ldots,n,\quad i\ne j.
$$
We have 
\begin{equation}\label{kibbutth}
Ax^T\cdot x^T = \sum_{i,j=1}^n A_{ij}x_ix_j = \sum_{i=1}^n(a_i+b_i)x_i^2 + \tfrac{1}{2}\sum_{i\ne j} c_{ij} (x_i-x_j)^2 \geq 0.
\end{equation}
We claim that $Ax^T\cdot x^T=0$ if and only if $x=0,$ and so, $A$ is positive definite. Indeed, since the graph $\sG$ associated to the matrix $C=(c_{ij})$ is connected, the last sum in \eqref{kibbutth} is zero if and only if $x_i=x_j$ for all $i\ne j.$ Indeed, by connectedness, we can reach from the vertex $i=1$ to each vertex $i>1$ of the graph through the vertices $i_1=1,i_2,\ldots,i_m=i.$ Then $c_{i_1i_2},c_{i_2i_3},\ldots, c_{i_{m-1}i_m}>0,$ and hence, $x_1=x_{i_2}=\ldots=x_{i_{m-1}}=x_i.$ Since $\sum_i(a_i+b_i)>0$ by assumption, this implies $Ax^T\cdot x^T=0$ iff $x=0.$ In particular all diagonal elements $A_{ii}$ of $A$ are positive.

Let us show that all entries $\hat a_{ij}$ of $A^{-1}$ are also positive. Indeed, let $D$ be the matrix formed by the diagonal elements of $A,$ i.e., $D := A+C.$ By assumption $c_{ij}\ge0,$ hence the entries of the matrix $X:=D^{-1/2}CD^{-1/2}$ are nonnegative. Moreover, since $D>C$ in the sense of linear operators, for any $x\in\S^{n-1}$
$$
Xx^T\cdot x^T = C(D^{-1/2}x)^T\cdot (D^{-1/2}x)^T =  Cy^T\cdot y^T<Dy^T\cdot y^T = |D^{1/2}y|^2 =|x|^2 =1,
$$
where $y:=D^{-1/2}x\ne0.$ Therefore, the norm $\|X\|=\sup_{x\in\R^n,\,\|x\|=1}\, Xx^T\cdot x^T$. satisfies $\|X\|<1.$ Then the matrix $I-X$ is invertible and its inverse is given by the Neumann series,  
$$
A^{-1}=D^{-1/2}(I-X)^{-1}D^{-1/2} = \sum_{k\ge0}D^{-1/2} X^k D^{-1/2}. 
$$
Clearly, the entries of  $D^{-1/2} X^k D^{-1/2}$ are nonnegative for all $k\ge0$. Since $C$ is irreducible with nonzero elements, and $D^{-1/2}$ is diagonal with positive diagonal elements, $X$ is also irreducible. 
In particular, all entries of $X^m$ for some $m\ge1$ are positive and hence, all elements of $A^{-1}$ are also positive. 

Therefore, the system \eqref{dominoDomino} has a unique solution 
$$
x^T:=A^{-1}b^T, \quad b=(b_1d,\ldots,b_nd); 
$$
and recalling that $b_id\ge0$ with $\sum_{i=1}^n b_i>0,$ we deduce $x_i^0>0$ for all $i=1,\ldots,n.$ This unique solution provides the unique minimizer of $\psi.$

To prove $x_i^0<d,$ we apply the previous argument to the quadratic function
$$
\psi^*(x_1,\ldots,x_n)=\psi(d-x_1,\ldots,d-x_n)
$$
(in this case we use $\sum a_i>0$) and conclude that the unique minimizer $y^0$ of $\psi^*$ satisfies $y_i^0>0.$ By uniqueness, this implies $d-y_i^0=x_i^0$ and hence, $x_i^0\in (0,d)$ for all $i=1,\dots,n$.
\end{proof}

Let $\sS$ be a simple network, and divide $\sS$ into connected graphs $G_1,\ldots,G_m$ removing all simple vertices. Since $\sS$ contains only triple junctions, each $G_i$ is either a single segment/half-line or a union of segments/half-lines at some triple junction. 

\begin{lemma}\label{lem:graph_triple}
Let $\Nmin$ be a minimal CH field of the simple network $\sS$. Suppose $G_i$ contains at least one triple junction. Then for $G_i$ the following holds:
\begin{itemize}
\item either ${\rm div}_\tau \Nmin=0$ on all segments/half-lines of $G_i$, and in this case the values of $\Nmin$ at all triple junctions of $G_i$ can be chosen as (three distinct) vertices of $\Wulff,$

\item or the values of $\Nmin$ at all triple junctions of $G_i$ do not coincide with any vertex of $\Wulff$ 
\end{itemize}
(see Figure  \ref{fig:parsing_grapf}).
\end{lemma}

\begin{figure}[htp!]
\includegraphics[width=0.6\textwidth]{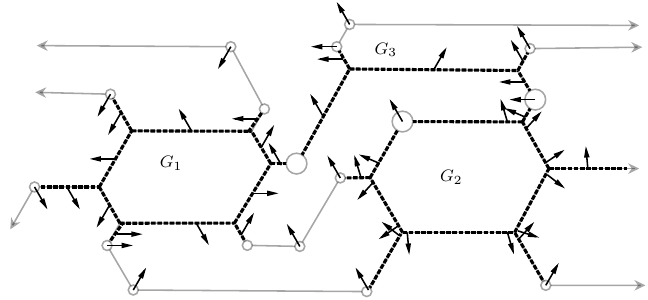}
\caption{\small A simple network parsed into connected graphs $G_1,G_2,G_3,$ by removing all simple vertices (small circles), and a possible CH field. Here $G_1,G_2,G_3,$ consisting of a union of bold dotted lines, are the only graphs containing triple junctions (other graphs are either isolated segments or half-lines). Notice that $G_1$ and $G_3$ admit a (locally) constant CH field.}
\label{fig:parsing_grapf}
\end{figure}

\begin{proof}
We assume that $i=1$ and $G_1$ contains at least one segment, and $G_1$ has exactly $r\ge1$ triple junctions, $X_1,X_2,\ldots,X_r$. Let $N\in \cahnhofman(\sS ).$ As we observed in Remark \ref{rem:prop_minimal_chn}, the values of $N$ at the simple vertices of $\sS $ are uniquely defined. Consider any $X_k,$ which is a junction of segments/half-lines, say,  $S_{k_1} ,S_{k_2}  $ and $S_{k_3} ,$ and assume that $S_{k_1}=[X_{k_1}Y_{k_1}]$ is a segment (oriented from $X_{k_1}$). Let $V_k^1$ be the vertex of $\Wulff$ such that $N_{k_1}(X_{k_1}),N_{k_1}(Y_{k_1})$ and $V_k^1$ lie in the same facet of $\Wulff.$ As in Section \ref{subsec:comput_curvatures} define
$$
x_k:=|N_{k_1}(X_k) - V_k^1|\in [0,\tfrac{2}{\sqrt3}].
$$
Repeating the same arguments in Section \ref{subsec:comput_curvatures}, the minimum problem leading to $\Nmin$ is reduced to minimizing the function 
\begin{align}\label{function_psi_qara}
\psi(x_1,\ldots,x_r) = &  \sum\limits_{k=1}^r  \Big(\tfrac{\alpha_k^1}{\cH^1(S_{k_1} )} +  \tfrac{\alpha_k^2}{\cH^1(S_{k_2} )}+\tfrac{\alpha_k^3}{\cH^1(S_{k_3} )}\Big) \,x_k^2 \nonumber \\
&+
\sum\limits_{k=1}^r  \Big(\tfrac{\beta_k^1}{\cH^1(S_{k_1} )} + \tfrac{\beta_k^2}{\cH^1(S_{k_2} )}+\tfrac{\beta_k^3}{\cH^1(S_{k_3} )}\Big)\,\big(\tfrac{2}{\sqrt3}-x_k\big)^2
+
\sum\limits_{1\le i<j\le r} \tfrac{\gamma_{ij}\,(x_i-x_j)^2}{\cH^1(S_{i,j})}
\end{align}
in the cube $[0,\frac{2}{\sqrt3}]^r,$ where $\alpha_k^i,\beta_k^i \in\{0,1\}$,  $\gamma_{ij}\in\{0,1\}$ and $\gamma_{ij}=1$ if and only if there is a segment $S_{i,j}$ of $\sS $ connecting  $X_i$ and $X_j,$ and for each $k,$ these coefficients are uniquely defined depending only on how segments/half-lines $S_{k_1},S_{k_2},S_{k_3},$ forming a junction at $X_k,$ behave at their other endpoints. Indeed, for shortness setting $\gamma_{ii}=0,$ let another endpoint $Y$ of $S_{k_1}$ be a simple vertex and $N_{k_1}(Y)$ bisects the interior resp. exterior angle of $\sS $ at this vertex. Then $\beta_k^1=0$ and $\alpha_k^1=1$ resp. $\beta_k^1=1$ and $\alpha_k^1=0$, and also $\gamma_{kl}=\gamma_{lk}=0$ for all $l$. On the other hand, if the other endpoint of $S_{k_1}$ is another triple junction, say, $X_l,$ then $S_{k,l}:=S_{k_1},$ $\gamma_{kl}=1$ and $\alpha_k^1=\beta_k^1=0.$ The same applies also to $S_{k_2}$ and $S_{k_3}$. 
In particular, for each $x_k,$ at most three of $\{\alpha_k^i,\beta_k^i,\gamma_{ki}\}_{i=1}^r$ can be nonzero. Moreover, by the connectedness of $G_1,$ the matrix $(\gamma_{ki})$ is irreducible.

Now by Lemma \ref{lem:quadratic_minimization} either $\min\psi=0,$ which is possible if and only if either all $\alpha_k^i=0$ (so that $x_1=\ldots=x_r=0$ minimizes $\psi$) or $\beta_k^i=0$ (so that $x_1=\ldots=x_r=\tfrac{2}{\sqrt3}$ minimizes $\psi$), or $\min\psi>0$ so that the unique minimizer $(x_1 ,\ldots,x_r )$ is an interior point of the cube $[0,\tfrac{2}{\sqrt3}]^r.$ By definition $\min\{x_k,\tfrac{2}{\sqrt3}-x_k\}$ measures the distance between the values of $\Nmin$ at $X_k$ and the corresponding vertices of $\Wulff,$ therefore, in case $\min\psi>0$ those values do not coincide with the vertices of $\Wulff.$
\end{proof}

\begin{definition}[\textbf{ic-triple junctions and bc-triple junctions}]
\label{def:interior_to_and_at_the_boundary_of_the_constraint}
Let $\sS$ be a simple network, and $X$ be a triple junction of $X$. $X$ is called \emph{interior to the constraint} (shortly, an ic-triple junction) if any minimal CH field at $X$ do not coincide with vertices of $\Wulff.$ $X$ is called at the \emph{boundary of the constraint} (shortly, a bc-triple junction) if there is a minimal CH field having values at $X$ coinciding with some vertices of $\Wulff$. 
\end{definition}

The following lemma shows that ic-triple junctions and bc-triple junctions are preserved in parallelness.

\begin{lemma}[\textbf{Preserving parallelness}]
\label{lem:preserving_parallelness}
Let $\sS$ and $\bar \sS$ be two parallel simple networks,  and let $\{G_i\}$ and $\{\bar G_i\}$ be their partitions into connected graphs as above. Then, for any $i=1,\dots,r$,  both $G_i$ and $\bar G_i$ can contain either only ic-triple junctions or only bc-triple junctions.
\end{lemma}

\begin{proof}
Let $G_i$ contain only bc-triple junctions. By Lemma \ref{lem:graph_triple} all segments/half-lines have zero $\phi$-curvature. By parallelness, we can choose the same CH field along the segments/half-lines of $\bar G_i$ so that all its segments/half-lines have also zero $\phi$-curvature. 
In particular, again by Lemma \ref{lem:graph_triple} all triple junctions of $\bar G_i$ are bc-triple junctions. This argument shows also that if $G_i$ contains only ic-triple junctions then $\bar G_i$ cannot contain bc-triple junctions.
\end{proof}

Lemmas \ref{lem:graph_triple} and \ref{lem:preserving_parallelness}  suggest that bc-triple junctions do not move. 

\begin{definition}[\textbf{The class $\Xi(\sS)$}]
Given a simple network $\sS,$ we denote by $\Xi(\sS)$ the collection of all networks $\sT$ parallel to $\sS$ such that if $X$ is a bc-triple junction of $\sS$, then $X$ is a (bc)-triple junction also for $\sT.$  
\end{definition}

Thus, by definition, the $\phi$-curvature flow $\{\sS(t)\}_{t\in [0,T)}$ starting from $\initialnetwork$ is a subset of $\Xi(\initialnetwork),$ i.e., $\sS(t)\in \Xi(\sS)$ for all times $t\in [0,T)$.

\subsection{Parallel networks}

Now we consider the problem of reconstructing a parallel network from a given set of heights. We shall see later in Lemma \ref{lem:heights_compatibility} that the heights at triple junctions cannot have too much freedom.

\begin{theorem}[\textbf{Reconstruction}]\label{teo:reconstruction}
Let $\sS:=\cup_{i=1}^nS_i$ be a simple network and define
\begin{equation}\label{eq:Delta12}
\Delta_1:=\frac{1}{3\sqrt3}\,\min_{1\le i\le n}\cH^1(S_i),\qquad \Delta_2:=\frac{1}{6}\,\min_{S_i\cap S_j=\emptyset} \,d(S_i,S_j).
\end{equation}
Let $\{h_i\}_{i=1}^n$ be any set of real numbers such that: 
\begin{itemize}
\item[\rm(a)] if $S_i$ is either a segment at a bc-triple junction or a half-line, then $h_i=0;$

\item[\rm(b)] if $S_i,S_j,S_k$ form a ic-triple junction, then 
$$
(-1)^{\exponent_i}h_i + (-1)^{\exponent_j}h_j + (-1)^{\exponent_k}h_k = 0,
$$
where $\exponent_i,\exponent_j,\exponent_k\in\{0,1\},$ and $\exponent_s=0$ is $S_s$ if oriented from the triple junction and $1$ otherwise;

\item[\rm(c)] $|h_j|\le \min\{\Delta_1,\Delta_2\}.$
\end{itemize}
Then there exists a unique network $\sT:=\cup_{i=1}^nT_i,$ parallel to $\sS$ such that $\sT\in \Xi(\sS)$ and 
$$
H(S_i,T_i)\cdot\nu_{S_i} = h_i,\quad i=1,\ldots,n.
$$
\end{theorem}

Assumption (a) in Theorem \ref{teo:reconstruction} says that any bc-triple junction 
of $\sS$ is also a bc-triple junction for $\sT$. Later in Lemma  \ref{lem:heights_compatibility} we shall see that assumption  (b)  allows us to construct a parallel triple junction. Finally, assumption (c) prevents self-intersections of segments in $\sT,$ see Figure \ref{fig:parallel_things} (c).

We postpone the proof after some auxiliary results.  The following lemma defines the distance between the vertices of two parallel cones of opening angle $120^o,$ knowing the heights between the corresponding parallel lines.

\begin{figure}[htp!]
\includegraphics[width=\textwidth]{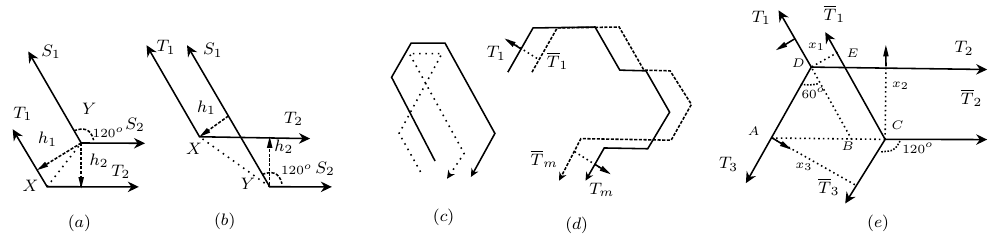}
\caption{\small}
\label{fig:parallel_things}
\end{figure}

\begin{lemma}\label{lem:distan_vertex}
Let $S_1,S_2$ and $T_1,T_2$ be two pairs of segments/half-lines with common starting points $X$ and $Y$ respectively, such that $S_i$ is parallel to $T_i$, the angle between $S_1$ and $S_2$ is $120^o,$ and let $h_i:=H(S_i,T_i)\cdot \nu_{S_i}$ (see Figure \ref{fig:parallel_things} (a)-(b)). Then 
\begin{equation}\label{nujna_rabote}
|XY| = \frac{2}{\sqrt3}\,\sqrt{|h_1|^2 + |h_2|^2 -h_1h_2}.
\end{equation}
\end{lemma}

\begin{proof}
Assume without loss of generality the case of Figure \ref{fig:parallel_things} (a), and let $\alpha$ be the angle between $T_1$ and $[XY].$ Then 
$$
|XY| = \tfrac{|h_1|}{\sin\alpha} = \tfrac{|h_2|}{\sin(120^o-\alpha)}.
$$
From the last equality we get 
$$
|h_1|\cot\alpha = |h_1|\cot120^o + \tfrac{|h_2|}{\sin120^o} = -\tfrac{|h_1|}{\sqrt3} + \tfrac{2|h_2|}{\sqrt3}
$$
and hence, 
$$
|XY|^2 = |h_1|^2 + |h_1|^2\cot^2\alpha = |h_1|^2 +\tfrac{(2|h_2| - |h_1|)^2}{3} = \tfrac{4|h_1|^2 + 4|h_2|^2 - 4|h_1||h_2|}{3}.
$$
Since $|h_1||h_2|=h_1h_2$ we get \eqref{nujna_rabote}. 
\end{proof}

Note that if $S$ is an (oriented) segment or half-line and  $h\in \R,$ then there exists a unique straight line $L$ parallel to $S$ such that 
$
H(S,L) \cdot \nu_S = h.
$
However, given $(h_1,\dots, h_n)\in \Rn$ and an oriented polygonal curve $\Gamma,$ consisting of a union of $n\ge2$ segments $S_1,\ldots, S_n,$ not always one can define a polygonal curve $\bar\Gamma:=\cup_{i=1}^n\bar S_i,$ parallel to $\Gamma$, satisfying $H(S_i,\bar S_i) \cdot \nu_{S_i} = h_i$ for all $i=1,\dots,n,$ 
see Figure \ref{fig:parallel_things} (c). Indeed, if some $|h_i|$ are large, then the relative interiors of two of the  $\bar S_i$'s may intersect.

To retain the injectivity of $\bar\Gamma,$ we have to ensure the smallness of all $|h_i|.$ This is done in the next lemma.

\begin{lemma}[\textbf{Injectivity}]\label{lem:injectivity}
Let $\Gamma$ be an oriented polygonal curve consisting of $n$-segments/half-lines $T_1,\ldots,T_n$ with the $120^o$-angle between $T_i$ and $T_{i+1}$ at their common point for $i=1,\ldots,n-1$ and let $\bar T_1,\ldots,\bar T_n$ be the $n$-segments/half-lines such that $T_i$ and $\bar T_i$ are parallel, the endpoint of $\bar T_j$ is an endpoint of $\bar T_{j+1},$ and 
\begin{equation*}
\dist(T_i,\bar T_i) \le \delta^0:=\frac{1}{3\sqrt3}\,\min\limits_{1\le j\le n} \cH^1(T_j)
\end{equation*}
for all $i=1,\ldots,n.$ Then $\cup_{j=2}^{n-1} \bar T_j$ is also a polygonal curve (without self-intersections). 
\end{lemma}

Observe that we cannot state the injectivity of $\cup_{j=1}^n \bar T_j,$ since a priori we have no information on $\bar T_1$ and $\bar T_n.$

\begin{proof}
We only need to show that $\Sigma:=\cup_{j=2}^{n-1} \bar T_j$ has no self-intersections. Recalling that self-intersections start after some segment disappears, it is enough to show that any segment in the union $\Sigma$ of segments satisfying the assumptions of the lemma has positive length. Note that by the parallelness, $\nu_{T_i} = \nu_{\bar T_i}$ and let 
$$
h_i:= H(T_i,\bar T_i)\cdot \nu_{T_i},\quad i=1,\ldots,n.
$$
Direct computations show (see Figure \ref{fig:parallel_things} (d) and also \cite[Eq. 3.3]{GP:2022})  that 
\begin{equation}\label{length_change}
\cH^1(\bar T_i) = \cH^1(T_i) - \tfrac{2h_{i-1}-2h_i+2h_{i+1}}{\sqrt3}
\end{equation}
for $i=2,\ldots,n-1.$ Thus, if $|h_j|\le \delta^0$ for all $1\le j\le n,$ then 
$$
\Big|\tfrac{2h_{i-1}-2h_i+2h_{i+1}}{\sqrt3}\Big|\le 2\sqrt3\delta^0\le \tfrac{2}{3}\cH^1(T_i),
$$
for all $2\le i\le n-1,$ and hence,
$$
\cH^1(\bar T_i) \ge \tfrac{1}{3}\,\cH^1(T_i)>0,
$$
and thus, $\Sigma$ cannot have self-intersections.
\end{proof}

Lemma \ref{lem:injectivity} has a further implication: if $\Gamma:=\cup_{i=1}^n T_i$ is the polygonal curve in Lemma \ref{lem:injectivity} and $(h_1,\ldots,h_n)$ is an $n$-tuple of real numbers satisfying $|h_i| \le \delta^0,$ then there exists a unique polygonal curve $\Sigma:=\cup_{i=2}^{n-1}\bar T_i$  with $\bar T_i$ parallel to $T_i$ and $H(T_i,\bar T_i)\cdot \nu_{T_i} = h_i$ for any $i=2,\ldots,n-1.$ We can also define $\bar T_1$ and $\bar T_n,$ parallel to $T_1$ and $T_n,$ respectively, with $H(T_i,\bar T_i)\cdot \nu_{T_i} = h_i$ for $i\in \{1,n\},$ however, $\bar T_1 $ and $\bar T_n$ are not uniquely defined (because there is no information on their length).

Next, we study how distances behave in parallel triple junctions.

\begin{lemma}[\textbf{Heights compatibility}]\label{lem:heights_compatibility}
Let $T_1,T_2,T_3$ be segments/half-lines parallel to some segments of $\Wulff$ and forming a triple junction with $120^o$-angles as in Figure \ref{fig:parallel_things} (e) and oriented out of their triple junction. Let $\bar T_1,\bar T_2,\bar T_3$ be another triple of segments/half-lines forming a triple junction such that $T_i$ and $\bar T_i$ are parallel (and so $\nu_{T_i}=\nu_{\bar T_i}$). 
Then the corresponding heights $h_i:=H(T_i,\bar T_i)\cdot \nu_{T_i},$ $ i=1,2,3,$ satisfy 
\begin{equation}\label{heights_compatibility}
h_1 + h_2 + h_3 = 0.
\end{equation}
Conversely, if real numbers $h_1,h_2,h_3$ satisfy \eqref{heights_compatibility}, then there exists a unique triple $\bar T_1,\bar T_2,\bar T_3$ of half-lines such that $\bar T_i$ is parallel to $T_i$ and $H(T_i,\bar T_i)\cdot \nu_{T_i} = h_i.$
\end{lemma}

Thus, the knowledge of two heights at a triple junction allows to determine uniquely the third one. Note that if any of $T_i$ oriented towards to the triple junction, then the corresponding height $h_i$ in \eqref{heights_compatibility} appears with the ``--'' sign.

\begin{proof}
By symmetry, we may assume that $T_i$ and $\bar T_i$ are as in Figure \ref{fig:parallel_things} (e), i.e., $h_3\ge 0$, $h_1 \leq 0$, $h_2 \leq 0$ and let $x_i = |h_i|.$ Then 
\begin{equation}\label{eq:H1_AC}
\cH^1([AC]) = \tfrac{x_3}{\sin60^o} = \tfrac{2x_3}{\sqrt3}.
\end{equation}
Similarly, 
\begin{equation}\label{eq:H1_AC_bis}
\cH^1([AC]) =\cH^1([AB]) + \cH^1([BC]) = \cH^1([EC]) + \cH^1([DE]) =\tfrac{2x_2}{\sqrt3} + \tfrac{2x_1}{\sqrt3},
\end{equation}
and hence, from \eqref{eq:H1_AC} and \eqref{eq:H1_AC_bis} it follows
$
x_3 = x_1+x_2
$
i.e., \eqref{heights_compatibility}. The proof in the other cases is similar. 

To prove the last assertion let us take two half-lines $\bar T_1,\bar T_2$ starting from common point $X,$ parallel to $T_1,T_2,$ respectively, and satisfying $H(T_i,\bar T_i)\cdot \nu_{T_i} = h_i$ for $i=1,2.$ Let $\bar T_3$ be any segment/half-line starting from $X$ and parallel to $T_3$ and define $h_3':=H(T_3,\bar T_3)\cdot \nu_{T_3}.$ By the first part of the proof $h_1+h_2+h_3'=0.$ On the other hand, by assumption \eqref{heights_compatibility}, $h_1+h_2+h_3=0,$ and thus, $h_3=h_3'.$  
\end{proof}

Now we are ready to construct parallel networks.

\subsection{Proof of Theorem \ref{teo:reconstruction}} 

{\it Step 1.} For each $i=1,\ldots,n,$ let $L_i$ be the straight line parallel to $S_i$ and satisfying $H(S_i,L_i)\cdot \nu_{S_i} = h_i.$ 

We define subsets $T_i$ of $L_i$ as follows. First consider any half-line $S_i$ of $\sS$ and let $S_j$ be any other segment/half-line of $\sS$ having a common endpoint with $S_i.$ Then the lines $L_i$ and $L_j$ intersect at a unique point separating both lines into two half-lines. Let $T_i\subset L_i$ be the one parallel to $S_i.$ By assumption $h_i=0$ so that by construction both $T_i$ and $S_i$ lie on the same line $L_i.$ 
Thus, by parallelness, $T_i\Delta S_i$ is bounded. 

Next, let $S_j$ be a segment and  $S_i,S_j,S_k$ be a polygonal line (in the same order). Since the angles between the common points of $S_i,S_j$ and $S_j,S_k$ are $120^o,$ the lines $L_i$ and $L_k$ cut from $L_j$ a segment, which we call $T_j.$ Notice that a priori we do not know if $T_j$ is parallel to $S_j$ (because it could be oriented oppositely). We repeat this argument with each segment $S_j$ of $\sS$ and construct all segments $T_j\subset L_j.$ 
\smallskip

{\it Step 2.} Let us find some estimates for $\{T_i\}.$ First observe that by construction if $S_i$ and $S_j,$ $i\ne j,$ have a common vertex, then so do $T_i$ and $T_j.$ Let $\{X\}=S_i\cap S_j$ and $\{Y\}=T_i\cap T_j.$ By Lemma \ref{lem:distan_vertex} and the definition of $\Delta_2,$
$$
|XY| = \tfrac{2}{\sqrt3}\sqrt{|h_i|^2+|h_j|^2-h_ih_j} \le 2\max\{|h_i|,|h_j|\} \le 2\Delta_2 \le \tfrac{1}{3}\min_{S_k\cap S_l=\emptyset}\, d(S_k,S_l).
$$
In particular, 
\begin{equation}\label{samouverennost0}
d(S_i,T_i),d(S_j,T_j)\le |XY| \le \tfrac{1}{3}\min_{S_k\cap S_l=\emptyset}\, d(S_k,S_l).
\end{equation}

We claim that if $S_i\cap S_j=\emptyset,$ then $T_i\cap T_j=\emptyset.$ Indeed, by \eqref{samouverennost0} and the triangle inequality if $S_i\cap S_j=\emptyset,$ then 
$$
d(S_i,S_j) \le d(S_i,T_i) + d(T_i,T_j)+d(T_j,S_j) \le d(T_i,T_j) + \tfrac{2}{3}d(S_i,S_j),
$$
and hence, 
$
d(T_i,T_j) \ge \tfrac{1}{3}d(S_i,S_j)>0.
$
\smallskip

{\it Step 3.} Let $\Gamma:=S_{i_1}\cup\ldots\cup S_{i_m}$ be any polygonal curve of $\sS$ (recall that $\sS$ is a finite union of polygonal curves). We claim that the union $\Gamma':=S_{i_1}\cup\ldots\cup S_{i_m}$ is also a polygonal curve. Indeed, by the definition of $\Delta_1$ and assumption (c) we can apply Lemma \ref{lem:injectivity} to get that the union $\Gamma'':=S_{i_2}\cup\ldots\cup S_{i_{m-1}}$ is a polygonal curve without self-intersections. By step 2 we know that $T_{i_1}$ resp. $T_{i_m}$ have a common vertex with only $T_{i_2}$ resp. $T_{i_{m-1}},$ and both $T_{i_1}$ and $T_{i_m}$ does not intersect the interior of $\Gamma''.$ 

Now if $S_{i_1}\cap S_{i_m} = \emptyset,$ then by step 2, $T_{i_1}\cap T_{i_m} = \emptyset,$ and hence, $\Gamma'$ is injective. On the other hand, if $S_{i_1}$ and $S_{i_m}$ have a common vertex so that $\Gamma$ is a closed curve, then by step 2 so do $T_{i_1}$ and $T_{i_m}.$ Since these segments do not intersect the interior of $\Gamma''$ and $T_{i_2}\cap T_{i_{m-1}}=\emptyset$ (in case $m\ge4$), $\Gamma'$ 
is also a closed injective curve.
\smallskip

{\it Step 4.} Let $S_i,S_j,S_k$ form a triple junction $\triplejunction.$ If $\triplejunction$ is a bc-triple junction, then by assumption $h_i=h_j=h_k=0,$ and therefore, by construction $\triplejunction$ is a triple junction of $T_i,T_j,T_k.$ On the other hand, if $\triplejunction$ is a ic-triple junction, then by assumption (b) and Lemma \ref{lem:heights_compatibility} $T_i,T_j,T_k$ also form an ic-triple junction.
\smallskip

{\it Step 5.} Let $\Gamma_1$ and $\Gamma_2$ be any curves of $\sS$ and let $\Gamma_1'$ and $\Gamma_2'$ be two corresponding curves in $\sT$ (defined in step 3). Since $\Gamma_1$ and $\Gamma_2$ can intersect only at the endpoints, by step 2 the interiors of $\Gamma_1'$ and $\Gamma_2'$ cannot have a common point. Thus, $\sT$ is a network parallel to $\sS.$

Since $h_i=0$ if $S_i$ is a half-line or a segment at a bc-triple junction, $\sT\in\Xi(\sS)$. Finally, the uniqueness of $\sT$ follows from the uniqueness of lines $\{L_i\}.$
\qed

\subsubsection{Expression of some quantities of parallel networks by a given set of heights}

Given a simple network $\sS:=\cup_{i=1}^nS_i,$ let $\Delta_1,\Delta_2>0$ be as in \eqref{eq:Delta12}.  Consider arbitrary real numbers $\{h_j\}_{j=1}^{n}$ satisfying assumptions (a)-(c) of Theorem \ref{teo:reconstruction} so that there exists a unique $\bar \sS:=\cup_{i=1}^n \bar S_i \in\Xi(\sS)$ such that $H(S_i,\bar S_i)\cdot \nu_{S_i} = h_i$ for any $i=1,\ldots,n.$
By \eqref{length_change} for any segment $\bar S_i$ of $\bar\sS,$ we have
\begin{equation}\label{sochi_sunbul98r}
\cH^1(\bar S_i) = \cH^1(S_i) - \tfrac{2}{\sqrt3}\sum_{k=1}^n \beta_k h_k,
\end{equation}
where $\beta_k\in \{0,\pm1\}$ and only three of them are nonzero and depend only on the orientation of the segments/half-lines of $\sS.$

Next, let us study the $\phi$-curvature of the segments of $\bar \sS.$ Namely, we claim that there exist $\Delta_3>0$ and real-analytic functions $\{u_i\}_{i=1}^n$ in $(-2\Delta_3,2\Delta_3)^n$ depending only on $\sS$, such that 
\begin{equation}\label{nice_K_k}
|u_i(h_1,\ldots,h_n)| \le \gamma_0\sum_{j=1}^n|h_j|,\quad 
|u_i(h_1',\ldots,h_n')-u_i(h_1'',\ldots,h_n'')| \le \gamma_0\sum_{j=1}^n |h_j'-h_j''|
\end{equation}
for some $\gamma_0=\gamma_0(\sS)>0$ and all $\{h_j,h_j',h_j''\}$ with $|h_j|,|h_j'|,|h_j''|\le \min\{\Delta_1,\Delta_2,\Delta_3\},$ for which
\begin{equation}\label{curvatre_changes}
\kappa_{\bar S_i}^\phi = \kappa_{S_i}^\phi + u_i(h_1,\ldots,h_n).
\end{equation}

Indeed, let $G$ resp. $\bar G$ be a connected graph associated to $\sS$ resp. $\bar \sS$ as in Lemma \ref{lem:graph_triple} containing at least one triple junction. By that lemma we know that all triple junctions are either ic-triple junctions or bc-triple junctions, simultaneously.
If $G$ contains only ic-triple junctions, then by Lemma \ref{lem:preserving_parallelness} so does $\bar G$ and hence, by Lemma \ref{lem:graph_triple} all segments/half-lines $S_i$ and $\bar S_i$ in both $G$ and $\bar G$ have zero $\phi$-curvature, and in this case we define $u_i\equiv0$ in \eqref{curvatre_changes}.
Therefore, we may assume that $G$ contains only ic-triple junctions. Write $G = \cup_{l=1}^m S_{k_l}$ and $\bar G = \cup_{l=1}^m \bar S_{k_l},$ and let the quadratic functions $\psi$ and $\bar\psi$ be defined as in \eqref{function_psi_qara} and associated to $\sS$ and $\bar \sS,$ respectively. By stability, the minimizers $x^0$ and $\bar x^0$ of $\psi$ and $\bar\psi$ lie in $(0,2/\sqrt3)^r,$ where $r$ is the number of the triple junctions in $G,$ and thus, solves the nondegenerate linear systems $\nabla \psi(x^0)=0$ and $\nabla \bar\psi(\bar x^0)=0.$ 
In particular, there exists an analytic function $\Psi$ depending only on $\sS$ such that 
$$
x^0 = \Psi\Big(\tfrac{1}{\cH^1(S_{k_1})},\ldots,\tfrac{1}{\cH^1(S_{k_m} )} \Big),\qquad 
\bar x^0 = \Psi\Big(\tfrac{1}{\cH^1(\bar S_{k_1})},\ldots,\tfrac{1}{\cH^1(\bar S_{k_m} )} \Big).
$$
By \eqref{sochi_sunbul98r} and the analyticity of $\Psi$ there exists $\Delta_3>0$ depending only on $\{\cH^1(S_{k_l})\}_{l=1}^m$ and the structure of $\sS$ such that 
\begin{equation*}
\bar x^0 = x^0 + \widehat \Psi(h_1,\ldots,h_n),\qquad \sup_{1\le j\le n}|h_j|\le \Delta_3,
\end{equation*}
where $\widehat\Psi$ is a real-analytic (vector-valued) function in $(-2\Delta_3,2\Delta_3)^n.$ This representation implies the existence of a family $\{u_{k_l}\}_{l=1}^m$ of real-analytic functions in $(-2\Delta_3,2\Delta_3)^n$ which satisfies \eqref{curvatre_changes}.

Recall that by the $\phi$-curvature-balance condition of Lemma \ref{lem:sum_curvature_balance}, if $S_i,S_j,S_k$ form an ic-triple junction, then
\begin{equation}\label{curvature_valvalva}
(-1)^{\exponent_i} u_i + (-1)^{\exponent_j} u_j +(-1)^{\exponent_k} u_k \equiv0,
\end{equation}
where $\exponent_i,\exponent_j,\exponent_k$ are $0$ or $1$ depending on whether the corresponding segment enters to or exits from the triple junction.

\section{Proof of Theorem \ref{teo:short_time}}
\label{sec:proof_of_theorem}

{\it Step 1.} Fix a simple network $\initialnetwork:=\cup_{i=1}^{n}S_i^0,$ let $\Delta_1,\Delta_2>0$ be as in \eqref{eq:Delta12}, applied with $\sS=\initialnetwork,$ and let $\Delta_3>0$ and real-analytic functions $\{u_i\}_{i=1}^n$ (depending only $\initialnetwork$) be as above satisfying \eqref{nice_K_k}, \eqref{curvatre_changes} and \eqref{curvature_valvalva}. Define
$$
\Delta_0:=\min\{\Delta_1,\Delta_2,\Delta_3\}>0.
$$

For any $T>0$ let $\cB_T$ be the collection of all $n$-tuples $h:=(h_1,\ldots,h_n)$ of continuous functions in $[0,T]$ such that 
\begin{itemize}
\item[(1)] $h_i(0)=0$ for all $i=1,\ldots,n,$

\item[(2)] if $S_i$ is a half-line or a segment at a bc-triple junction, then $h_i\equiv0$ in $[0,T],$ 

\item[(3)] if $S_i^0,S_j^0,S_k^0$ form an ic-triple junction, then 
\begin{equation*}
(-1)^{\exponent_i} h_i + (-1)^{\exponent_j} h_j +(-1)^{\exponent_k} h_k = 0\quad \text{in $[0,T]$},
\end{equation*}
where $\exponent_i,\exponent_j,\exponent_k$ are $0$ or $1$ depending on whether the corresponding segment/half-line enters to or exits from the triple junction.
\end{itemize}
Since the assumptions (1)-(3)  are linear, $\cB_T$ is a Banach space with respect to the norm 
$$
\|h\|_\infty:=\sum\limits_{i=1}^n\,\max_{t\in[0,T]}\,|h_j(t)|.
$$
Let 
$$
T_0:=\tfrac{1}{n}\,\min\Big\{\tfrac{\Delta_0}{1+ \max_i |\kappafi_{S_i^0}| + \gamma_0\Delta_0}, \tfrac{1}{1+\gamma_0}\Big\}>0,
$$
where $\gamma_0$ is as in \eqref{nice_K_k}, and
$$
\cK:=\{h\in\cB_{T_0}:\,\, \|h\|_\infty\le\Delta_0\}.
$$
Clearly, $\cK$ is a closed convex subset of $\cB_{T_0}.$ For any $h\in \cK$ and $i=1,\ldots,n$ let us define
$$
\Phi_i[h](t) = -\int_0^t \Big(\kappafi_{S_i^0} + u_i(h_1(s),\ldots,h_n(s))\Big)\,d s, \qquad t \in [0,T_0].
$$
Note that $\Phi:=(\Phi_1,\ldots,\Phi_n)$ maps $\cK$ into itself. Indeed, clearly, $\Phi_i[h]\in C^0[0,T_0]$ and $\Phi_i[h](0)=0$ for each $i=1,\ldots,n$ and all $h\in\cK.$ If $S_i^0$ is a half-line or a segment at a bc-triple junction, then $\kappa_{S_i^0}^\phi=0$ and by definition $u_i\equiv0,$ hence, $\Phi_i[h]\equiv0$ for all $i=1,\ldots,n.$ 
Next, if $S_i^0,S_j^0,S_k^0$ form a ic-triple junction, then by the $\phi$-curvature-balance condition \eqref{sum_kurvacha0} and \eqref{curvature_valvalva} 
$$
(-1)^{\exponent_i} \Phi_i[h] + (-1)^{\exponent_j} \Phi_j[h] + (-1)^{\exponent_k} \Phi_k[h] \equiv0\qquad\text{in $[0,T]$}.
$$
Furthermore, by the definition of $T_0$,
$$
\|\Phi_i[h]\|_\infty \le \Big(|\kappafi_{S_i^0}| + \gamma_0 \|h\|_\infty\Big)T_0 \le \tfrac{1}{n}\,\Delta_0,\qquad h\in \cK,
$$
and hence, $\|\Phi[h]\|_\infty\le \Delta_0.$ 
In particular, $\Phi[h]\in\cK$ whenever $h\in\cK.$

Finally, since $\kappa_{S_i^0}^\Phi$ is constant, by the second relation in \eqref{nice_K_k}
$$
\|\Phi_i[h] - \Phi_i[\bar h]\|_\infty \le \gamma_0 \|h-\bar h\|_\infty T_0,\quad h,\bar h\in\cK,
$$
and therefore, by the definition of $T_0$
$$
\|\Phi[h] -\Phi[\bar h]\|_\infty \le \tfrac{\gamma_0}{1+\gamma_0}\,\|h-\bar h\|_\infty,\qquad h,\bar h\in \cK,
$$
and $\Phi$ is a contraction in $\cK.$ By the Banach fixed-point theorem, there exists a unique $h\in\cK$ such that $\Phi[h] = h$ in $[0,T_0].$ By the definition of $\Phi$ and the analyticity of $u_i,$ for each $i=1,\ldots,n$ we have $h_i\in C^\infty[0,T_0]$ and 
\begin{equation}\label{mukazzibin}
h_i'= - \kappa_{S_i^0}^\phi - u_i(h_1,\ldots,h_n)\quad \text{in $[0,T_0]$}.
\end{equation}

In view of assumptions (1)-(3) on $\cB_{T_0}$ and inequality $|h_i(t)|\le\Delta_0\le \min\{\Delta_1,\Delta_2\}$ for all $i=1,\ldots,n$ and $t\in[0,T_0]$ we can apply Theorem \ref{teo:reconstruction} to find a network $\sS(t)\in\Xi(\initialnetwork)$ parallel to $\initialnetwork.$ Since $|h_i(t)|\le \Delta_3,$ to compute the $\phi$-curvatures of $S_i(t)$ we can use \eqref{curvatre_changes}, which combined with \eqref{mukazzibin} gives 
$$
h_i' = -\kappa_{S_i(t)}^\phi\quad\text{in $[0,T_0].$}
$$
Since $h_i(0)=0$ for all $i=1,\ldots,n,$ $\sS(0) = \initialnetwork,$ and therefore, $\sS(\cdot)$ is the $\phi$-curvature flow starting from $\initialnetwork.$
The uniqueness of $\{\sS(\cdot)\}$ in $[0,T_0]$ follows from the uniqueness of the fixed point.
\smallskip

{\it Step 2.} Let $\maximaltime$ be the maximal time for which the regular flow $\sS(t)$ starting from $\initialnetwork$ exists for all times $t\in[0,\maximaltime).$ We have two possibilities:

\begin{itemize}
\item $\maximaltime=+\infty,$ i.e., the flow $\sS(t)$ exists for all times $t\ge0;$ 

\item $\maximaltime<+\infty.$ In this case, 
$$
\liminf_{t\nearrow \maximaltime} S_i(t) = 0
$$
for some $i=1,\ldots,n.$ Indeed, otherwise the limit network $\sS(\maximaltime)$ (for instance defined as a Kuratowski limit of sets) would be simple, and the heights $h_i\in C^1[0,\maximaltime].$ 
In particular, we could apply step 1 with $\initialnetwork:=\sS(\maximaltime-\epsilon)$ for sufficiently small $\epsilon>0$ and find $T_0>0$ independent of $\epsilon$ such that the regular flow $\sS(\cdot)$ exists also in $[\maximaltime-\epsilon,\maximaltime + T_0].$ But this contradicts to the maximality of $\maximaltime.$
\end{itemize}
\qed

\begin{remark}\label{rem:nice_symmetry}
The (geometric) uniqueness of a $\phi$-regular curvature flow starting from a simple network implies, remarkably, that the flow \emph{preserves the (axial, rotational and mirror) symmetries} of the initial network.  
\end{remark}

\subsection{Some extensions of Theorem \ref{teo:short_time} to networks with multiple junctions}\label{sec:some_extensions}

In view of the proof of Theorem \ref{teo:short_time}, its assertions remain valid even if the initial network has  junctions with degree $\geq 3$ provided that the concurring segments/half-lines have zero $\phi$-curvature. To this aim let us call an admissible network $\sS$ \emph{simple with multiple junctions} if either segments/half-lines form a triple junction with $120^o$ or a $m\ge4$-junction with zero $\phi$-curvature.

\begin{theorem}\label{teo:chalpak_flow}
Let $\initialnetwork$ be any simple network with multiple junctions. Then there exists $\maximaltime \in (0, +\infty]$ and a unique family $\{\sS(t)\}_{t\in[0,\maximaltime)}$ of parallel networks such that $\sS(\cdot)$ is the $\phi$-curvature flow starting from $\initialnetwork.$ Moreover, if $\maximaltime<+\infty$ then some segment of $\sS(t)$ vanishes as $t\nearrow \maximaltime.$
\end{theorem}

As mentioned earlier, the proof of this theorem runs along the lines of Theorem \ref{teo:short_time}, since the $m\ge4$-junctions do not move, therefore, we omit it.
This theorem in some cases allows to restart the flow even after some segments at the maximal time vanish.

\begin{corollary}[\textbf{Restarting the flow}]\label{cor:restart_of_flow}
Let $\initialnetwork$ be a simple network with multiple junctions and $\{\sS(t)\}_{t\in[0,\maximaltime)}$ be the unique $\phi$-curvature flow starting from $\initialnetwork.$ Assume that the Kuratowski-limit
$$
\sS(\maximaltime):=\lim\limits_{t\nearrow \maximaltime}\,\sS(t)
$$
is well-defined and $\sS(\maximaltime)$ is simple with multiple junctions. Then there exist $T^\ddag \in (0, +\infty]$ and a unique $\phi$-curvature flow $\{\sS(t)\}_{t\in[\maximaltime, T^\ddag)}$ starting from $\sS(\maximaltime).$
\end{corollary}

Notice that, for $t\in [\maximaltime,T^\ddag)$,  the networks $\sS(t)$ are not parallel to $\initialnetwork.$

\begin{remark}
In Example \ref{exa:formation_of_two_triple_junctions} we will see that networks with quadruple junctions can produce a regular $\phi$-curvature flow, in which quadruple junctions do not move. However, they are not stable under small (simple) perturbations -- in physical situations those quadruple junctions can split into two triple junctions to become stable under small perturbations (see Figure \ref{fig:non_minimal}). We expect such a phenomenon to appear, when defining the evolution using the minimizing movement method (see \cite{ATW:1993,BKh:2018,BChKh:2020} in the case of two phases, and in the case of more phases).
\end{remark}

\section{Examples}\label{sec:examples}

\begin{example}\label{exa:formation_of_two_triple_junctions}
Consider the evolution of the simple network $\sS^0$ (with a quadruple junction) in Figure \ref{fig:non_minimal}, where we assume that segments at the quadruple junction have length $a>0.$ 
By criticality, $\sS^0$ does not move, i.e., its  unique $\phi$-curvature evolution (in the sense of Theorem \ref{teo:chalpak_flow}) is constant $\sS(t)\equiv \sS^0$. Now, let us parse the quadruple junction into two triple junctions at distance $2x>0$ (dotted) and denote the obtained network by $\mathcal N^0= \mathcal N^0(x).$ By symmetry and uniqueness of the geometric flow, the horizontal segment does not translate up or down (it has vanishing $\phi$-curvature) and non-horizontal segments have constant curvature $\pm\frac{1}{\sqrt3a}$ (independent of $x$).  Therefore, these segments move linearly to infinity in the direction  of the corresponding half-lines. In particular, in both cases the flow $\{\mathcal N(t)\}_{t \in [0,T^\dag)}$  (given by 	
Theorem \ref{teo:short_time}) uniquely exists with $T^\dag = +\infty$ and in addition $\mathcal N(t)$ is simple for any $t\geq 0$.
\end{example}

In Example \ref{exa:formation_of_two_triple_junctions} the maximal existence time $\maximaltime$ of the flow if infinite, as opposite to the following example.

\begin{wrapfigure}[8]{l}{0.5\textwidth}
\includegraphics[width=0.48\textwidth]{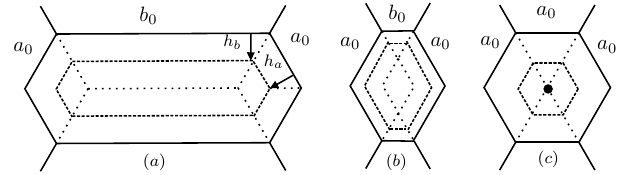}
\caption{\small}
\label{fig:four_junctions09}
\end{wrapfigure}
$\,$ \vspace*{-6mm}
\begin{example}\label{exa:high_multiple_chuvaks}
Let us consider the simple networks 
$\initialnetwork$ in Figure \ref{fig:four_junctions09}, consisting of a hexagon, symmetric with respect to the horizontal axis and clockwise oriented, and of four half-lines starting at the endpoints of the horizontal segments. Denoting by $b(t)$ and $a(t)$ the length of horizontal and lateral segments of $\sS(t)$, we can compute the heights from $\sS^0$ (whose sidelengths are $b_0$ and $a_0$) as
$$
h_a = -\tfrac{\sqrt3}{2}\,\tfrac{a_0+b_0-a-b}{2},\qquad h_b =- \tfrac{\sqrt3}{2}\,(a_0-a),
$$
and the $\phi$-curvature of all segments is equal to $\frac{2}{\sqrt3(b+2a)}.$ Thus, the $\phi$-curvature equation \eqref{mcf_heights} is equivalent to the system
$$
\begin{cases}
\tfrac{\sqrt3}{4} \,(a'+b') = -\tfrac{2}{\sqrt3(b+2a)},\\[2mm]
\tfrac{\sqrt3}{2} \,a' = -\tfrac{2}{\sqrt3(b+2a)},
\end{cases}
$$
which implies $b-a = b_0-a_0.$

\begin{itemize}
\item Consider Figure \ref{fig:four_junctions09} (a), where $b_0>a_0.$ In this case $b(t)>a(t)$ for all times $t\in [0,\maximaltime)$, and hence,  $a(t)\to 0^+$ as $t\nearrow \maximaltime.$ At the maximal time the network  $\sS(\maximaltime)$ is the union of a horizontal segment of length $b_0-a_0$ and four half-lines starting from the endpoints of this segment. Clearly, $\sS(\maximaltime)$ is critical and admissible.

\item In Figure (b), $b_0<a_0,$ and hence, $b(t)<a(t)$ for all $t\in[0,\maximaltime)$. Then $b(t)\to 0^+$ as $t\nearrow \maximaltime$ and at the maximal time the network  $\sS(\maximaltime)$ is the union of a rhombus and four half-lines starting from the vertical vertices of the rhombus. 
Clearly, $\sS(\maximaltime)$ is critical and admissible.

\item In Figure (c), $b_0=a_0,$ and hence, $b(t)=a(t)$ for all $t\in[0,\maximaltime)$. Thus, $a(t)\to 0^+$ as $t\nearrow \maximaltime$ and at the maximal time the hexagon shrinks to a point, and the limiting network $\sS(\maximaltime)$ is a minimal conical admissible network forming an ``X''-type quadruple junction. Later in Section \ref{sec:homo_shrinkos} we will show that this evolution is indeed a self-shrinker.
\end{itemize}
\end{example}

\begin{wrapfigure}[9]{l}{0.35\textwidth}
\vspace*{-6mm}
\includegraphics[width=0.33\textwidth]{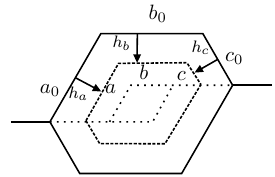}
\caption{\small}
\label{fig:bad_finish}
\end{wrapfigure}
In Figure \ref{fig:four_junctions09} (a) at the maximal time two horizontal segments of the hexagon collapse to the same segment, which has therefore {\it multiplicity two}. 

\begin{example}\label{ex:hexagon09o1}
Consider the network $\sS_0$ in Figure  \ref{fig:bad_finish}, and assume that the opposite sides of the initial hexagon are equal. Since $\sS_0$ has a mirror symmetry, by Remark \ref{rem:nice_symmetry} so does its unique $\phi$-curvature flow $\sS(t).$ In the notation of Figure  \ref{fig:bad_finish} the $\phi$-curvature equation is 
\begin{equation}\label{shdz368fnnm}
\begin{cases}
h_a' = h_c'=\frac{2}{\sqrt3(a+c)},\\
h_b' = \frac{2}{\sqrt3 b}.
\end{cases}
\end{equation}
Moreover, it is possible to check that $a_0-c_0=a-c,$ $h_a=h_c = \frac{b_0+c_0-b-c}{2},$ and therefore from \eqref{shdz368fnnm} we get 
$$
(2a-a_0+c_0)(a'+b') = 2a'b.
$$
Thus, representing $b=F(a)$ we can write this equality as 
$$
F'(a) = \tfrac{2}{2a-a_0+c_0}\,F(a) - 1,
$$
which has the unique solution 
\begin{equation}\label{klientirzs}
b=F(a) = (2a-a_0+c_0) \Big(\tfrac{b_0}{a_0+c_0} + \ln\tfrac{a_0+c_0}{2a-a_0+c_0}\Big).
\end{equation}

(a) First consider the case $a_0>c_0.$ Then $a_0-c_0=a-c,$ $a\ge a_0-c_0>0$ and hence $a(\maximaltime)>0$ and  $b(\maximaltime)>0.$ Therefore, $c(\maximaltime)=0$ and $\sS(\maximaltime)$ is a union of two half-lines and a parallelogram, which is noncritical.

(b) In case $a_0=c_0,$ $\initialnetwork$ has a horizontal symmetry and \eqref{klientirzs} is represented as 
$$
b=2a\Big(\tfrac{b_0}{2a_0} + \ln a_0\Big) - 2a\ln a,
$$
and hence $a(\maximaltime)=b(\maximaltime)=0,$ i.e., $\sS(t)$ converges to the straight line (the hexagon disappears).
\end{example}

Next we analyse some examples of simple  networks with multiple junctions.

\begin{example}\label{exa:two_cases}
Consider the three situations  in Figure \ref{fig:hexagon_with_quad}.

\begin{itemize}
\item  In case (a), at time $\maximaltime>0$ the two lateral segments adjacent to $S_1(t)$ disappear first and thus, $\sS(\maximaltime)$ looses $120^o$-condition  at both quadruple junctions. Similar situation happens in case (c),  where only one of the lateral segments (provided that they are short enough) disappears at time $\maximaltime >0$, destroying the $120^o$-condition on the right quadruple junction. In both cases $\sS(\maximaltime)$ is not simple anymore.
 
\begin{figure}[htp!]
\vspace*{-3mm}
\includegraphics[width=0.9\textwidth]{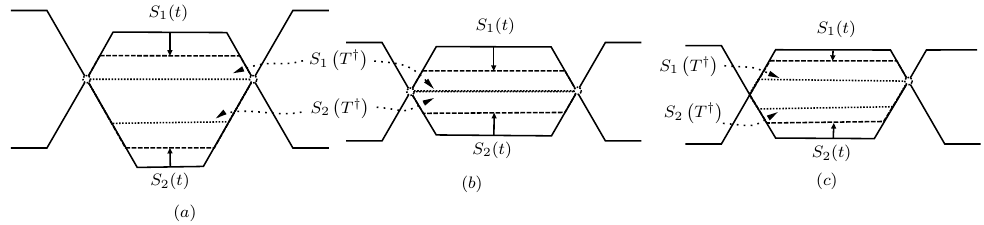}
\caption{\small}
\label{fig:hexagon_with_quad}
\end{figure}

\item In case (b), $\initialnetwork$ is axially symmetric with respect to the  horizontal and vertical axes. Since the flow preserves those symmetries, as $t \nearrow \maximaltime$ both segments $S_1(t)$ and $S_2(t)$ converge to the horizontal segment connecting the two quadruple junctions. Thus, $\sS(\maximaltime)$ becomes a network having all axial symmetries of $\initialnetwork,$ but having two triple junctions connected by a segment. 
\end{itemize}
\end{example}

\begin{wrapfigure}[13]{l}{0.3\textwidth}
\vspace*{-3mm}
\includegraphics[height=5cm]{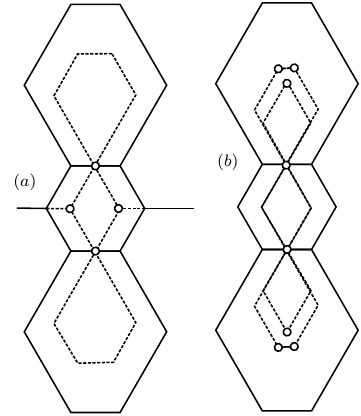}
\caption{\small}
\label{fig:three_hexaooons}
\end{wrapfigure}

In Example \ref{exa:two_cases} 
we observed the transition  of quadruple junctions   into triple junctions. Now we analyze the converse situation.

\begin{example}
Consider the networks in Figure  \ref{fig:three_hexaooons}, differing each other only by the two horizontal half-lines starting from the vertices of the mid-hexagon in the left figure. Both networks are simple  and axially symmetric. Therefore $\sS(\maximaltime)$ becomes axially symmetric, but the two triple junctions above and two triple junctions below join forming
 quadruple junctions (two horizontal segments of the mid hexagon disappear at $\maximaltime$). However, in (a) the quadruple junctions are linked to the triple junctions by a segment which makes nonzero the $\phi$-curvature of those segments, and thus, invalidating the simpleness.  Whereas in (b) $\sS(\maximaltime)$ becomes simple with two quadruple junctions, but not parallel to the initial one. Thus, we may continue the flow after $\maximaltime$, until some other maximal time $T^{\ddag}>\maximaltime,$ at which two horizontal segments of pentagons above and below disappear, and hence, $\sS(T^\ddag)$ ceases to be simple.
\end{example}

\section{Homothetically shrinking solutions with one bounded component}\label{sec:homo_shrinkos}

All homotheties we consider have center at the origin of the coordinates.

\begin{definition}[\textbf{Self-shrinker}]
A family $\{\sS(t)\}_{t\in[0,T)}$ of admissible polygonal networks evolving by $\phi$-curvature is called a \emph{homothetically shrinking solution} starting from $\initialnetwork$, if there exists a strictly decreasing function $r:[0,T)\to[0,1]$ such that 
$$
\lim\limits_{t\to T^-} r(t) = 0
\quad\text{and}\quad
\sS(t) = r(t)\initialnetwork\,\,\text{for all $t\in[0,T).$}
$$
For shortness, we call $\initialnetwork$ a \emph{self-shrinker}.
\end{definition}

In this section we classify the homothetically shrinking solutions starting from a polygonal admissible network $\initialnetwork,$ partitioning $\R^2$ into phases with only one bounded component consisting of a convex hexagon $\bA$ (see Figure \ref{fig:all_shrinkers}); when necessary, the hexagon is parametrized counterclockwise. By definition, any half-line of a self-shrinker must lie on a straight line passing through the origin and, being $\bA$ convex, the homothety center cannot be in its exterior. Furthermore, if the origin is located in the (relative) interior of some segment $S=[XY]$ of $\bA,$ then by the self-similarity, that segment does not move (it has zero $\phi$-curvature). In this case, by the convexity of $\bA$ at least one endpoint of $S,$ say $X,$ should be the starting point of a half-line, lying on the same line with $S.$ Then one can readily check that another segment of $\bA$ sharing common point $X$ with $S$ should also have zero $\phi$-curvature, i.e., it does not move, which falsifies the self-similarity of $\initialnetwork.$ Therefore, we only have to deal with two situations: the origin is either in the interior of $\bA$ (see Figure \ref{fig:all_shrinkers}) or it is one of the vertices (see Figure \ref{fig:masheynik09}).

Recall that without half-lines a network $\initialnetwork=\p\bA$ is a self-shrinker if and only if $\bA=\Wulff_{R_0}$ for some $R_0>0.$ Since the (signed) height is $\frac{\sqrt3}{3}(R_0 - R(t))$ and the $\phi$-curvature of any segment of $\bA$ is $\frac{2}{\sqrt3 R(t)},$ where $\frac{2}{\sqrt3}$ is the sidelength of the 
unit Wulff shape,  the radius $R(t)$ of $\sS(t)=\p \Wulff_{R(t)}$ is given by 
$$
R(t):=\sqrt{R_0^2 - \tfrac{8}{3}t},\quad t\in[0,\tfrac{3R_0^2}{8}),
$$
so that the function $r(t):=\sqrt{1-\frac{8t}{3R_0^2}}$ satisfies $\sS(t)=r(t)\initialnetwork.$

The main result of this section is the following classification theorem (see Figure \ref{fig:self_shrinko}).

\begin{theorem}[\textbf{Classification of shrinkers}]
Let a network $\initialnetwork$ be a union of an admissible convex hexagon $\bA$ and $n\ge1$-half-lines.

\begin{itemize}
\item[\rm (a)] Suppose the origin is an interior point of $\bA.$ Then $n\le 5$ and up to a rotation and a mirror reflection:
\begin{itemize}
\item if $n=1$ (Brakke-type spoon), then the sidelengths of $\bA$ are  $a, 2a, a, a,2a, a$ for some $a>0,$ and the unique half-line starts from a vertex of $\bA$ at which both adjacent segments 
have length $a.$ In this case the origin divides the largest diagonal of $\bA$ in proportion $1:2$ starting from the half-line, 
and
\begin{equation}\label{eq:S_t_r_t}
\sS(t)=r(t)\initialnetwork \qquad \text{with}\quad r(t)=\sqrt{1-\tfrac{2t}{3a_o^2}},\quad t\in [0,\tfrac{3a_o^2}{2});
\end{equation}

\item if $n=2,$ then the sidelengths of $\bA$ are approximately $a,2.94771a,$ $2.33925a,$ $1.60847a,$ $2.33924a,$ $2.94772a$ for some $a>0,$ and the two half-lines start from the endpoints of the segment of length $a.$ Moreover \eqref{eq:S_t_r_t} holds;

\item if $n=3,$ then $\bA$ is a regular hexagon and the three half-lines form a $120^o$-angle. Moreover \eqref{eq:S_t_r_t} holds;

\item if $n=4,$ then $\bA$ is a regular hexagon and the (prolongations of) four half-lines form an $X$. Moreover
\begin{equation}\label{eq:S_4_9t_r_t}
\sS(t)=r(t)\initialnetwork
\qquad \text{with}\quad  
r(t)=\sqrt{1-\tfrac{4t}{9a_o^2}},\quad t\in [0,\tfrac{9a_o^2}{4});
\end{equation}
\item if $n=5,$ then $\bA$ is a regular hexagon. Moreover \eqref{eq:S_4_9t_r_t} holds.
\end{itemize}

\item[\rm(b)]  Suppose the origin is a vertex (say $A_1$) of $\bA.$ Then $\bA$ is homothetic to a Wulff shape, at $A_1$ there are at least one half-line parallel to adjacent segments at $A_1$, and there is another half-line at 
vertex $A_4$ opposite to $A_1$ (see Figures \ref{fig:masheynik09} (a) and (c)).
\end{itemize}

\end{theorem}

\begin{wrapfigure}[18]{l}{0.44\textwidth}
\includegraphics[width=0.42\textwidth]{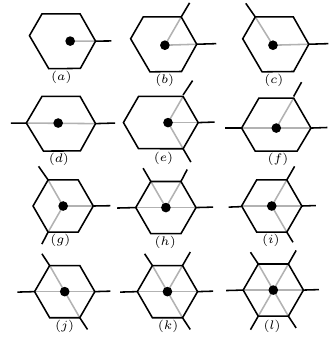} 
\caption{\small}
\label{fig:all_shrinkers}
\end{wrapfigure}
\subsection{Homothety center inside $\bA$}
\label{subsec:homothety_center_inside_the_hexagon}

In this section we assume that the homothety center -- the origin -- is an interior point of $\bA$ so that all straight lines containing the half-lines pass through the origin and bisect the corresponding angles of $\bA.$ Accordingly, we write $\initialnetwork=\cup_{i=1}^n S_i,$ where $S_1,\ldots,S_6$ are the sides of $\bA$ (counterclockwise order) with lengths $a_1,\dots, a_6$ respectively, and $S_7,\ldots,S_n$ are half-lines, $n \geq 7$ (see Figure \ref{fig:tarelka}). Note that $n\le 12.$

Since the origin is 
inside $\bA$, there might be at most one half-line emanating from each vertex. In particular, there are  at least one and at most six such half-lines. Up to a rotation and mirror reflection, there are twelve possible variants (see Figure \ref{fig:all_shrinkers}); in what follows we characterize which ones are a self-shrinker. 

We use the notation of Figure \ref{fig:tarelka} and, as in Figure \ref{fig:all_shrinkers}, we always assume that $A_1$ is a triple junction of segments $S_1,S_6$ and the half-line $S_7$, and $\p \bA$ is oriented counterclockwise, so that the heights $h_i$ and $\bar h_i$ are nonnegative. Clearly, $\theta_4 + \bar\theta_4=120^o,$ and $\theta_1=\bar\theta_1 = 60^o$ so that $h_1=\bar h_1.$ Let $a_1:=a >0$ and $h_1=\bar h_1=:h\ge0.$ 

\begin{figure}[htp!]
\includegraphics[width=0.48\textwidth]{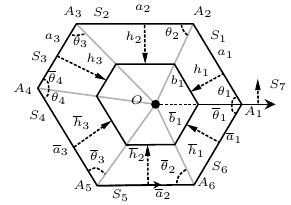}
\caption{\small}
\label{fig:tarelka}
\end{figure}

We proceed as follows. We express the sidelengths $a_i$ and $\bar a_i$ of $\bA$, for $i=2,\dots, 6$, by means of $a$ and the angles $\theta_i$, $\bar \theta_i$ (see \eqref{express_ai_with_a}). Similarly, we represent the remaining heights by means of $h$ 
and $\theta_i,\bar\theta_i$ (see \eqref{represen_hi_with_h}). The driving reason here is that, by the homothety, $\theta_i$ and $\bar\theta_i$ are independent of time, and therefore the equation $h_i' = -{\kappafi_{S_i}} = \frac{c_i}{\cH^1(S_i)}$ (see \eqref{eq:mcf_heights_simplified}) can be rewritten only using $a,h$ and those angles; here $c_i$ are some numbers depending only on the structure of the network (see \eqref{no_triple}-\eqref{five_triple} below), which somehow account for how many triple junctions ``linked'' to each other are present. These curvature equations for segments  $S_1,\ldots,S_6$ imply a system with respect to $\theta_i$ and $\bar \theta_i$, see \eqref{bizning_kamonlar}. Since $c_i$ changes as the number of half-lines in the network increases, we need to analyze each network in Figure \ref{fig:all_shrinkers} separately.

\subsection*{Preliminary computations}
Now we enter into the details. Let us introduce the notation 
$$
\F_i:=\frac{\sin(60^o+\theta_i)}{\sin\theta_i},\quad
\bar\F_i:=\frac{\sin(60^o+\bar \theta_i)}{\sin \bar\theta_i},\quad 
i=1,2,3,4.
$$
Clearly $\F_i>0$ and $\bar\F_i>0$, since $\theta_i\in (0,120^o)$. We express $a_i,$ $\bar a_i,$ $h_i$ and $\bar h_i$ by means of $a,$ $h,$ $\F_i$ and $\bar\F_i.$ We have:
\begin{equation}\label{represen_hi_with_h}
\begin{cases}
h_1= h,\\
h_2=\frac{\sin\theta_2}{\sin(60^o+\theta_2)}h = \frac{h}{\F_2},\\
h_3=\frac{\sin\theta_3}{\sin(60^o+\theta_3)}\frac{\sin\theta_2}{\sin(60^o+\theta_2)}h = \frac{h}{\F_2 \F_3},
\end{cases}
\quad 
\begin{cases}
\bar h_1= h,\\
\bar h_2=\frac{\sin \bar\theta_2}{\sin(60^o+\bar\theta_2)}h = \frac{h}{\bar \F_2},\\
\bar h_3=\frac{\sin \bar\theta_3}{\sin(60^o+\bar\theta_3)}\frac{\sin \bar\theta_2}{\sin(60^o+\bar\theta_2)}h =\frac{h}{\bar \F_2 \bar \F_3}.
\end{cases}
\end{equation}
Since $\theta_4+\bar\theta_4=120^o$ and $h_3\sin \theta_4 = \bar h_3\sin \bar\theta_4,$ from 
\eqref{represen_hi_with_h} the angles $\theta_i$ satisfy the identities
\begin{equation*}
\frac{\sin\theta_4}{\sin(60^o+\theta_4)} = \frac{\sin(60^o+\bar\theta_4)}{\sin \bar\theta_4},\quad 
\frac{\sin(60^o + \bar\theta_3)}{\sin(60^o + \theta_3)}
\frac{\sin(60^o + \bar\theta_2)}{\sin(60^o + \theta_2)}
=
\frac{\sin\bar \theta_2}{\sin \theta_2}
\frac{\sin\bar \theta_3}{\sin \theta_3}
\frac{\sin\bar \theta_4}{\sin \theta_4}
\end{equation*}
or, in terms of $\F_i$ and $\bar \F_i,$
\begin{equation}\label{dorime_umbaloiqd}
\F_4\bar \F_4=1,\quad \bar \F_3 \bar \F_2= \F_4 \F_3 \F_2,\quad \F_3 \F_2= \bar \F_4 \bar \F_3 \bar \F_2.
\end{equation}

Next, using the law of sines  and the equalities 
$$
\frac{\sin(60^o+\theta_{i+1}-\theta_i)}{\sin\theta_{i+1} \sin\theta_i} = \frac{2}{\sqrt3}(\F_{i+1}\F_i-\F_{i+1}+1),\quad 
\frac{\sin(60^o+\bar\theta_{i+1}-\bar \theta_i)}{\sin \bar\theta_{i+1} \sin \bar \theta_i} = \frac{2}{\sqrt3}(\bar\F_{i+1}\bar \F_i-\bar \F_{i+1}+1),
$$
we represent the sidelengths of $\bA$ as follows:
\begin{equation}\label{express_ai_with_a}
\begin{aligned}
a_1 = &a,\\
a_2 = &\frac{\sqrt3 \sin(60^o + \theta_3-\theta_2)}{2\sin\theta_2\sin(60^o+\theta_3)} a = \frac{\F_3\F_2-\F_3+1}{\F_3}\,a,\\
a_3 = &\frac{\sqrt3 \sin(60^o + \theta_4-\theta_3)}{2\sin(60^o+\theta_3)\sin(60^o+\theta_4)} a = \frac{\F_3\F_4-\F_4+1}{\F_3\F_4}\,a,\\
\bar a_1 = & \frac{\sin\bar\theta_2}{\sin(60^o+\bar \theta_2)}\frac{\sin(60^o + \theta_2)}{\sin \theta_2}a = \frac{\F_2}{\bar\F_2}a,\\
\bar a_2 = & \frac{\sqrt3 \sin(60^o + \bar\theta_3-\bar\theta_2)\sin(60+\theta_2)}{2\sin\theta_2 \sin(60^o+\bar\theta_2) \sin(60^o+\bar\theta_3)} a = \frac{ \bar \F_3 \bar \F_2-\bar\F_3+1}{\F_3\F_4}\,a,\\
\bar a_3 = &\frac{\sqrt3 \sin(60^o + \bar\theta_4-\bar\theta_3)\sin \bar\theta_2 \sin(60^o+\theta_2)}{2\sin\theta_2\sin(60^o+\bar \theta_2)\sin(60^o+\bar \theta_3)\sin(60^o+\bar\theta_4)} a 
\\
= &\frac{\bar\F_3 \bar \F_4- \bar \F_4+1}{\F_3}\,a = \frac{\F_3 + \F_4- 1}{\F_3\F_4}\,a,
\end{aligned}
\end{equation}
where in the last equality we used $\F_4\bar\F_4=1.$

The $\phi$-curvature of $S_i$ is given by
\begin{equation*}
\kappafi_{S_i}
=-\frac{c_i}{\cH^1(S_i)},
\end{equation*}
where $c_i\ge0$ is computed using similar arguments of Section \ref{subsec:comput_curvatures}. Indeed, neglecting for the moment the half-lines and assuming $S_i = S_{i-6}$ for $i>6,$ one can show that $\initialnetwork$ is simple and:
\begin{itemize}
\item if both endpoints $A_i$ and $A_{i+1}$ of $S_i$ are simple vertices of $\initialnetwork$, then
\begin{equation}\label{no_triple}
c_i:=\frac{2}{\sqrt3};
\end{equation}

\item if $S_i$ and $S_{i+1}$ join at a triple junction, and both other endpoints are simple vertices, then 
\begin{equation}\label{one_triple}
c_j = \frac{2\cH^1(S_j)}{\sqrt3(\cH^1(S_i)+\cH^1(S_{i+1}))},\quad j=i,i+1;
\end{equation}

\item if $(S_i,S_{i+1})$ and $(S_{i+1},S_{i+2})$ join at two triple junctions\footnote{I.e. the pairs $(S_i,S_{i+1})$ of segments and one half-line form a triple junction.}, and the other endpoints of $S_i$ and $S_{i+2}$ are simple vertices, then
\begin{equation}\label{two_triple}
c_j = \frac{2\cH^1(S_j)}{\sqrt3(\cH^1(S_i)+\cH^1(S_{i+1}) + \cH^1(S_{i+2}))},\quad j=i,i+1,i+2;
\end{equation}

\item if $(S_i,S_{i+1}),$ $(S_{i+1},S_{i+2})$ and $(S_{i+2},S_{i+3})$ join at three triple junctions, and the other endpoints of $S_i$ and $S_{i+3}$ are simple vertices, then
\begin{equation}\label{three_triple}
c_j = \frac{2\cH^1(S_j)}{\sqrt3\sum_{l=i}^{i + 3}\cH^1(S_l)},\quad j=i,i+1,i+2,i+3;
\end{equation}

\item if $(S_i,S_{i+1}),$ $(S_{i+1},S_{i+2}),$ $(S_{i+2},S_{i+3})$ and $(S_{i+3},S_{i+4})$ join at four triple junctions, and the other endpoints of $S_i$ and $S_{i+4}$ are simple vertices, then
\begin{equation}\label{four_triple}
c_j = \frac{2\cH^1(S_j)}{\sqrt3\sum_{l=i}^{i + 4}\cH^1(S_l)},\quad j=i,i+1,i+2,i+3,i+4.
\end{equation}
Note that if there are only four triple junctions, then 
$
c_{i+5} = \frac{2}{\sqrt3};
$

\item if $(S_i,S_{i+1}),$ $(S_{i+1},S_{i+2}),$ $(S_{i+2},S_{i+3}),$ $(S_{i+3},S_{i+4})$ and $(S_{i+4},S_{i+5})$ join at five triple junctions, and the other endpoints of $S_i$ and $S_{i+5}$ are simple vertices, then
\begin{equation}\label{five_triple}
c_j = \frac{2\cH^1(S_j)}{\sqrt3 \sum_{l=i}^{i + 5}\cH^1(S_l)},\quad j=i,i+1,i+2,i+3,i+4,i+5.
\end{equation}
\end{itemize}
These computations show that in homothetic networks these numbers do not change.  In particular, with the notation of Figure \ref{fig:tarelka} (d), the $\phi$-curvature equation for segments $S_i$ and $S_{7-i}$ in the homothetically shrinking solution can be represented as 
\begin{equation*}
h_i'=-\kappafi_{S_i} = \frac{c_i}{a_i}\quad\text{resp.}\quad \bar h_i'=-\kappafi_{S_{7-i}} = \frac{c_{7-i}}{\bar a_i},\quad i=1,2,3.
\end{equation*}
By homothety, the angles $\theta_i$ and $\bar\theta_i$ are independent of time, and hence, using the equalities \eqref{represen_hi_with_h} and the representations of $a_i$ with $a$ in \eqref{express_ai_with_a} we obtain six equalities 
\begin{equation}\label{bizning_kamonlar}
ah' = \gamma_1=\gamma_2=\gamma_3=\gamma_4=\gamma_5=\gamma_6,
\end{equation} 
where
\begin{equation*}
\gamma_i:=
\begin{cases}
c_1 & \text{for $i=1,$}\\[2mm]
\frac{2c_2\sin(60^o+\theta_2)\sin(60^o+\theta_3)}{\sqrt3 \sin(60^o+\theta_3-\theta_2)} = \frac{c_2\F_2\F_3}{\F_2\F_3-\F_3 +1} & \text{for $i=2,$}\\[2mm]
\frac{2c_3\sin(60^o+\theta_2)\sin^2(60^o+\theta_3)\sin(60^o+\theta_4)}{\sqrt3 \sin\theta_2\sin\theta_3\sin(60^o+\theta_4-\theta_3)} = \frac{c_3\F_2\F_3^2\F_4}{\F_4\F_3-\F_4+1} & \text{for $i=3,$}\\[2mm]
\frac{2c_4\sin\theta_2 \sin^2(60^o+\bar\theta_2) \sin^2(60^o+\bar\theta_3) \sin(60^o+\bar\theta_4)}{\sqrt3\sin^2\bar\theta_2 \sin\bar\theta_3 \sin(60^o+\theta_2)\sin(60^o + \bar \theta_4 - \bar\theta_3)} = \frac{c_4\F_2\F_3^2\F_4}{\bar \F_4\bar\F_3-\bar\F_4+1} & \text{for $i=4,$}\\[2mm]
\frac{2c_5\sin\theta_2\sin^2(60^o+\bar\theta_2)\sin(60^o+\bar\theta_3)}{\sqrt3 \sin\bar\theta_2 \sin(60^o+ \theta_2)\sin(60^o+\bar\theta_3 - \bar\theta_2)}  = \frac{c_5\bar \F_2\F_3\F_4}{\bar\F_3\bar\F_2 - \bar\F_3 + 1} & \text{for $i=5,$}\\[2mm]
\frac{c_6\sin\theta_2\sin(60^o+\bar\theta_2)}{\sin\bar\theta_2\sin(60^o+\theta_2)} =\frac{c_6\bar \F_2}{\F_2} & \text{for $i=6.$}
\end{cases}
\end{equation*}
These equations provide a necessary and sufficient condition for $\initialnetwork$ to be a self-shrinker. 

Now, we examine the networks in Figure \ref{fig:all_shrinkers}.

\subsection*{One half-line case}
Let us study self-shrinking Brakke-type spoons $\initialnetwork$ as in Figure \ref{fig:all_shrinkers} (a), i.e., in the notation of Figure \ref{fig:tarelka}, the only half-line starts from the vertex $A_1$ of the hexagon $\bA$. 
Since $\theta_1=\bar\theta_1=60^o,$
$$
\F_1=\bar\F_1=1,\quad \F_4\bar\F_4=1,\quad \F_4\F_3\F_2=\bar\F_3 \bar\F_2,\quad  \bar\F_2\bar \F_3\F_4=\F_2\F_3.
$$
The sidelengths of $\bA$ are represented by means of $\F_i$ and $\bar\F_i$ as 
$$
a_1= a,\quad 
a_2= \frac{\F_3\F_2-\F_3+1}{\F_3}a, \quad 
a_3= \frac{\F_4\F_3-\F_4+1}{\F_3\F_4}a, \quad
\bar a_1= \frac{\F_2}{\F_2}a
$$
and
$$
\bar a_2= \frac{\bar \F_3\bar \F_2-\bar\F_3+1}{\F_3\F_4}a=
\frac{\F_4\F_3\F_2-\bar\F_3+1}{\F_3\F_4}a, \quad
\bar a_3= \frac{\bar \F_4 \bar \F_3-\bar\F_4+1}{\F_3}a=
\frac{\F_4+\bar\F_3-1}{\F_3\F_4}a.
$$
Next, recalling the definitions of $c_i$ in \eqref{no_triple} and \eqref{one_triple}, we obtain the following representations of $\gamma_i$:
$$
\gamma_1 = \gamma_6 = \frac{2}{\sqrt3}\frac{\bar \F_2}{\F_2 + \bar \F_2},\quad 
\gamma_2 = \frac{2}{\sqrt3} \frac{\F_2 \F_3}{\F_2\F_3 - \F_3 + 1}, \quad 
\gamma_3 = \frac{2}{\sqrt3} \frac{\F_2 \F_3^2 \F_4}{\F_3\F_4 - \F_4 + 1},
$$
and 
$$
\gamma_4 = \frac{2}{\sqrt3} \frac{\F_2 \F_3^2 \F_4}{\bar\F_3 \bar\F_4 - \bar\F_4 + 1},\quad 
\gamma_5 = \frac{2}{\sqrt3} \frac{\bar\F_2 \F_3\F_4}{\bar\F_2 \bar\F_3 - \bar\F_3 + 1}.
$$
First, from the equality $\gamma_1=\gamma_2$ we deduce $\bar \F_2 = \frac{\F_2^2\F_3}{1-\F_3}.$ Moreover, from $\gamma_3=\gamma_4$ we get $\bar\F_3 = \F_3\F_4^2 - \F_4^2 + 1$ and from $\gamma_2=\gamma_3$ we get $\F_4 = \frac{1}{\F_2\F_3^2 - \F_3^2 + 1}.$ Inserting the values found of $\bar \F_2$ and $\bar\F_3$ in the equality $\gamma_2 = \gamma_5$ we obtain 
$$
\frac{1}{\F_2\F_3-\F_3+1} = \frac{\F_3\F_4}{(1-\F_3)(\F_4-\F_4\F_3+\F_3\F_2)}.
$$
This equality is equivalent to $\F_4(1-2\F_3 +\F_3^2) = \F_3^2\F_2^2.$ Inserting here the earlier values found of $\F_4$ we obtain the following fourth order equation:
\begin{equation}\label{first_equa}
\F_3^4(\F_2^3-\F_2^2) + \F_3^2(\F_2^2-1) + 2\F_3-1 = 0.
\end{equation}
On the other hand, inserting the values of $\bar\F_2,\bar \F_3$ and $\F_4$ in the equation $\F_4\F_3\F_2=\bar\F_3 \bar\F_2$ we obtain another fourth order equation:
\begin{equation}\label{second_equa}
\F_3^4(\F_2^3 - 2\F_2^2 + \F_2) + \F_3^3(\F_2-1) + \F_3^2(2\F_2^2 -3\F_2+1) +\F_3(\F_2+1) -1 = 0.
\end{equation}
Subtracting \eqref{second_equa} from \eqref{first_equa} we get
\begin{equation}\label{thid_equa}
\F_3^2(\F_2-\F_2^2) + \F_3^2(\F_2-1) + \F_3(\F_2^2 - 3\F_2+2) + (\F_2-1) = 0.
\end{equation}
From this equality we deduce $\F_2 = 1.$ Inserting this in \eqref{first_equa} we find $\F_3=\frac12.$ Let us check whether there are other solutions. Factoring out \eqref{thid_equa} the term $\F_2-1$ we get 
\begin{equation}\label{fourth_equa}
-\F_3^2\F_2 + \F_3^2 + \F_3(\F_2-2) + 1 = 0.
\end{equation}
Adding this to \eqref{first_equa} we establish 
$$
\F_3^4(\F_3^3-\F_3^2) - \F_3^3\F_2 + \F_3^2\F_2^2 + \F_2\F_3 = 0.
$$
Recalling $\F_2,\F_3>0,$ the last equation is equivalent to 
\begin{equation}\label{fifth_eq}
\F_3^3(\F_2^2-\F_2) - \F_3^2 + \F_2\F_3 +1 = 0.
\end{equation}
Subtracting \eqref{fourth_equa} from \eqref{fifth_eq} gives 
\begin{equation}\label{first_square_eq}
\F_3^2\F_2^2 - 2\F_3+2 = 0. 
\end{equation}
Multiplying \eqref{fourth_equa} by $-2$ and adding to \eqref{first_square_eq} gives 
$$
2\F_3^3\F_2 +\F_3^2 (\F_2^2 - 2) - 2\F_3(\F_2-1)=0,
$$
and hence,
\begin{equation}\label{second_square_eq}
2\F_3^2\F_2 +\F_3 (\F_2^2-2)-2(\F_2-1)=0. 
\end{equation}
Now multiplying \eqref{first_square_eq}  by $\F_2-1$ and adding to  \eqref{second_square_eq} we get 
$$
\F_3^2 (\F_2^3-\F_2^2+2\F_2)+ \F_3(\F_2^2-2\F_2)=0,
 $$
and thus, 
$$
\F_3 = \frac{2-\F_2}{\F_2^2-\F_2+2}.
$$
Inserting this expression of $\F_3$ in \eqref{first_equa} and simplifying the similar terms we get 
$$
\frac{(2-\F_2)^2}{\F_2^2-\F_2+2} + 2 = 0,
$$
which does not have real solutions. 

Thus, we have a unique solution
$$
\F_1=\bar\F_1 =1,\quad \F_2=\bar\F_2 = 1,\quad \F_3=\bar\F_3=\frac{1}{2},\quad \F_4=\bar\F_4=1.
$$
Then the corresponding angles are 
$$
\theta_1=\theta_2=\theta_4=\bar\theta_1 =\bar\theta_2= \bar \theta_4= 60^o,\quad \theta_3 = \bar\theta_3 = 90^o
$$
and the sidelengths are 
$$
a_1 = a_3=\bar a_1 = \bar a_3 =  a,\quad a_2=\bar a_2=2a.
$$
Thus, $\bA$ is that hexagon, symmetric with respect to horizontal axis, whose three consecutive segments in the upper half-plane 
have lengths $a,2a$ and $a.$ Note that the homothety center (the origin) is not the midpoint of $[A_1A_4],$ rather, it divides this segment in proportion $1:2$ (starting from $A_1$). 

Assuming initially $\initialnetwork$ have lengths $a_o:=\cH^1(S_1^0)$ and $2a_o = \cH^1(S_2^0)$, let us find the function $r(\cdot)$ satisfying $\sS(t)=r(t)\initialnetwork.$ Since the triangle $A_1OA_2$ is equilateral, $h(t) = \frac{\sqrt3}{2}(a_o-a(t))$ and $\kappafi_{S_1(t)} = \frac{1}{\sqrt3a(t)}.$ Thus, the $\phi$-curvature equation $h'(t)=-\kappafi_{S_1(t)}$ is expressed as 
$$
\frac{\sqrt3}{2}a' = \frac{1}{\sqrt3 a}\quad\text{so that}\quad a(t) = \sqrt{a_o^2 - \frac{2}{3}t},\quad t\in[0,\tfrac{3a_o^2}{2}).
$$
Whence the function $r(t)=\sqrt{1-\frac{2t}{3a_o^2}}$ satisfies $\sS(t)=r(t)\initialnetwork.$

\subsection*{Two half-lines: case 1.} 
Let us check whether the network $\initialnetwork$ in Figure \ref{fig:all_shrinkers} (b) is a self-shrinker. With the notation of 
Figure \ref{fig:tarelka}, the half-lines of $\initialnetwork$ 
start from $A_1$ and $A_2$, so that $\theta_1 = \bar\theta_1= \theta_2 = 60^o.$ Thus, \eqref{dorime_umbaloiqd} is represented as 
$$
\F_1=\bar\F_1=\F_2=1,\quad \bar\F_4\F_4=1,\quad  \bar\F_3\bar\F_2=\F_4\F_3.
$$
In this case 
$$
a_1=a,\quad 
a_2=\frac{a}{\F_3},\quad
a_3=\frac{\F_4 \F_3-\F_4 +1}{\F_4\F_3}a,
$$
and
$$
\bar a_1 = \frac{a}{\bar\F_2},\quad 
\bar a_2=\frac{\F_4\F_3-\bar\F_3+1}{\F_3\F_4},\quad
\bar a_3=\frac{\F_4 + \bar\F_3-1}{\F_4\F_3},
$$
and hence, by the definition of $c_i$ in \eqref{no_triple} and \eqref{two_triple}, 
$$
\gamma_1=\gamma_2=\gamma_6 = \frac{2}{\sqrt3}\frac{\bar\F_2\F_3}{\bar\F_2\F_3+\F_3+\bar\F_2},\quad 
\gamma_3=\frac{2}{\sqrt3}\frac{\F_4\F_3^2}{\F_4\F_3-\F_4+1},
$$
and
$$
\gamma_4=\frac{2}{\sqrt3}\frac{\F_4\F_3^2}{\bar\F_4\bar\F_3-\bar\F_4+1},\quad 
\gamma_5=\frac{2}{\sqrt3}\frac{\bar\F_2\F_3\F_4}{\bar\F_4\bar\F_3-\bar\F_3+1}.
$$
From the equalities $\gamma_1=\gamma_3$ and $\gamma_3=\gamma_4$ as well as $\bar\F_4=\frac{1}{\F_4}$ we obtain
$$
\bar\F_2 = \frac{\F_3^2\F_4}{1-(1+\F_3^2)\F_4} \quad \text{and} \quad \F_3 = 1-(1-\F_3)\F_4^2.
$$
Inserting these relations in the equality $\gamma_3=\gamma_5$ we get
$$
\frac{1}{\F_3\F_4 - \F_4 + 1} = \frac{\F_3^2\F_4}{(1-\F_4-\F_3^2\F_4)(\F_4+\F_3 - \F_3\F_4)},
$$
and hence
$$
\F_4 = \frac{\F_3^3 + \F_3^2 +\F_3-1}{\F_3^3 - \F_3^2 +\F_3-1}.
$$
Inserting the expressions of $\bar\F_2,$ $\bar\F_3$ and $\F_4$ in the identity $\bar\F_2\bar\F_3 = \F_3\F_4$ gives
$$
(\F_3^3 + \F_3^2 + \F_3-1)^2  +  (\F_3^2 + 2\F_3 - 1)(\F_3^2+1)^2 = 0.
$$
After simplification (recalling that $\F_3>0$), this equation reduces to 
\begin{equation*}
\F_3^4 + 2\F_3^2+2\F_3^2 +2\F_3 -1 = 0
\end{equation*}
which admits a unique positive solution $\F_3\approx 0.33925.$ Then the sidelengths of $\bA$ are defined as
$$
a_1 = a,\quad 
a_2 \approx 2.94771a,\quad 
a_3 \approx 2.33925a,\quad 
a_4 \approx 1.60847a,\quad
a_5 \approx 2.33924a,\quad 
a_6 \approx 2.94772a.
$$
Moreover, the corresponding angles $\theta_i$ and $\bar\theta_i$ are 
$$
\theta_1 = 60^o,\quad 
\theta_2 = 60^o,\quad 
\theta_3 \approx 100.51566^o,\quad 
\theta_4 \approx  77.77741^o,
$$
$$
\bar \theta_1 = 60^o,\quad
\bar \theta_2 \approx  100.51576^o,\quad 
\bar \theta_3 \approx  77.77726^o,\quad 
\bar \theta_4 \approx  42.22259^o.
$$
In particular, $\bA$ has no symmetry. Assuming $a_1=a_o$ for the initial hexagon, let us 
look for a function $r(\cdot)$ satisfying $\sS(t) = r(t)\initialnetwork.$ Since the triangle $A_1OA_2$ is equilateral, $h(t) = \frac{\sqrt3}{2}(a_o-a(t))$ and $\kappafi_{S_1(t)} = \frac{1}{\sqrt3a(t)}.$ Thus, the $\phi$-curvature equation $h'(t)=-\kappafi_{S_1(t)}$ is expressed as 
$$
\frac{\sqrt3}{2}a' = \frac{1}{\sqrt3 a}\quad\text{so that}\quad a(t) = \sqrt{a_o^2 - \frac{2}{3}t},\quad t\in[0,\tfrac{3a_o^2}{2}).
$$

Whence the function $r(t)=\sqrt{1-\frac{2t}{3a_o^2}}$ satisfies $\sS(t)=r(t)\initialnetwork.$

\subsection*{Two half-lines: case 2.} 
Let $\initialnetwork$ be as in Figure \ref{fig:all_shrinkers} (c) so that 
with the notation of Figure \ref{fig:tarelka} (d) the half-lines of $\initialnetwork$ start from $A_1$ and $A_3.$ Then $\theta_1=\theta_3=\bar\theta_1=60^o$ so that 
\begin{equation*}
\F_1=\F_3=\bar\F_1=1,\quad \bar\F_3\bar\F_2 = \F_4\F_2. 
\end{equation*}
Then 
$$
a_1=a,\quad a_2=\F_2a,\quad a_3=\frac{a}{\F_4},\quad \bar a_1=\frac{\F_2}{\bar\F_2}a,\quad \bar a_2=\frac{\F_4\F_2-\bar\F_3 +1}{\F_4}a,\quad \bar a_3 = \frac{\bar\F_3\bar\F_4-\bar\F_3+1}{\F_3}a.
$$
Now recalling the definitions of $c_i$ in \eqref{no_triple} and \eqref{one_triple} we compute
$$
\gamma_1=\gamma_6=\frac{2}{\sqrt3}\frac{\bar\F_2}{\F_2+\bar\F_2},\quad 
\gamma_2=\gamma_3=\frac{2}{\sqrt3}\frac{\F_2\F_4}{\F_2\F_4 + 1},
$$
and
$$
\gamma_4=\frac{2}{\sqrt3}\frac{\F_2\F_4}{\bar\F_3\bar\F_4-\bar\F_4+1},\quad 
\gamma_5=\frac{2}{\sqrt3}\frac{\bar\F_2\F_4}{\F_4\F_2-\bar\F_3+1}.
$$
From the equality $\gamma_1=\gamma_2$ we get $\bar\F_2=\F_2^2\F_3$;
hence, inserting this in the equality $\gamma_4=\gamma_5$ we obtain $\bar\F_3=\frac{1}{\F_2}.$ Since $\bar\F_4=\frac{1}{\F_4},$ the system \eqref{bizning_kamonlar} is reduced to three equalities
$$
\frac{1}{\F_2\F_4 + 1} = \frac{\F_2\F_4}{\F_2\F_4 - \F_2 +1} = \frac{\F_2^2\F_4}{\F_2^2\F_4 + \F_2 - 1}.
$$
From the second equality (recalling that $\F_2>0$) we get $\F_2=1.$ Then the first equality in implies $\F_4=0,$ which contradicts the positivity of $\F_4.$ Thus, $\initialnetwork$ cannot be a self-shrinker.

\subsection*{Two half-lines: case 3.} 
Let $\initialnetwork$ be as in Figure \ref{fig:all_shrinkers} (d) 
so that with the notation of Figure \ref{fig:tarelka} the half-lines of $\initialnetwork$ start from $A_1$ and $A_4.$ Then $\theta_1=\bar\theta_1=\theta_4=\bar\theta_4=60^o$ 
so that  
$$
\F_1=\bar\F_1=\F_4=\bar\F_4=1,\quad \F_3\F_2=\bar\F_3\bar\F_2
$$
and 
$$
a_1=a_3=a,\quad \bar a_1=\bar a_3 = \frac{\F_2}{\bar\F_2}a,\quad 
a_2 = \frac{\F_2\F_3-\F_3+1}{\F_3}a,\quad 
\bar a_2 = \frac{\F_2\F_3-\bar \F_3+1}{\F_3}a.
$$
Now using the definitions of $c_i$ in \eqref{no_triple} and \eqref{one_triple} we find 
$$
\gamma_1=\gamma_6=\frac{2}{\sqrt3}\frac{\bar\F_2}{\F_2+\bar\F_2},\quad 
\gamma_2=\frac{2}{\sqrt3}\frac{\F_2\F_3}{\F_2\F_3-\F_3+1}
$$
and
$$
\gamma_3=\gamma_4=\frac{2}{\sqrt3}\frac{\F_2\bar\F_2\F_3}{\F_2+\bar\F_2},\quad 
\gamma_5=\frac{2}{\sqrt3}\frac{\bar\F_2\F_3}{\F_2\F_3-\bar\F_3+1}.
$$
From the equality $\gamma_1=\gamma_3$ we get $\F_2\F_3=\bar\F_2\bar\F_3=1.$ Thus, inserting the relations $F_3=\frac{1}{\F_2}$ and $\bar\F_3=\frac{1}{\bar\F_2}$ in \eqref{bizning_kamonlar} we deduce
$$
\frac{\bar\F_2}{\F_2+\bar\F_2} = \frac{\F_2}{2\F_2-1} = \frac{\bar\F_2^2}{2\bar\F_2\F_2-\F_2}.
$$
From the second equality it follows $\bar\F_2=\F_2$ or $\bar\F_2=\frac{\F_2}{2\F_2-1}.$ In the former case, the first equality reduces  to $\frac{1}{2\F_2}=\frac{1}{2\F_2-1},$ which is impossible. In the latter case, from the first equality we get the quadratic equation $2\F_2^2-2\F_2+1=0,$ which has no positive solutions. Therefore, $\initialnetwork$ is not a self-shrinker. 

\begin{remark}
In the Euclidean case 
there exists a unique convex self-shrinking lense-shaped network \cite{Sch:2011}. As we have seen Example \ref{ex:hexagon09o1},
 at the maximal time the hexagon $\bA$ shrinks to a point, but not self-similarly.
\end{remark}

\subsection*{Three half-lines: case 1.}
Let $\initialnetwork$ be as in Figure \ref{fig:all_shrinkers} (e) so that with the notation of Figure \ref{fig:tarelka} the half-lines of $\initialnetwork$ start from $A_1,A_2$ and $A_6.$ Then $\theta_1=\theta_2=\bar\theta_1=\bar\theta_2=\theta_4=\bar\theta_4=60^o$ 
so that $\F_2=\bar \F_2=1,$ and \eqref{dorime_umbaloiqd} is rewritten as 
\begin{equation*}
\F_4\bar \F_4=1,\quad \F_4\F_3 = \bar \F_3,\quad \bar \F_3\bar\F_4 = \F_3.
\end{equation*}
In this case, 
$$
a_1=\bar a_1= a,\quad 
a_2=  \frac{a}{\F_3}, \quad
a_3= \bar a_3 = \frac{\F_4\F_3-\F_4+1}{\F_4\F_3}a, \quad 
\bar a_2=  \frac{a}{\F_4\F_3},
$$
and hence, using the above discussions for the definition of $c_i$ in \eqref{no_triple} and \eqref{three_triple} we get 
$$
\gamma_1 = \gamma_2 = \gamma_6 = \gamma_5 = \frac{2}{\sqrt{3}}\, \frac{\F_4\F_3}{2\F_4\F_3+\F_4+1},\quad  \gamma_3=\frac{2}{\sqrt3}\,\frac{\F_4\F_3^2}{\F_4\F_3-\F_4+1},\quad 
\gamma_4=\frac{2}{\sqrt3}\,\frac{\F_4^2\F_3^2}{\F_4\F_3+\F_4-1}.
$$
Thus, \eqref{bizning_kamonlar} reduced to the following system:
\begin{equation}\label{dgzetrefs}
\frac{2}{\sqrt{3}}\, \frac{\F_4\F_3}{2\F_4\F_3+\F_4+1} = \frac{2}{\sqrt3}\,\frac{\F_4\F_3^2}{\F_4\F_3-\F_4+1} = \frac{2}{\sqrt3}\,\frac{\F_4^2\F_3^2}{\F_4\F_3+\F_4-1}.
\end{equation}
From the second equality it follows that $\F_4=1$ or $\F_4=\frac{1}{\F_3-1}.$ 
If $\F_4=1,$ inserting this in the second equality we find $\frac{1}{\F_3+1}=1,$ i.e., $\F_3=0,$ which contradicts the positivity of $\F_3.$ In case $\F_4=\frac{1}{\F_3-1},$ inserting this in the first equation in \eqref{dgzetrefs} we get $\frac{1}{\frac{2\F_3+1}{\F_3-1}+1} = \frac{\F_3}{2},$ which does not admit any positive solution. Thus, $\initialnetwork$ is not a self-shrinker.

\subsection{Three half-lines: case 2.} 
Let $\initialnetwork$ be as in Figure \ref{fig:all_shrinkers} (f) so that with the notation of Figure \ref{fig:tarelka} the half-lines of $\sS$ start from $A_1,A_2$ and $A_4.$ Then $\theta_1=\theta_2=\theta_4=\bar\theta_1=\bar\theta_4=60^o$ 
so that 
$$
\F_2=\F_4=\bar\F_4=1,\quad \F_3=\bar\F_3\bar\F_2.
$$
Whence 
$$
a_1=a_3=a,\quad a_2=\frac{a}{\bar\F_3\bar\F_2},\quad 
\bar a_1=\bar a_3 = \frac{a}{\bar \F_2},\quad \bar a_2 = \frac{\bar\F_2\bar\F_3 - \bar \F_3 + 1}{\bar \F_2\bar \F_3}a
$$
and hence, using the computations of $c_i$ in \eqref{no_triple}-\eqref{two_triple} above we deduce
$$
\gamma_1=\gamma_2=\gamma_6=\frac{2}{\sqrt3}\,\frac{\bar \F_3\bar\F_2}{\bar\F_3\bar \F_2 +\bar\F_3+1},\quad \gamma_3=\gamma_4=\frac{2}{\sqrt3}\frac{\bar\F_3\bar\F_2^2}{\bar\F_2+1},\quad \gamma_5 = \frac{2}{\sqrt3}\frac{\bar\F_3\bar\F_2^2}{\bar\F_3\bar \F_2 - \bar \F_3 +1}.
$$
Thus, \eqref{bizning_kamonlar} reduces to 
$$
\frac{2}{\sqrt3}\,\frac{\bar \F_3\bar\F_2}{\bar\F_3\bar \F_2 +\bar\F_3+1} = \frac{2}{\sqrt3}\frac{\bar\F_3\bar\F_2^2}{\bar\F_2+1} = \frac{2}{\sqrt3}\frac{\bar\F_3\bar\F_2^2}{\bar\F_3\bar \F_2 - \bar \F_3 +1}.
$$
Then from the first equality we have $\bar\F_3=\frac{1}{\bar\F_2(\bar\F_2+1)}$, and from the second equality $\bar\F_3=\frac{\bar\F_2}{\bar\F_2-1}.$ Therefore, $\bar\F_2^3 + \bar\F_2^2-\bar\F_2+1 =0,$ which has no positive roots. Thus, $\initialnetwork$ is not a self-shrinker.

\subsection*{Three half-lines: case 3.} 
Let $\initialnetwork$ be as in Figure \ref{fig:all_shrinkers} (g) so that with the notation of Figure \ref{fig:tarelka} the half-lines of $\initialnetwork$ start from $A_1,A_3$ and $A_5.$  Since the quadrangles $A_1OA_3A_2,$ $A_3OA_5A_4$ and $A_1OA_5A_6$ are rhombi with the same sidelength and one $60^o$ interior angle, $\theta_1 = \theta_3=\theta_4=\bar\theta_1=\bar\theta_2=\bar\theta_4=60^o$ and $\theta_2 = \bar\theta_2.$ 
Then 
$$
\F_1=\bar\F_1=\F_3=\bar\F_3=\F_4=\bar\F_4=1,\quad \F_2=\bar \F_2,
$$
and 
$$
a_1=a_3=\bar a_1 = \bar a_3 = a,\quad a_2=\bar a_2 = \F_2a.
$$
Thus, using \eqref{one_triple} in the definitions of $\gamma_i$ we find
$$
\gamma_1 = \gamma_6=\frac{1}{\sqrt3},\quad \gamma_2=\gamma_3=\gamma_4=\gamma_5=\frac{2}{\sqrt3}\frac{\F_2}{\F_2+1}.
$$
Therefore \eqref{bizning_kamonlar} reads as 
$$
\frac{1}{\sqrt3} = \frac{2}{\sqrt3}\frac{\F_2}{\F_2+1},
$$
which has a unique solution $\F_2=1.$ By the definition of $\F_2,$
$
\sin\theta_2 = \sin(60^o+\theta_2),
$
which has a unique (admissible) solution $\theta_2=60^o.$ Then $\bA$ is a homothetic Wulff shape of radius $a$ and the homothety center -- the origin of $\bA$ -- is located at the center. In this case, clearly, all $\gamma_i$ are equal to $1,$ so that $\initialnetwork$ is a self-shrinker. 

Let us seek the function $r(\cdot)$ satisfying $\sS(t)=r(t)\initialnetwork$ with $a_o:=\cH^1(S_1^0).$ Since $h(t)  = \frac{\sqrt3}{2}(a_o - a(t))$ and $\kappafi_{S_i(t)} = \frac{1}{\sqrt3a(t)},$ the equation $h'(t) = -\kappafi_{S_i(t)}$ is equivalent to 
$$
\frac{\sqrt3}{2} a' = \frac{1}{\sqrt3a}\quad\text{so that}\quad a(t) = \sqrt{a_o^2 - \frac{2}{3}t},\quad t\in[0,\tfrac{3a_o^2}{2}).
$$
Then $r(t) = \sqrt{1-\frac{2t}{3a_o^2}}$ satisfies $\sS(t)=r(t)\initialnetwork.$

\subsection*{Four half-lines: case 1.} 
Let $\initialnetwork$ be as in Figure \ref{fig:all_shrinkers} (h) so that with the notation of Figure \ref{fig:tarelka} the half-lines of $\initialnetwork$ start from $A_1,A_2,A_3$ and $A_4.$ Then $\theta_1=\theta_2=\theta_3=\theta_4=\bar\theta_1=\bar\theta_4=60^o$
so that 
$$
\F_1 = \bar\F_1 = \F_2=\F_3=\F_4=\bar\F_4 =1,\quad \bar\F_3 \bar \F_2 = 1.
$$
Then 
$$
a_1=a_2=a_3=a,\quad \bar a_1 = \bar a_3= \frac{a}{\bar \F_2} = \bar\F_3 a,\quad \bar a_2 = (2-\bar \F_3)a
$$
and recalling the definitions of $c_i$ in \eqref{no_triple} and \eqref{four_triple}
$$
\gamma_1=\gamma_2=\gamma_3=\gamma_4=\gamma_6=\frac{2}{\sqrt3}\frac{1}{3 + 2\bar \F_3},\quad \gamma_5=\frac{2}{\sqrt3}\frac{\bar \F_2}{2-\bar\F_3}.
$$
Therefore, \eqref{bizning_kamonlar} is equivalent to 
$$
\frac{2}{\sqrt3}\frac{1}{3 + 2\bar \F_3} =  \frac{2}{\sqrt3}\frac{\bar \F_2}{2-\bar\F_3}.
$$
This equation implies $\bar\F_3^2=-3,$ which is impossible. Thus, $\initialnetwork$ is not a self-shrinker.

\subsection*{Four half-lines: case 2.} 
Let $\initialnetwork$ be as in Figure \ref{fig:all_shrinkers} (i) so that with the notation of Figure \ref{fig:tarelka} (d) the half-lines of $\initialnetwork$ start from $A_1,A_2,A_6$ and $A_4.$ Then $\theta_1=\theta_2=\bar\theta_1=\bar\theta_2=\theta_4=\bar\theta_4=60^o$ and $\theta_3=\bar\theta_3$ so that 
$$
\F_1=\bar\F_1=\F_2 = \bar\F_2 = \F_4=\bar\F_4 = 1,\quad \bar\F_3 = \F_3.
$$
Then 
$$
a_1=\bar a_1=a_3=\bar a_3=a,\quad a_2=\bar a_2=\frac{a}{\F_3}.
$$
Thus, by the definition of $c_i$ in \eqref{no_triple} and \eqref{three_triple}
$$
\gamma_1=\gamma_6 = \gamma_2=\gamma_5 = \frac{2}{\sqrt3}\frac{\F_3}{2+2\F_3},\quad 
\gamma_3=\gamma_4 = \frac{\F_3}{\sqrt3}.
$$
Therefore, by \eqref{bizning_kamonlar}
$
\frac{2}{\sqrt3}\frac{\F_3}{2+2\F_3} = \frac{\F_3}{\sqrt3}
$
which implies $\F_3=0,$ a contradiction. Thus, $\initialnetwork$ is not a self-shrinker.

\subsection*{Four half-lines: case 3.} 
Let $\initialnetwork$ be as in Figure \ref{fig:all_shrinkers} (j). Then $\bA$ is a homothetic Wulff shape and the homothety center is located at the center of $\bA.$ With the notation of Figure \ref{fig:tarelka} (d), the half-lines of $\sS$ start from $A_1,A_6,A_3$ and $A_4.$ Then $\theta_i=\bar\theta_i=60^o,$ $a_i=\bar a_i=a$ and by \eqref{three_triple} $c_i=\frac{2}{3\sqrt3}$ for all possible $i$. Hence, all $\gamma_i$ equal to $\frac{1}{\sqrt3}$ and $\sS^0$ 
is a self-shrinker. 

Let us define $r(\cdot)$ satisfying $\sS(t)=r(t)\initialnetwork$ with $a_o:=\cH^1(S_1^0).$ Since $h(t) = \frac{\sqrt3}{2}(a_o-a(t))$ and $\kappafi_{S_i(t)} = -\frac{2}{3\sqrt3a(t)},$ the equation $h'(t) = -\kappafi_{S_i(t)}$ is equivalent to 
$$
\frac{\sqrt3}{2}a' = -\frac{2}{3\sqrt3 a}\quad\text{so that}\quad a(t) = \sqrt{a_o^2 - \frac{4}{9}t},\quad t\in[0,\tfrac{9a_o^2}{4}).
$$
Then $r(t):=\sqrt{1- \frac{4t}{9a_o^2}}$ satisfies $\sS(t) = r(t)\initialnetwork.$

\subsection*{Five half-lines} 
Let $\initialnetwork$ be as in Figure \ref{fig:all_shrinkers} (k). Then $\bA$ is a homothetic Wulff shape and the homothety center is located at the center of $\bA.$ Let $\sS(t)=r(t)\initialnetwork$ for some $r(\cdot)$ to be defined later and let $R(t)$ be the sidelength of $\bA.$ In this case all heights of segments of $\sS(t)$ from $\initialnetwork$ are equal to $h(t):=\frac{\sqrt3}{2}(R_0 - R(t))$, and by \eqref{five_triple} the $\phi$-curvatures of all segments are equal to $-\frac{2}{6\sqrt3 R(t)}.$ Thus, $\phi$-curvature flow equation is represented as 
$$
-\frac{\sqrt3}{2}R'(t) = \frac{1}{3\sqrt3R(t)},\quad \text{hence,}\quad R(t) = \sqrt{R_0^2-\frac{4}{9}t},\quad t\in [0,\tfrac{9}{4R_0^2}).
$$
Hence, the function $r(t):=\sqrt{1 - \frac{4t}{9R_0^2}}$ satisfies $\sS(t)=r(t)\initialnetwork.$

\begin{wrapfigure}[9]{l}{0.4\textwidth}
\vspace*{-2mm}
\includegraphics[width=0.38\textwidth]{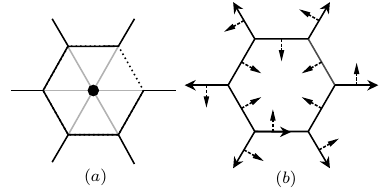}
\caption{\small}
\label{fig:pozdravlyayu}
\end{wrapfigure}
\subsection*{Six half-lines} 
Let $\initialnetwork$ be as in Figure \ref{fig:all_shrinkers} (l). Then $\bA$ is a homothetic Wulff shape and the homothety center is located at the center of $\bA.$ $\initialnetwork$ admits a locally constant CH field; one example of such a field is drawn in Figure \ref{fig:pozdravlyayu} (b). Thus, $\initialnetwork$ is critical, and therefore, it 
stays still. Notice that $\initialnetwork$ is not a local minimizer of $\ell_\phi,$ since removing one facet of the Wulff shape decreases the length of $\sS$ in every disc $D_R$ compactly containing $\bA.$

\subsection{Homothety center on $\partial \bA$}

In this short section we assume that more than one half-line of a self-shrinker $\initialnetwork$ start from the same vertex of $\bA$ or a half-line is collinear with a segment of $\bA$ having a common vertex. In this situation the homothety center is necessarily located at this vertex. In particular, the segments of $\bA$ ending at this vertex should have zero $\phi$-curvature. For instance, in Figure \ref{fig:all_shrinkers} (a) two half-lines of $\initialnetwork$ start from the same vertex $A_1$ (coinciding with the origin) of the hexagon $\bA$ of sidelength $a_o>0$ and one more half-line bisects the angle at $A_4$; let $\sS(t):=r(t)\initialnetwork$ be the family of homotheties of $\initialnetwork,$ with $r(t)$ to be defined.

\begin{figure}[htp!]
\includegraphics[width=0.8\textwidth]{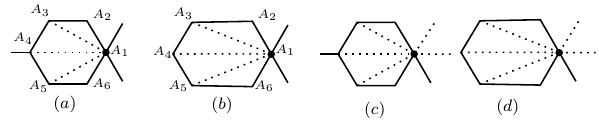}
\caption{\small}
\label{fig:masheynik09}
\end{figure}

\noindent
Repeating the same arguments of Section \ref{subsec:homothety_center_inside_the_hexagon}, we can show that:

\begin{itemize}
\item for the network $\initialnetwork$ in Figure \ref{fig:masheynik09} (a), $\bA$ is homothetic to the Wulff shape. Let $a(t)$ be the sidelength of $\bA(t)$ in $\sS(t)=\cup_iS_i(t).$ Assuming that $\p \bA(t)$ is oriented counterclockwise, so that its unit normal vector field points inside $\bA(t)$, we compute 
$$
\kappafi_{S_1(t)} = \kappafi_{S_6(t)} = 0, \quad 
\kappafi_{S_2(t)} = \kappafi_{S_5(t)} =- \frac{2}{\sqrt3a(t)},\quad 
\kappafi_{S_3(t)} = \kappafi_{S_4(t)} =- \frac{4}{\sqrt3a(t)},
$$ 
where $a(t)>0$ is the length of $\bA(t).$ Note that $A_1A_3$ is orthogonal to $A_3A_4$ so that $\cH^1([A_1A_3]) = \sqrt3a(t)$, and $[A_1Q]$ is orthogonal 
to the straight line containing $[A_2A_3]$, and therefore $\cH^1([A_1Q]) = \frac{\sqrt3a(t)}{2}.$ Now consider the heights between the corresponding segments of $\p \bA$ and $\p\bA(t)$: we have 
$$
h_1(t)=h_6(t)=0,\quad h_2(t) = h_5(t) = \frac{\sqrt3}{2}(a_o - a(t)),\quad h_3(t)=h_4(t) = \sqrt3(a_o-a(t)).
$$
Then the $\phi$-curvature equation $h_i'(t) = -\kappafi_{S_i(t)}$ is equivalent to 
$
-a'(t) = \frac{4}{3a(t)},
$
which admits the unique solution 
$$
a(t) = \sqrt{a_o^2 - \frac{8}{3}t},\quad t\in[0,\tfrac{3a_o^2}{8}).
$$
Thus, $\sS(\cdot)$ is the homothetically shrinking solution starting from $\initialnetwork,$ with $r(t) = \sqrt{1- \frac{8t}{3a_o^2}}.$

\item One checks that the network $\initialnetwork$ in Figure \ref{fig:masheynik09} (b) is not a self-shrinker. The same holds for the network in Figure \ref{fig:masheynik09} (d).

\item The network in Figure \ref{fig:masheynik09} (c) is obtained from (a) adding one or two dotted half-lines. Clearly, this does not affect to the self-similarity.
\end{itemize}

Finally, the networks in Figure \ref{fig:masheynik09} (a) and (c) are examples of simple self-shrinkers with multiple junctions.

\end{document}